\documentclass[a4paper,11pt]{amsart}

\usepackage{a4wide}
\usepackage{amssymb}
\usepackage[foot]{amsaddr}

\newtheorem{theorem}{Theorem}[section]

\newtheorem{lemma}[theorem]{Lemma}

\newtheorem{proposition}[theorem]{\revision{P}roposition}

\theoremstyle{definition}
\newtheorem{definition}[theorem]{Definition}

\theoremstyle{remark}
\newtheorem{remark}[theorem]{Remark}

\numberwithin{equation}{section}

\usepackage{mathrsfs}
\usepackage{stmaryrd}
\usepackage{amssymb}

\usepackage{tikz}

\usepackage{pgfplots}

\newcommand{\SlopeTriangle}[6]
{

    \pgfplotsextra
    {
        \pgfkeysgetvalue{/pgfplots/xmin}{\xmin}
        \pgfkeysgetvalue{/pgfplots/xmax}{\xmax}
        \pgfkeysgetvalue{/pgfplots/ymin}{\ymin}
        \pgfkeysgetvalue{/pgfplots/ymax}{\ymax}

        \pgfmathsetmacro{\xArel}{#1}
        \pgfmathsetmacro{\yArel}{#3}
        \pgfmathsetmacro{\xBrel}{#1-#2}
        \pgfmathsetmacro{\yBrel}{\yArel}
        \pgfmathsetmacro{\xCrel}{\xArel}

        \pgfmathsetmacro{\lnxB}{\xmin*(1-(#1-#2))+\xmax*(#1-#2)} 
        \pgfmathsetmacro{\lnxA}{\xmin*(1-#1)+\xmax*#1} 
        \pgfmathsetmacro{\lnyA}{\ymin*(1-#3)+\ymax*#3} 
        \pgfmathsetmacro{\lnyC}{\lnyA+#4*(\lnxA-\lnxB)}
        \pgfmathsetmacro{\yCrel}{\lnyC-\ymin)/(\ymax-\ymin)} 

        \coordinate (A) at (rel axis cs:\xArel,\yArel);
        \coordinate (B) at (rel axis cs:\xBrel,\yBrel);
        \coordinate (C) at (rel axis cs:\xCrel,\yCrel);

        \draw[#6]   (A)-- node[anchor=north] {#5}
                    (B)--
                    (C)--
                    cycle;
    }
}

\newcommand{\ee}{\boldsymbol A}
\newcommand{\hh}{\boldsymbol h}
\newcommand{\bzero}{\boldsymbol 0}

\newcommand{\jj}{\boldsymbol j}

\newcommand{\sig}{\boldsymbol \sigma}
\newcommand{\ch}{\boldsymbol \chi}

\newcommand{\vv}{\boldsymbol v}
\newcommand{\tvv}{\widetilde{\vv}}

\newcommand{\ww}{\boldsymbol w}

\newcommand{\bvf}{\boldsymbol \varphi}

\newcommand{\VVh}{\VV_{{\!h}}}

\newcommand{\rr}{\boldsymbol r}
\newcommand{\trr}{\widetilde \rr}

\newcommand{\xxi}{\boldsymbol \xi}
\newcommand{\zzeta}{\boldsymbol \zeta}

\newcommand{\xx}{\boldsymbol x}
\newcommand{\yy}{\boldsymbol y}

\newcommand{\nn}{\boldsymbol n}
\newcommand{\ttau}{{\boldsymbol \tau}}

\newcommand{\uu}{\boldsymbol u}

\newcommand{\LL}{\boldsymbol L}
\newcommand{\HH}{\boldsymbol H}

\newcommand{\tU}{\widetilde U}
\newcommand{\tuu}{\widetilde \uu}
\newcommand{\tG}{\widetilde \Gamma}
\newcommand{\tq}{\widetilde q}

\newcommand{\ccurl}{\boldsymbol{\operatorname{curl}}}
\newcommand{\scurl}{\operatorname{curl}}

\newcommand{\ddiv}{\operatorname{div}}

\newcommand{\supp}{\operatorname{supp}}

\newcommand{\conv}{\operatorname{conv}}

\newcommand{\grad}{\boldsymbol \nabla}
\renewcommand{\div}{\grad \cdot}
\newcommand{\curl}{\grad \times}

\newcommand{\VV}{\boldsymbol V}

\newcommand{\TT}{\mathcal T}
\newcommand{\EE}{\mathcal E}
\newcommand{\FF}{\mathcal F}

\newcommand{\EED}{\EE^{\rm D}}
\newcommand{\EEN}{\EE^{\rm N}}

\newcommand{\FFe}{{\FF^\edge}}
\newcommand{\TTe}{{\TT^\edge}}

\newcommand{\FFei}{{\FF_{\rm int}^\edge}}
\newcommand{\FFee}{{\FF_{\rm ext}^\edge}}

\newcommand{\RR}{\boldsymbol{\mathcal R}}

\newcommand{\PP}{\boldsymbol{\mathcal P}}
\newcommand{\NN}{\boldsymbol{\mathcal N}}
\newcommand{\RT}{\boldsymbol{\mathcal{RT}}}

\newcommand{\ppsi}{\boldsymbol \psi}
\newcommand{\pphi}{\boldsymbol \phi}
\newcommand{\oome}{\boldsymbol \omega}

\newcommand{\tpphi}{\widetilde{\pphi}}

\newcommand{\Clift}{C_{\rm L}}

\newcommand{\CPe}{C_{{\rm P},\edge}}

\newcommand{\Cke}{C_{\kappa_\edge}}
\newcommand{\Cconte}{C_{{\rm cont},\edge}}
\newcommand{\Ccont}{C_{\rm cont}}

\newcommand{\CPVe}{C_{{\rm \revision{PFW}},\edge}}

\newcommand{\Cste}{C_{{\rm st},\edge}}

\newcommand{\edge}{e}
\newcommand{\aaa}{\boldsymbol a}

\newcommand{\ome}{{\omega_\edge}}

\newcommand{\ppi}{\boldsymbol \pi}

\newcommand{\osc}{\operatorname{osc}}

\newcommand{\GD}{{\Gamma_{\rm D}}}
\newcommand{\GN}{{\Gamma_{\rm N}}}
\newcommand{\GeD}{{\Gamma_{\rm D}^\edge}}
\newcommand{\GeN}{{\Gamma_{\rm N}^\edge}}
\newcommand{\gD}{{\gamma_{\rm D}}}
\newcommand{\gN}{{\gamma_{\rm N}}}

\newcommand{\LH}{\boldsymbol {\mathcal H}(\Omega,\GD)}
\newcommand{\KK}{\boldsymbol K}
\newcommand{\ttheta}{\boldsymbol \theta}
\newcommand{\vvphi}{\boldsymbol \varphi}

\newcommand{\TTT}{\boldsymbol T}
\newcommand{\JJJ}{\mathbb J}

\newcommand{\adown}{\aaa_{\rm d}}
\newcommand{\aup}{\aaa_{\rm u}}
\newcommand{\Fdown}{F^{\rm d}}
\newcommand{\Fup}{F^{\rm u}}

\newcommand{\Kin} {K_{\rm in}}
\newcommand{\Kout}{K_{\rm out}}

\newcommand{\jump}[1]{\llbracket #1 \rrbracket}

\newcommand{\argmin}[1]{\underset{\substack{#1}}{\operatorname{argmin}}}
\newcommand{\eq}{:=}

\newcommand{\sign}{\operatorname{sign}}

\newcommand{\pp}{{p'}}
\newcommand\ie{i.e.}
\newcommand\cf{cf.}
\newcommand\eg{e.g.}

\newcommand{\bse}{\begin{subequations}}
\newcommand{\ese}{\end{subequations}}

\usepackage{a4wide}

\usepackage[breaklinks,bookmarks=false]{hyperref}

\hypersetup{colorlinks, linkcolor=blue, citecolor=blue,
urlcolor=blue, plainpages=false, pdfwindowui=false,
pdfstartview={FitH}}

\newcommand{\revision}[1]{#1}

\usepackage[breaklinks,bookmarks=false]{hyperref}

\hypersetup{colorlinks, linkcolor=blue, citecolor=blue,
urlcolor=blue, plainpages=false, pdfwindowui=false,
pdfstartview={FitH}}

\begin{document}

\title%
[Stable broken {${\boldsymbol H}(\ccurl)$} polynomial extensions \&
$\MakeLowercase{p}$-robust \revision{broken} equilibration]%
{Stable broken {${\boldsymbol H}(\MakeLowercase{\ccurl})$} polynomial extensions\\
and $\MakeLowercase{p}$-robust a posteriori error estimates
\revision{by broken patchwise equilibration} for \revision{the curl--curl problem}$^\star$}

\author{T. Chaumont-Frelet$^{1,2}$}
\author{A. Ern$^{3,4}$}
\author{M. Vohral\'ik$^{4,3}$}
\address{\vspace{-.5cm}}
\address{\noindent \tiny \textup{$^\star$This project has received funding from the European Research Council (ERC) under the European Union’s Horizon 2020}}
\address{\noindent \tiny \textup{\hspace{.2cm}research and innovation program (grant agreement No 647134 GATIPOR).}}
\address{\noindent \tiny \textup{$^1$Inria, 2004 Route des Lucioles, 06902 Valbonne, France}}
\address{\noindent \tiny \textup{$^2$Laboratoire J.A. Dieudonn\'e, Parc Valrose, 28 Avenue Valrose, 06108 Nice Cedex 02, 06000 Nice, France}}
\address{\noindent \tiny \textup{$^3$Universit\'e Paris-Est, CERMICS (ENPC), 6 et 8 av. Blaise Pascal 77455 Marne la Vall\'ee cedex 2, France}}
\address{\noindent \tiny \textup{$^4$Inria, 2 rue Simone Iff, 75589 Paris, France}}

\date{}

\begin{abstract}
We study extensions of piecewise polynomial data prescribed in a patch of tetrahedra sharing
an edge. We show stability in the sense that the minimizers over piecewise polynomial spaces
with prescribed tangential component jumps across faces and prescribed piecewise curl in
elements are subordinate in the broken energy norm to the minimizers over the \revision{broken}
$\HH(\ccurl)$ space with the same prescriptions. Our proofs are constructive and
yield constants independent of the polynomial degree. We then detail the application of this
result to the a posteriori error analysis of \revision{the curl--curl problem} discretized with
N\'ed\'elec finite elements of arbitrary order. The resulting estimators are \revision{reliable,}
locally efficient, polynomial-degree-robust, and inexpensive. \revision{They are constructed by a
\revision{broken patchwise} equilibration which, in particular, does not produce a globally
$\HH(\ccurl)$-conforming flux. The equilibration is only related to} edge patches
and can be realized without solutions of patch problems by a sweep through tetrahedra
around every mesh edge. The error estimates become guaranteed when the regularity pick-up
constant is explicitly known. Numerical experiments illustrate the theoretical findings.

\noindent
{\sc Key Words.} A posteriori error estimates; Finite element methods;
\revision{Electromagnetics}; High order methods.

\noindent
{\sc AMS subject classification.} Primary 65N30, 78M10, 65N15.
\end{abstract}

\maketitle

\section{Introduction}
\label{sec_int}

The so-called N\'ed\'elec or also edge element spaces of~\cite{Ned_mix_R_3_80} form,
on meshes consisting of tetrahedra, the most natural piecewise polynomial subspace of
the space $\HH(\ccurl)$ composed of square-integrable fields with square-integrable weak curl.
They are instrumental in numerous applications in link with electromagnetism,
see for example~\cite{Ass_Ciar_Lab_foundat_electr_18,Boss_elctr_98,Hiptmair_acta_numer_2002,Monk_FEs_Maxwell_03}.
The goal of this paper is to study two different but connected questions related to these spaces.

\subsection{Stable broken {\em H}$(\ccurl)$ polynomial extensions}

Polynomial extension operators are an essential tool in numerical analysis involving
N\'ed\'elec spaces, in particular in the case of high-order discretizations.
Let $K$ be a tetrahedron. Then, given a boundary datum in the form of a suitable polynomial
on each face of $K$, satisfying some compatibility conditions,
a {\em polynomial extension} operator constructs a curl-free polynomial in the interior of
the tetrahedron $K$ whose tangential trace fits the {\em boundary datum} and which is stable
with respect to the datum in the intrinsic norm. Such an operator was derived
in~\cite{Demk_Gop_Sch_ext_II_09}, as a part of equivalent developments in the
$H^1$ and $\HH(\ddiv)$ spaces respectively in~\cite{Demk_Gop_Sch_ext_I_09}
and~\cite{Demk_Gop_Sch_ext_III_12}, see also~\cite{MunSol_pol_lift_97} and the references therein.
An important achievement extending in a similar stable way a given polynomial {\em volume datum}
to a polynomial with curl given by this datum in a single simplex, along with a similar result
in the $H^1$ and $\HH(\ddiv)$ settings, was presented in~\cite{Cost_McInt_Bog_Poinc_10}.

The above results were then combined together and extended
from a single simplex to a patch of simplices sharing the given vertex in several cases:
in $\HH(\ddiv)$ in two space dimensions in~\cite{Brae_Pill_Sch_p_rob_09} and in $H^1$ and
$\HH(\ddiv)$ in three space dimensions in~\cite{Ern_Voh_p_rob_3D_20}. These results have
important applications to a posteriori analysis but also
to localization and optimal $hp$ estimates in a priori analysis,
see~\cite{Ern_Gud_Sme_Voh_loc_glob_div_21}. To the best of our knowledge,
a similar patchwise result in the $\HH(\ccurl)$ setting is not available yet,
and it is our goal to establish it here. We achieve it in our first main result,
Theorem~\ref{theorem_stability}, see also the equivalent form in
Proposition~\ref{prop_stability_patch} \revision{and the construction in Theorem~\ref{thm_sweep}}.

Let $\TTe$ be a patch of tetrahedra sharing a given edge $\edge$ from a shape-regular mesh
$\TT_h$ and let $\ome$ be the corresponding patch subdomain. Let $p\ge0$ be a polynomial
degree. \revision{Let $\jj_p \in \RT_p(\TTe) \cap \HH(\ddiv,\ome)$ with $\div \jj_p = 0$ be
a divergence-free Raviart--Thomas field, and let $\ch_p$ be in the broken N\'ed\'elec space
$\NN_p(\TTe)$.} In this work, we establish that
\begin{equation} \label{eq_BPE}
\min_{\substack{\vv_p \in \NN_p(\TTe) \cap \HH(\ccurl,\ome) \\ \curl \vv_p = \jj_p}}
\|\ch_p - \vv_p\|_\ome
\leq
C
\min_{\substack{\vv \in \HH(\ccurl,\ome) \\ \curl \vv = \jj_p}}
\|\ch_p - \vv\|_\ome,
\end{equation}
which means that the {\em discrete} constrained {\em best-approximation error} in the patch
is subordinate to the {\em continuous} constrained best-approximation error up to a
constant $C$. Importantly, $C$ only depends on the shape-regularity of the edge patch
and does {\em not depend} on the {\em polynomial degree $p$} under consideration.
Our proofs are constructive, which has a particular application in a posteriori error
analysis, as we discuss now.

\subsection{$p$-robust a posteriori error estimates \revision{by broken patchwise equilibration}
for \revision{the curl--curl problem}}
\label{sec_a_post_intr}

Let $\Omega \subset \mathbb R^3$ be a Lipschitz polyhedral domain with unit outward normal $\nn$.
Let $\GD,\GN$ be two disjoint, open, possibly empty subsets
of $\partial \Omega$ such that $\partial \Omega = \overline \GD \cup \overline \GN$.
Given a divergence-free field $\jj: \Omega \to \mathbb R^3$ \revision{with zero normal trace on $\GN$},
\revision{the curl--curl problem} amounts to seeking a field $\ee: \Omega \to \mathbb R^3$
satisfying
\begin{subequations}
\label{eq_maxwell_strong}
\begin{alignat}{2}
\label{eq_maxwell_strong_volume}
&\curl \curl \ee = \jj, \quad \div \ee = 0,&\qquad&\text{in $\Omega$}, \\
&\ee \times \nn = \bzero,&\qquad&\text{on $\GD$},\\
&(\curl \ee) \times \nn = \bzero,\quad \ee \cdot \nn = 0,&\qquad&\text{on $\GN$}.
\end{alignat}
Note that $\ee \times \nn = 0$ implies that $(\curl \ee) \cdot \nn=0$ on $\GD$.
\revision{When $\Omega$ is not simply connected and/or when $\GD$ is not connected,
the additional conditions
\begin{equation}
\label{eq_maxwell_cohomology}
(\ee,\ttheta)_\Omega = 0,
\qquad
(\jj,\ttheta)_\Omega = 0,
\qquad \forall \ttheta \in \LH
\end{equation}
must be added in order to ensure existence and uniqueness of a solution
to~\eqref{eq_maxwell_strong}, where $\LH$ is the finite-dimensional ``cohomology''
space associated with $\Omega$ and the partition of its boundary (see Section \ref{sec_notat}).}
\end{subequations}
\revision{The boundary-value problem \eqref{eq_maxwell_strong} appears immediately
in this form in magnetostatics. In this case, $\jj$ and $\ee$ respectively represent a (known)
current density and the (unknown) associated magnetic vector potential, while the key quantity
of interest is the magnetic field $\hh \eq \curl \ee$.
We refer the reader to \cite{Ass_Ciar_Lab_foundat_electr_18,Boss_elctr_98,Hiptmair_acta_numer_2002,Monk_FEs_Maxwell_03}
for reviews of models considered in computational electromagnetism.}

In the rest of the introduction, we assume for simplicity that $\GD=\partial\Omega$
(so that the boundary conditions reduce to $\ee \times \nn = \bzero$ on $\partial\Omega$)
and that $\jj$ is a piecewise polynomial in the Raviart--Thomas space,
$\jj \in \RT_p(\TT_h) \cap \HH(\ddiv,\Omega)$, $p \geq 0$. Let
$\ee_h \in \NN_p(\TT_h) \cap \HH_0(\ccurl,\Omega)$ be a numerical approximation to $\ee$
in the N\'ed\'elec space. Then, the Prager--Synge equality~\cite{Prag_Syng_47}, \cf, \eg,
\cite[equation~(3.4)]{Rep_a_post_Maxw_07} or~\cite[Theorem~10]{Braess_Scho_a_post_edge_08},
implies that
\begin{equation}
\label{eq_PS}
\|\curl(\ee - \ee_h)\|_\Omega \leq \min_{\substack{
\hh_h \in \NN_p(\TT_h) \cap \HH(\ccurl,\Omega)
\\
\curl \hh_h = \jj}}
\|\hh_h - \curl \ee_h\|_\Omega.
\end{equation}
Bounds such as~\eqref{eq_PS} have been used in, \eg,
\cite{%
Creus_Men_Nic_Pir_Tit_guar_har_19,%
Creus_Nic_Tit_guar_Maxw_17,%
Han_a_post_Maxw_08,%
Neit_Rep_a_post_Maxw_10},
see also the references therein.

The estimate~\eqref{eq_PS} leads to a guaranteed and sharp upper bound. Unfortunately,
as written, it involves a global minimization over $\NN_p(\TT_h) \cap \HH(\ccurl,\Omega)$,
and is consequently too expensive in practical computations.
Of course, a further upper bound follows from~\eqref{eq_PS} for {\em any}
$\hh_h \in \NN_p(\TT_h) \cap \HH(\ccurl,\Omega)$
such that $\curl \hh_h = \jj$. At this stage, though, it is not clear how to find an
{\em inexpensive local} way of constructing a suitable field $\hh_h$, called an
{\em equilibrated flux}. A proposition for the lowest degree $p=0$ was
given in~\cite{Braess_Scho_a_post_edge_08}, but suggestions
for higher-order cases were not available until very recently
in~\cite{Ged_Gee_Per_a_post_Maxw_19,Licht_FEEC_a_post_H_curl_19}.
In particular, the authors in~\cite{Ged_Gee_Per_a_post_Maxw_19} also prove efficiency,
\ie, they devise a field $\hh_h^* \in \NN_p(\TT_h) \cap \HH(\ccurl,\Omega)$ such that,
up to a generic constant $C$ independent of the mesh size $h$ but possibly
depending on the polynomial degree $p$,
\begin{equation} \label{eq_eff}
    \|\hh_h^* - \curl \ee_h\|_{\Omega} \leq C \|\curl(\ee - \ee_h)\|_{\Omega},
\end{equation}
as well as a local version of~\eqref{eq_eff}.
Numerical experiments in~\cite{Ged_Gee_Per_a_post_Maxw_19} reveal very good effectivity indices,
also for high polynomial degrees $p$.

A number of a posteriori error estimates that are {\em polynomial-degree robust}, \ie,
where no generic constant depends on $p$, were obtained recently. For equilibrations
(reconstructions) in the $\HH(\ddiv)$ setting in two space dimensions, they were first
obtained in~\cite{Brae_Pill_Sch_p_rob_09}. Later, they were extended to the $H^1$ setting
in two space dimensions in~\cite{Ern_Voh_p_rob_15} and to both $H^1$ and $\HH(\ddiv)$ settings
in three space dimensions in~\cite{Ern_Voh_p_rob_3D_20}. Applications to problems with arbitrarily
jumping diffusion coefficients, second-order eigenvalue problems, the Stokes problem, linear
elasticity, or the heat equation are reviewed in~\cite{Ern_Voh_p_rob_3D_20}.
In the $\HH(\ccurl)$ setting, with application to
\revision{the curl--curl problem}~\eqref{eq_maxwell_strong}, however, to the best of
our knowledge, such a result was missing\footnote{We have learned very recently that a
modification of~\cite{Ged_Gee_Per_a_post_Maxw_19} can lead to a polynomial-degree-robust
error estimate, see~\cite{Ged_Gee_Per_Sch_post_Maxw_20}.}. It is our goal to establish it here,
and we do so in our second main result, Theorem~\ref{theorem_aposteriori}.

Our upper bound in Theorem~\ref{theorem_aposteriori} actually does {\em not derive}
from the Prager--Synge equality to take the form~\eqref{eq_PS}, since we do not construct
an equilibrated flux $\hh_h^* \in \NN_p(\TT_h) \cap \HH(\ccurl,\Omega)$. We instead perform a
{\em \revision{broken patchwise} equilibration} producing locally on each edge patch $\TTe$
a piecewise polynomial $\hh_h^{\edge} \in \NN_p(\TTe) \cap \HH(\ccurl,\ome)$ such that
$\curl \hh_h^{\edge} = \jj$.  Consequently, our error estimate rather takes the form
\begin{equation} \label{eq_up_intr}
\|\curl(\ee - \ee_h)\|_\Omega
\leq
\sqrt{6} \Clift \Ccont \left (\sum_{\edge \in \EE_h}
\|\hh_h^{\edge} - \curl \ee_h\|_{\ome}^2 \right )^{1/2}.
\end{equation}
We obtain each local contribution $\hh_h^{\edge}$ in a single-stage procedure,
in contrast to the three-stage procedure of~\cite{Ged_Gee_Per_a_post_Maxw_19}.
Our \revision{broken patchwise} equilibration is also rather inexpensive,
since the edge patches are smaller than the usual vertex patches employed 
in~\cite{Braess_Scho_a_post_edge_08, Ged_Gee_Per_a_post_Maxw_19}.
Moreover, we can either solve the patch
problems, see~\eqref{eq_definition_estimator_2}, or replace them by a {\em sequential sweep}
through tetrahedra sharing the given edge $e$, see~\eqref{eq_definition_estimator_sweep_2}.
This second option yields \revision{a cheaper procedure where merely elementwise, in place of
patchwise, problems are to be solved and even delivers} a {\em fully explicit} a posteriori
error estimate in the {\em lowest-order} setting $p=0$. The price we pay for these
advantages is the emergence of the constant $\sqrt{6} \Clift \Ccont$ in our upper
bound~\eqref{eq_up_intr}; here $\Ccont$ is fully computable, only depends on the mesh
shape-regularity, and takes values around 10 for usual meshes, whereas $\Clift$ only depends on
the shape of the domain $\Omega$ \revision{and boundaries $\GD$ and $\GN$}, with in particular
$\Clift = 1$ whenever $\Omega$ is convex. Crucially, our error estimates are
{\em locally efficient} and polynomial-degree robust in that
\begin{equation} \label{eq_low_intr}
    \|\hh_h^{\edge} - \curl \ee_h\|_{\ome} \leq C \|\curl(\ee - \ee_h)\|_{\ome}
\end{equation}
for all edges $\edge$, where the constant $C$ only depends on the shape-regularity of the
mesh, as an immediate application of our first main result
in Theorem~\ref{theorem_stability}. It is worth noting that the lower bound~\eqref{eq_low_intr}
is completely local to the edge patches $\ome$ and does not comprise any neighborhood.

\subsection{Organization of this contribution}

The rest of this contribution is organised as follows.
In Section~\ref{sec_not}, we recall the functional spaces,
state a weak formulation of problem~\eqref{eq_maxwell_strong},
describe the finite-dimensional Lagrange, N\'ed\'elec, and Raviart--Thomas spaces,
and introduce the numerical discretization of~\eqref{eq_maxwell_strong}.
Our two main results, Theorem~\ref{theorem_stability}
\revision{(together with its sequential form in Theorem~\ref{thm_sweep})}
and Theorem~\ref{theorem_aposteriori}, are formulated and discussed in Section~\ref{sec_main_res}.
Section~\ref{sec_num} presents a numerical illustration of our a posteriori error estimates for
\revision{curl--curl problem}~\eqref{eq_maxwell_strong}. Sections~\ref{sec_proof_a_post}
and~\ref{sec_proof_stability} are then dedicated to the proofs of our two main results.
\revision{Finally, Appendix~\ref{appendix_weber} establishes an auxiliary result of independent
interest: a Poincar\'e-like inequality using the curl of divergence-free fields in an edge patch.}

\section{\revision{Curl--curl problem} and N\'ed\'elec finite element discretization}
\label{sec_not}

\subsection{Basic notation}
\label{sec_notat}

Consider a Lipschitz polyhedral subdomain $\omega \subseteq \Omega$. We denote by
$H^1(\omega)$ the space of scalar-valued $L^2(\omega)$ functions with
$\LL^2(\omega)$ weak gradient, $\HH(\ccurl,\omega)$ the space of vector-valued $\LL^2(\omega)$
fields with $\LL^2(\omega)$ weak curl, and $\HH(\ddiv,\omega)$ the space of vector-valued
$\LL^2(\omega)$ fields with $L^2(\omega)$ weak divergence. Below, we use the notation
$({\cdot},{\cdot})_\omega$ for the $L^2(\omega)$ or $\LL^2(\omega)$ scalar product and
$\|{\cdot}\|_\omega$ for the associated norm.
$L^\infty(\omega)$ and $\LL^\infty(\omega)$ are the spaces of essentially bounded
functions with norm $\|{\cdot}\|_{\infty,\omega}$.

Let $\HH^1(\omega) \eq \{\vv \in \LL^2(\omega)| \, v_i \in H^1(\omega), \, i=1, 2, 3\}$.
Let $\gD$, $\gN$ be two disjoint, open, possibly empty subsets of $\partial \omega$ such that
$\partial \omega = \overline \gD \cup \overline \gN$.
Then $H^1_\gD(\omega) \eq \{v \in H^1(\omega)| \, v=0$ on $\gD\}$
is the subspace of $H^1(\omega)$ formed by functions vanishing on $\gD$ in the sense of traces.
Furthermore, $\HH_\gD(\ccurl,\omega)$ is the
subspace of $\HH(\ccurl,\omega)$ composed of fields with vanishing tangential trace on
$\gD$, $\HH_\gD(\ccurl,\omega) \eq \{\vv \in \HH(\ccurl,\omega)$ such that
$(\curl \vv, \bvf)_\omega - (\vv, \curl \bvf)_\omega = 0$ for all functions
$\bvf \in \HH^1(\omega)$ such that $\bvf \times \nn_{\omega} = \bzero$ on
$\partial \omega \setminus \gD\}$, where $\nn_{\omega}$
is the unit outward normal to $\omega$.
Similarly, $\HH_\gN(\ddiv,\omega)$ is the subspace of $\HH(\ddiv,\omega)$ composed of fields
with vanishing normal trace on $\gN$, $\HH_\gN(\ddiv,\omega) \eq \{\vv \in \HH(\ddiv,\omega)$
such that $(\div \vv, \varphi)_\omega + (\vv, \grad \varphi)_\omega = 0$ for all functions
$\varphi \in H^1_\gD(\omega)\}$. We refer the reader to~\cite{Fer_Gil_Maxw_BC_97}
for further insight on vector-valued Sobolev spaces with mixed boundary conditions.

\revision{The space $\KK(\Omega) \eq \{ \vv \in \HH_\GD(\ccurl,\Omega) \; | \;
\curl \vv = \bzero \}$ will also play an important role. When
$\Omega$ is simply connected and $\GD$ is connected, one simply
has $\KK(\Omega) = \grad \left (H^1_\GD(\Omega)\right )$. In the general
case, one has $\KK(\Omega) = \grad \left (H^1_\GD(\Omega)\right ) \oplus \LH$, where
$\LH$ is a finite-dimensional space called the ``cohomology space'' associated
with $\Omega$ and the partition of its boundary \cite{Fer_Gil_Maxw_BC_97}.}

\subsection{\revision{The curl--curl problem}}
\label{sec_Maxw}

If $\jj \in \revision{\KK(\Omega)^\perp}$
\revision{(the orthogonality being understood in $\LL^2(\Omega)$)}, then the
classical weak formulation of~\eqref{eq_maxwell_strong} consists in finding a pair
$(\ee,\revision{\vvphi}) \in \HH_{\GD}(\ccurl,\Omega) \times \revision{\KK(\Omega)}$ such that
\begin{equation}
\label{eq_maxwell_weak}
\left \{
\begin{alignedat}{2}
(\ee,\revision{\ttheta})_\Omega &= 0 &\quad& \forall \revision{\ttheta \in \KK(\Omega)}
\\
(\curl \ee,\curl \vv)_\Omega + (\revision{\vvphi},\vv)_\Omega &= (\jj,\vv)_\Omega &\quad& \forall \vv \in \HH_{\GD}(\ccurl,\Omega).
\end{alignedat}
\right .
\end{equation}
Picking the test function $\vv = \revision{\vvphi}$ in the second equation
of~\eqref{eq_maxwell_weak} shows that $\revision{\vvphi = \bzero}$, so that we actually have
\begin{equation}
\label{eq_maxwell_weak_II}
(\curl \ee,\curl \vv)_\Omega = (\jj,\vv)_\Omega \quad \forall \vv \in \HH_{\GD}(\ccurl,\Omega).
\end{equation}
\revision{Note that when $\Omega$ is simply connected and $\GD$ is connected,
the condition $\jj \in \KK(\Omega)^\perp$ simply means that $\jj$ is
divergence-free with vanishing normal trace on $\GN$, $\jj \in \HH_{\GN}(\ddiv,\Omega)$
with $\div \jj = 0$, and the same constraint follows from the first equation
of~\eqref{eq_maxwell_weak} for $\ee$.}

\subsection{Tetrahedral mesh} \label{sec_mesh}

We consider a matching tetrahedral mesh $\TT_h$ of $\Omega$, \ie,
$\bigcup_{K \in \TT_h} \overline K$ $= \overline \Omega$, each $K$ is a tetrahedron,
and the intersection of two distinct tetrahedra is either empty or their common vertex,
edge, or face. We also assume that $\TT_h$ is compatible with the partition
$\partial \Omega = \overline{\GD} \cup \overline{\GN}$ of the boundary, which means
that each boundary face entirely lies either in $\overline{\GD}$ or in $\overline{\GN}$.
We denote by $\EE_h$ the set of edges of the mesh $\TT_h$ and by $\FF_h$ the set of
faces. The mesh is oriented which means that every edge
$\edge\in\EE_h$ is equipped with a fixed unit tangent vector $\ttau_\edge$ and every face
$F\in\FF_h$ is equipped with a fixed unit normal vector $\nn_F$
(see~\cite[Chapter~10]{Ern_Guermond_FEs_I_21}). Finally for every mesh cell $K\in\TT_h$,
$\nn_K$ denotes its unit outward normal vector.
The choice of the orientation is not relevant in what follows,
but we keep it fixed in the whole work.

If $K \in \TT_h$, $\EE_K \subset \EE_h$ denotes the set of edges of $K$, whereas
for each edge $\edge \in \EE_h$,
we denote by $\TTe$ the associated ``edge patch'' that consists of those
tetrahedra $K \in \TT_h$ for which $\edge \in \EE_K$,
see Figure~\ref{fig_patch}. We also employ the notation
$\ome \subset \Omega$ for the open subdomain associated with the patch $\TTe$.
We say that $\edge\in\EE_h$ is a boundary edge if it lies on $\partial\Omega$ and that
it is an interior edge otherwise (in this case, $\edge$ may touch the boundary at one of its
endpoints). The set of boundary edges is partitioned into the subset of Dirichlet edges $\EED_h$
with edges $\edge$ that lie in $\overline{\GD}$ and the subset of Neumann edges $\EEN_h$
collecting the remaining boundary edges.
For all edges $\edge \in \EE_h$, we denote by $\GeN$ the
open subset of $\partial \ome$ corresponding to the collection of faces having
$e$ as edge and lying in $\overline \GN$. Note that for interior edges,
$\GeN$ is empty and that for boundary edges, $\GeN$ never equals the whole $\partial \ome$.
We also set $\GeD \eq (\partial \ome \setminus \GeN)^\circ$.
\revision{Note that, in all situations, $\ome$ is simply connected and $\GeD$ is connected,
so that we do not need to invoke here the cohomology spaces}.

\begin{figure}[htb]
\centerline{\includegraphics[height=0.35\textwidth]{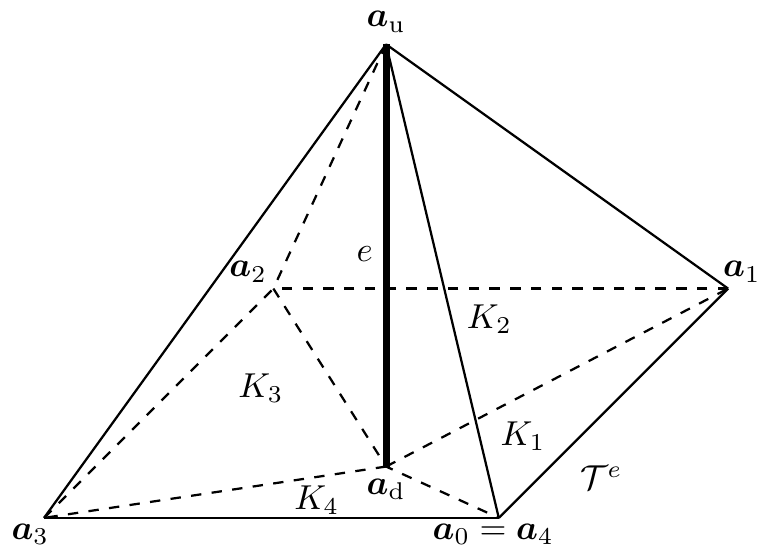} \qquad \qquad \includegraphics[height=0.35\textwidth]{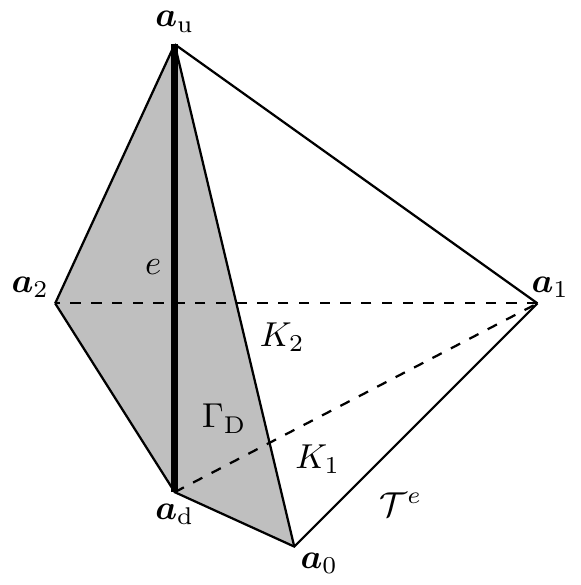}}
\caption{Interior (left) and Dirichlet boundary (right) edge patch $\TTe$}
\label{fig_patch}
\end{figure}

For every tetrahedron $K \in \TT_h$, we denote the diameter and the inscribed ball
diameter respectively by
\begin{equation*}
h_K \eq \sup_{\xx,\yy \in K} |\xx - \yy|,
\quad
\rho_K \eq \sup \left \{ r > 0 \; | \; \exists \xx \in K; B(\xx,r) \subset K \right \},
\end{equation*}
where $B(\xx,r)$ is the ball of diameter $r$ centered at $\xx$.
For every edge $\edge \in \EE_h$, $|\edge|$ is its measure (length) and
\begin{equation}
\label{eq_patch_not}
h_\ome \eq \sup_{\xx,\yy \in \ome} |\xx - \yy|, \quad
\rho_\edge \eq \min_{K \in \TTe} \rho_K.
\end{equation}
The shape-regularity parameters of the tetrahedron $K$ and of the edge patch $\TTe$
are respectively defined by
\begin{equation}
\label{eq_regularities}
\kappa_K \eq h_K/\rho_K \quad \text{ and } \quad \kappa_\edge \eq h_{\ome}/\rho_\edge.
\end{equation}

\subsection{Lagrange, N\'ed\'elec, and Raviart--Thomas elements}

If $K$ is a tetrahedron and $\pp \geq 0$ is an integer,
we employ the notation $\mathcal P_\pp(K)$ for the space of scalar-valued (Lagrange) polynomials
of degree less than or equal to $\pp$ on $K$ and $\widetilde{\mathcal P}_\pp(K)$
for homogeneous polynomials of degree $\pp$. The notation $\PP_\pp(K)$ (resp. $\widetilde{\PP}_\pp(K)$)
then stands for the space of vector-valued polynomials such that all their components belong to
$\mathcal P_\pp(K)$ (resp. $\widetilde{\mathcal P}_\pp(K)$).
Following~\cite{Ned_mix_R_3_80} and~\cite{Ra_Tho_MFE_77}, we then define on each tetrahedron
$K \in \TT_h$ the polynomial spaces of N\'ed\'elec and Raviart--Thomas functions as follows:
\begin{equation} \label{eq_RT_N}
\NN_\pp(K) \eq \PP_\pp(K) + \widetilde{\boldsymbol{\mathcal S}}_{\pp+1}(K) \quad \text{ and }
\quad
\RT_\pp(K) \eq \PP_\pp(K) + \xx \widetilde{\mathcal P}_\pp(K),
\end{equation}
where
$\widetilde{\boldsymbol{\mathcal S}}_\pp(K) \eq \big\{
\vv \in \widetilde{\PP}_\pp(K) \; | \; \xx \cdot \vv(\xx) = 0 \quad \forall \xx \in \overline{K} \big \}$.
For any collection of tetrahedra $\TT = \bigcup_{K \in \TT}\{K\}$ and the corresponding open
subdomain $\omega = \big(\bigcup_{K \in \TT} \overline{K} \big)^\circ \subset \Omega$,
we also write
\begin{align*}
\mathcal P_\pp(\TT)
&\eq
\left \{
v \in L^2(\omega) \; | \; v|_K \in \mathcal P_\pp(K) \quad \forall K \in \TT
\right \},
\\
\NN_\pp(\TT)
&\eq
\left \{
\vv \in \LL^2(\omega) \; | \; \vv|_K \in \NN_\pp(K) \quad \forall K \in \TT
\right \},
\\
\RT_\pp(\TT)
&\eq
\left \{
\vv \in \LL^2(\omega) \; | \; \vv|_K \in \RT_\pp(K) \quad \forall K \in \TT
\right \}.
\end{align*}

\subsection{N\'ed\'elec finite element discretization}

For the discretization of problem~\eqref{eq_maxwell_weak}, we consider in this work\revision{,
for a fixed polynomial degree $p \geq 0$, the N\'ed\'elec finite element space given by}
\begin{equation*}
\VVh \eq \NN_p(\TT_h) \cap \HH_{\GD}(\ccurl,\Omega).
\end{equation*}
\revision{The discrete counterpart of $\KK(\Omega)$, namely
\begin{equation*}
\KK_h \eq \left \{\vv_h \in \VVh \; | \; \curl \vv_h = \bzero \right \}
\end{equation*}
can be readily identified as a preprocessing step by introducing cuts in the mesh
\cite[Chapter 6]{gross_kotiuga_2004a}.}
The discrete problem then consists in finding a pair
$(\ee_h,\revision{\vvphi}_h) \in \VVh \times \revision{\KK_h}$ such that
\begin{equation}
\label{eq_maxwell_discrete}
\left \{
\begin{alignedat}{2}
(\ee_h,\revision{\ttheta_h})_{\Omega} &= 0 && \quad \forall \revision{\ttheta_h \in \KK_h}
\\
(\curl \ee_h,\curl \vv_h)_{\Omega} + (\revision{\vvphi_h},\vv_h)_{\Omega}
&=
(\jj,\vv_h)_{\Omega} && \quad \forall \vv_{h} \in \VVh.
\end{alignedat}
\right.
\end{equation}
Since \revision{$\KK_h \subset \KK(\Omega)$}, picking $\vv_h = \revision{\vvphi_h}$ in the second
equation of~\eqref{eq_maxwell_discrete} shows that \revision{$\vvphi_h = \bzero$}, so that
we actually have
\begin{equation}
\label{eq_maxwell_discrete_II}
(\curl \ee_h,\curl \vv_h)_{\Omega} = (\jj,\vv_h)_{\Omega} \quad \forall \vv_h \in \VVh.
\end{equation}
\revision{As for the continuous problem, we remark that when $\Omega$ is simply connected
and $\GD$ is connected, $\KK_h = \grad S_h$, where
$S_h \eq \mathcal P_{p+1}(\TT_h) \cap H^1_\GD(\Omega)$ is the usual Lagrange finite element space.}


\section{Main results} \label{sec_main_res}

This section presents our two main results.

\subsection{Stable discrete best-approximation of broken polynomials in {\em H}$(\ccurl)$}

Our first main result is the combination and extension of~\cite[Theorem~7.2]{Demk_Gop_Sch_ext_II_09}
and~\cite[Corollary~3.4]{Cost_McInt_Bog_Poinc_10} to the edge patches $\TTe$, complementing
similar previous achievements in $\HH(\ddiv)$ in two space dimensions
in~\cite[Theorem~7]{Brae_Pill_Sch_p_rob_09} and in $H^1$ and $\HH(\ddiv)$ in three space
dimensions~\cite[Corollaries~3.1 and~3.3]{Ern_Voh_p_rob_3D_20}.

\begin{theorem}[$\HH(\ccurl)$ best-approximation in an edge patch]
\label{theorem_stability}
Let an edge $\edge\in\EE_h$ and the associated edge patch $\TTe$ with subdomain $\ome$ be
fixed. Then, for every polynomial degree $p \geq 0$, all
$\jj_h^{\edge} \in \RT_p(\TTe) \cap \HH_{\GeN}(\ddiv,\ome)$
with $\div \jj_h^{\edge} = 0$, and all $\ch_h \in \NN_p(\TTe)$,
the following holds:
\begin{equation}
\label{eq_stab}
\min_{\substack{
\hh_h \in \NN_p(\TTe) \cap \HH_{\GeN}(\ccurl,\ome)
\\
\curl \hh_h = \jj_h^\edge}}
\|\hh_h - \ch_h\|_\ome
\leq
\Cste
\min_{\substack{
\hh \in \HH_{\GeN}(\ccurl,\ome)
\\
\curl \hh = \jj_h^{\edge}}}
\|\hh - \ch_h\|_\ome.
\end{equation}
Here, both minimizers are uniquely defined and the constant $\Cste$ only depends on the
shape-regularity parameter $\kappa_\edge$ of the patch $\TTe$ defined in~\eqref{eq_regularities}.
\end{theorem}

Note that the converse inequality to~\eqref{eq_stab} holds trivially with constant $1$, \ie,
\[
\min_{\substack{
\hh \in \HH_{\GeN}(\ccurl,\ome)
\\
\curl \hh = \jj_h^\edge}}
\|\hh - \ch_h\|_\ome
\leq
\min_{\substack{
\hh_h \in \NN_p(\TTe) \cap \HH_{\GeN}(\ccurl,\ome)
\\
\curl \hh_h = \jj_h^\edge}}
\|\hh_h - \ch_h\|_\ome.
\]
This also makes apparent the power of the result~\eqref{eq_stab}, stating that for piecewise
polynomial data $\jj_h^\edge$ and $\ch_h$, the best-approximation error over a piecewise polynomial
subspace of $\HH_{\GeN}(\ccurl,\ome)$ of degree $p$ is, up to a $p$-independent constant,
equivalent to the best-approximation error over the entire space $\HH_{\GeN}(\ccurl,\ome)$.
The proof of this result is presented in Section~\ref{sec_proof_stability}.
We remark that Proposition~\ref{prop_stability_patch} below gives an equivalent
reformulation of Theorem~\ref{theorem_stability} in the form of a stable broken $\HH(\ccurl)$
polynomial extension in the edge patch. \revision{Finally, the following form, which follows
from the proof in Section~\ref{sec:proof_stability_patch}, see Remark~\ref{rem_sweep_brok},
has  important practical applications:

\begin{theorem}[$\HH(\ccurl)$ best-approximation by an explicit sweep through an edge patch]
\label{thm_sweep}
Let the assumptions of Theorem~\ref{theorem_stability} be satisfied.
Consider a sequential sweep over all elements $K$ sharing the edge $\edge$,
$K \in \TTe$ such that \textup{(i)} the enumeration starts from an arbitrary
tetrahedron if $\edge$ is an interior edge and from a tetrahedron containing a
face that lies in $\GeN$ (if any) or in $\GeD$ (if none in $\GeN$) if $\edge$ is a
boundary edge; \textup{(ii)} two consecutive tetrahedra in the enumeration share a face.
On each $K \in \TTe$, consider
\begin{equation}
\label{eq_def_estimator_sweep}
\hh_h^{\edge,\heartsuit}|_K \eq \min_{\substack{
\hh_h \in \NN_p(K)\\
\curl \hh_h = \jj_h^\edge\\
\hh_h|_{\FF}^\ttau = \rr_\FF}}
\|\hh_h - \ch_h\|_K.
\end{equation}
Here, $\FF$ is the set of faces that $K$ shares with elements $K'$
previously considered or lying in $\GeN$, and $\hh_h|_{\FF}^\ttau$ denotes
the restriction of the tangential trace of $\hh_h$ to the faces of $\FF$
(see Definition \ref{definition_partial_trace} below for details). The boundary datum
$\rr_{\FF}$ is either the tangential trace of $\hh_h^{\edge,\heartsuit}|_{K'}$
obtained after minimization over the previous tetrahedron $K'$, or $\mathbf 0$ on $\GeN$. Then,
\begin{equation}
\label{eq_stab_sweep}
\|\hh_h^{\edge,\heartsuit} - \ch_h\|_\ome
\leq
\Cste
\min_{\substack{
\hh \in \HH_{\GeN}(\ccurl,\ome)
\\
\curl \hh = \jj_h^{\edge}}}
\|\hh - \ch_h\|_\ome.
\end{equation}
\end{theorem}
}

\subsection{$p$-robust \revision{broken patchwise} equilibration a posteriori error
estimates for \revision{the curl--curl problem}}

Our second main result is a polynomial-degree-robust a posteriori error analysis of
N\'ed\'elec finite elements~\eqref{eq_maxwell_discrete} applied to
\revision{curl--curl problem}~\eqref{eq_maxwell_strong}. The local efficiency proof is an
important application of Theorem~\ref{theorem_stability}. To present these results in detail,
we need to prepare a few tools.

\subsubsection{Functional inequalities and data oscillation}
\label{sec_Poinc}

For every edge $\edge \in \EE_h$, we associate with the subdomain $\ome$ a local Sobolev space
$H^1_\star(\ome)$ with mean/boundary value zero,
\begin{equation}
\label{eq_Hse}
H^1_\star(\ome) \eq
\begin{cases}
    \{ v \in H^1(\ome)\; | \; v=0 \text{ on faces having $e$ as edge}\\
     \hspace{5.3cm}\text{and lying in $\overline \GD$}\} \qquad & \text{ if } \edge \in \EED_h,\\
    \left \{ v \in H^1(\ome) \; | \; \int_\ome v = 0 \right \} \qquad & \text{ otherwise}.
\end{cases}
\end{equation}
Poincar\'e's inequality then states that there exists a constant $\CPe$ only
depending on the shape-regularity parameter $\kappa_\edge$ such that
\begin{equation}
\label{eq_local_poincare}
\|v\|_{\ome} \leq \CPe h_\ome \|\grad v\|_{\ome} \qquad \forall v \in H^1_\star(\ome).
\end{equation}

To define our error estimators, it is convenient to introduce a piecewise polynomial
approximation of the datum $\jj\in \HH_{\GN}(\ddiv,\Omega)$ by setting on every
edge patch $\TTe$ associated with the edge $\edge \in \EE_h$,
\begin{equation}
\label{eq_definition_jj_h}
\jj_h^\edge \eq \argmin{\jj_h \in \RT_p(\TTe) \cap \HH_{\GeN}(\ddiv,\ome) \\ \div \jj_h = 0}
\|\jj - \jj_h\|_{\ome}.
\end{equation}
This leads to the following data oscillation estimators:
\begin{equation}
\label{eq_definition_osc}
\osc_\edge \eq \CPVe h_\ome \|\jj - \jj_h^\edge\|_{\ome},
\end{equation}
\revision{where the constant $\CPVe$ is such that for every edge $\edge \in \EE_h$, we have}
\begin{equation}
\label{eq_local_poincare_vectorial}
\|\vv\|_{\ome} \leq \CPVe h_{\ome}\|\curl \vv\|_{\ome}
\quad
\forall \vv \in \HH_{\GeD}(\ccurl,\ome) \cap \HH_{\GeN}(\ddiv,\ome) \text{ with } \div \vv = 0.
\end{equation}
\revision{We show in Appendix~\ref{appendix_weber} that $\CPVe$ only depends on
the shape-regularity parameter $\kappa_\edge$. Notice that \eqref{eq_local_poincare_vectorial}
is a local Poincar\'e-like inequality using the curl of divergence-free fields in the edge patch.
This type of inequality is known under various names in the literature. Seminal
contributions can be found in the work of Friedrichs~\cite[equation~(5)]{Fried:55}
for smooth manifolds (see also Gaffney~\cite[equation~(2)]{Gaffney:55}) and later in
Weber~\cite{Weber:80} for Lipschitz domains. This motivates the present use of the subscript
${}_{\rm PFW}$ in~\eqref{eq_local_poincare_vectorial}.}

\revision{Besides the above local functional inequalities, we shall also use the fact}
that there exists a constant $\Clift$ such that for all $v \in \HH_\GD(\ccurl,\Omega)$,
there exists $\ww \in \HH^1(\Omega) \cap \HH_\GD(\ccurl,\Omega)$ such that $\curl \ww = \curl \vv$
and
\begin{equation}
\label{eq_estimate_lift}
\|\grad \ww\|_{\Omega} \leq \Clift \|\curl \vv\|_\Omega.
\end{equation}
When either $\GD$ or $\GN$ has zero measure, the existence of $\Clift$ follows
from Theorems 3.4 and 3.5 of~\cite{Cost_Dau_Nic_sing_Maxw_99}.
If in addition $\Omega$ is convex, one can take $\Clift = 1$
(see~\cite{Cost_Dau_Nic_sing_Maxw_99} together with~\cite[Theorem~3.7]{Gir_Rav_NS_86}
for Dirichlet boundary conditions and
\cite[Theorem~3.9]{Gir_Rav_NS_86} for Neumann boundary conditions).
\revision{For mixed boundary conditions, the existence of $\Clift$ can be obtained
as a consequence of \cite[Section~2]{Hipt_Pechs_discr_reg_dec_19}. Indeed, we first
project $\vv \in \HH_\GD(\ccurl,\Omega)$ onto $\tvv \in \KK(\Omega)^\perp$ without
changing its curl. Then, we define $\ww \in \HH^1(\Omega)$ from $\tvv$ using
\cite{Hipt_Pechs_discr_reg_dec_19}. Finally, we control $\|\tvv\|_\Omega$ by
$\|\curl \vv\|_\Omega$ with the inequality from \cite[Proposition 7.4]{Fer_Gil_Maxw_BC_97}
which is a global Poincar\'e-like inequality in the spirit of~\eqref{eq_local_poincare_vectorial}.}

\subsubsection{\revision{Broken patchwise} equilibration by edge-patch problems}

Our a posteriori error estimator is constructed via a simple restriction of the right-hand side
of~\eqref{eq_PS} to edge patches, where no hat function is employed, no modification of
the source term appears, and no boundary condition is imposed for interior edges,
in contrast to the usual equilibration
in~\cite{Dest_Met_expl_err_CFE_99, Braess_Scho_a_post_edge_08, Ern_Voh_p_rob_15}.
For each edge $\edge \in \EE_h$, introduce
\begin{subequations}\label{eq_definition_estimator}
\begin{equation}
\eta_\edge \eq \|\hh_h^{\edge,\star} - \curl \ee_h\|_\ome, \label{eq_definition_estimator_1}
\end{equation}
\revision{where $\hh_h^{\edge,\star}$ is the argument of the left minimizer in~\eqref{eq_stab}
for the datum $\jj_h^\edge$ from~\eqref{eq_definition_jj_h} and
$\ch_h \eq (\curl \ee_h)|_{\ome}$, \ie,}
\begin{equation}
\hh_h^{\edge,\star} \eq \argmin{\substack{
\hh_h \in \NN_p(\TTe) \cap \HH_{\GeN}(\ccurl,\ome)
\\
\curl \hh_h = \jj_h^\edge}}
\|\hh_h - \curl \ee_h\|_\ome. \label{eq_definition_estimator_2}
\end{equation}
\end{subequations}
\revision{In practice, $\hh_h^{\edge,\star}$ is computed from the
Euler--Lagrange conditions for the minimization problem~\eqref{eq_definition_estimator_2}.
This leads to the} following patchwise mixed finite element problem:
Find $\hh_h^{\revision{\edge},\star} \in \NN_p(\TTe) \cap \HH_{\GeN}(\ccurl,\ome)$,
$\sig_h^{\revision{\edge},\star} \in \RT_p(\TTe) \cap \HH_{\GeN}(\ddiv,\ome)$,
and $\zeta^{\revision{\edge},\star}_h \in \mathcal P_p(\TTe)$ such that
\begin{equation} \label{broken_equil}
\left \{ \arraycolsep=2pt
\begin{array}{rclclcl}
(\hh_h^{\revision{\edge},\star},\vv_h)_{\ome}
&+&
(\sig_h^{\revision{\edge},\star},\curl \vv_h)_{\ome}
& &
&=&
(\curl \ee_h,\vv_h)_{\ome},
\\
(\curl \hh_h^{\revision{\edge},\star},\ww_h)_{\ome}
&&
&+ &
(\zeta^{\revision{\edge},\star}_h,\div \ww_h)_\ome
&=&
(\jj,\ww_h)_{\ome},
\\
& &
(\div \sig^{\revision{\edge},\star}_h,\varphi_h)_\ome
& &
&=&
0
\end{array}
\right .
\end{equation}
for all
$\vv_h \in \NN_p(\TTe) \cap \HH_{\GeN}(\ccurl,\ome)$,
$\ww_h \in \RT_p(\TTe) \cap \HH_{\GeN}(\ddiv,\ome)$,
and $\varphi_h \in \mathcal P_p(\TTe)$. We note that from the optimality condition associated with~\eqref{eq_definition_jj_h},
using $\jj$ or $\jj_h^\edge$ in~\eqref{broken_equil} is equivalent.

\subsubsection{\revision{Broken patchwise} equilibration by sequential sweeps}

The patch problems~\eqref{eq_definition_estimator_2} lead to the solution of the linear
systems~\eqref{broken_equil}. Although these are local around each edge and are mutually
independent, they \revision{entail} some computational cost. This cost can be significantly
reduced by taking inspiration from~\cite{Dest_Met_expl_err_CFE_99},
\cite{Luce_Wohl_local_a_post_fluxes_04}, \cite[Section~4.3.3]{Voh_guar_rob_VCFV_FE_11},
the proof of~\cite[Theorem~7]{Brae_Pill_Sch_p_rob_09}, or~\cite[Section~6]{Ern_Voh_p_rob_3D_20}
and literally following the proof in Section~\ref{sec:proof_stability_patch} below\revision{,
as summarized in Theorem~\ref{thm_sweep}}.
This leads to \revision{an alternative error estimator} whose price is the sequential sweep through
tetrahedra sharing the given edge, \revision{where for each tetrahedron, one} solve\revision{s
the elementwise problem~\eqref{eq_def_estimator_sweep} for the
datum $\jj_h^\edge$ from~\eqref{eq_definition_jj_h} and $\ch_h \eq (\curl \ee_h)|_{\ome}$, \ie,
\begin{subequations}
\label{eq_definition_estimator_sweep}
\begin{equation}
\label{eq_definition_estimator_sweep_2}
\hh_h^{\edge,\heartsuit}|_K \eq \min_{\substack{
\hh_h \in \NN_p(K)\\
\curl \hh_h = \jj_h^\edge\\
\hh_h|_{\FF}^\ttau = \rr_\FF}}
\|\hh_h - \curl \ee_h\|_K \qquad \forall K \in \TTe,
\end{equation}
and then set
\begin{equation}
\label{eq_definition_estimator_sweep_1}
\eta_\edge \eq \|\hh_h^{\edge,\heartsuit} - \curl \ee_h\|_{\ome}.
\end{equation}
\end{subequations}}

\subsubsection{Guaranteed, locally efficient, and $p$-robust a posteriori error estimates}

For each edge $\edge \in \EE_h$, let $\ppsi_\edge$ be the (scaled) edge basis functions
of the lowest-order N\'ed\'elec space, in particular satisfying
$\supp \ppsi_\edge = \overline{\ome}$. More precisely, let $\ppsi_\edge$
be the unique function in $\NN_0(\TT_h) \cap \HH(\ccurl,\Omega)$ such that
\begin{equation}
\label{eq_BF}
\int_{\edge'} \ppsi_\edge \cdot \ttau_{\edge'} = \delta_{\edge,\edge'} |\edge|,
\end{equation}
recalling that $\ttau_{\edge'}$ is the unit tangent vector orienting the edge $\edge'$.
We define
\begin{equation}
\label{eq_definition_Ccont}
\Cconte
\eq
\|\ppsi_\edge\|_{\infty,\ome}
+
\CPe h_\ome \|\curl \ppsi_\edge\|_{\infty,\ome} \quad \forall \edge \in \EE_h,
\end{equation}
where $\CPe$ is Poincar\'e's constant from~\eqref{eq_local_poincare} and $h_\ome$ is the diameter
of the patch domain $\ome$. We actually show in Lemma~\ref{lem_c_stab} below that
\begin{equation}
\label{eq_bound_Ccont}
\Cconte \leq \Cke \eq \frac{2|\edge|}{\rho_\edge} \left ( 1 + \CPe \kappa_\edge \right )
\quad \forall \edge \in \EE_h,
\end{equation}
where $\rho_\edge$ is defined in~\eqref{eq_patch_not};
$\Cconte$ is thus uniformly bounded by the patch-regularity parameter $\kappa_\edge$
defined in~\eqref{eq_regularities}.

\begin{theorem}[$p$-robust a posteriori error estimate]
\label{theorem_aposteriori}
Let $\ee$ be the weak solution of \revision{the curl--curl problem}~\eqref{eq_maxwell_weak}
and let $\ee_h$ be its N\'ed\'elec finite element approximation
solving~\eqref{eq_maxwell_discrete}. Let the data oscillation estimators $\osc_\edge$ be defined
in~\eqref{eq_definition_osc} and the \revision{broken patchwise} equilibration estimators
$\eta_\edge$ be defined in either~\eqref{eq_definition_estimator}
or~\eqref{eq_definition_estimator_sweep}.
Then, with the constants $\Clift$, $\Cconte$, and $\Cste$ from
respectively~\eqref{eq_estimate_lift}, \eqref{eq_definition_Ccont},
and~\eqref{eq_stab}, the following global upper bound holds true:
\begin{equation}
\label{eq_upper_bound}
\|\curl(\ee - \ee_h)\|_\Omega
\leq
\sqrt{6} \Clift \left (\sum_{\edge \in \EE_h}\Cconte^2
\left (\eta_\edge + \osc_\edge\right )^2 \right )^{1/2},
\end{equation}
as well as the following lower bound local to the edge patches $\ome$:
\begin{equation}
\label{eq_lower_bound}
\eta_\edge \leq \Cste \left (\|\curl(\ee - \ee_h)\|_\ome + \osc_\edge\right )
\qquad \forall \edge \in \EE_h.
\end{equation}
\end{theorem}

\subsection{Comments}

A few comments about Theorem~\ref{theorem_aposteriori} are in order.

\begin{itemize}

\item The constant $\Clift$ from~\eqref{eq_estimate_lift} can be taken as $1$ for convex domains
$\Omega$ and if either $\GD$ or $\GN$ is empty. In the general case however, we do not know
the value of this constant. The presence of the constant $\Clift$ is customary in a posteriori
error analysis of \revision{the curl--curl problem}, it appears, e.g., in Lemma 3.10
of~\cite{Nic_Creus_a_post_Maxw_03} and Assumption 2 of~\cite{Beck_Hipt_Hopp_Wohl_a_post_Maxw_00}.

\item The constant $\Cconte$ defined in~\eqref{eq_definition_Ccont} can be fully computed
in practical implementations. Indeed, computable values of Poincar\'e's constant $\CPe$
from~\eqref{eq_local_poincare} \revision{are discussed in, \eg,
\cite{Chua_Whee_est_Poin_06, Vees_Verf_Poin_stars_12},
see also the concise discussion in~\cite{Blech_Mal_Voh_loc_res_NL_20}; $\CPe$ can be taken
as $1/\pi$ for convex interior patches and as $1$ for most Dirichlet boundary patches.}
Recall also that $\Cconte$ only depends on the shape-regularity parameter $\kappa_\edge$
of the edge patch $\TTe$.

\item A computable upper bound on the constant $\Cste$ from~\eqref{eq_stab} can be obtained
by proceeding as in~\cite[Lemma~3.23]{Ern_Voh_p_rob_15}. The crucial property is again that
$\Cste$ can be uniformly bounded by the shape-regularity parameter $\kappa_\edge$ of the edge
patch $\TTe$.

\item The key feature of the error estimators of Theorem~\ref{theorem_aposteriori} is their
polynomial-degree-ro\-bust\-ness (or, shortly, $p$-robustness). This suggests to use
them in $hp$-adaptation strategies, \cf, \eg,
\cite{Dan_Ern_Sme_Voh_guar_red_18,Demk_hp_book_07,Schwab_hp_98}
and the references therein.

\item In contrast to~\cite{Braess_Scho_a_post_edge_08, Ged_Gee_Per_a_post_Maxw_19,Ged_Gee_Per_Sch_post_Maxw_20, Licht_FEEC_a_post_H_curl_19},
we do not obtain here an equilibrated flux, \ie, a piecewise polynomial $\hh_h^\star$ in the
global space $\NN_p(\TT_h) \cap \HH_{\GN}(\ccurl,\Omega)$ satisfying, for piecewise polynomial
$\jj$, $\curl \hh_h^\star = \jj$. We only obtain from~\eqref{eq_definition_estimator_2}
or~\eqref{eq_definition_estimator_sweep_2} that
$\hh_h^{\edge,\star} \in \NN_p(\TTe) \cap \HH_{\GeN}(\ccurl,\ome)$ and
$\curl \hh_h^{\edge,\star} = \jj$ locally in every edge patch $\TTe$
and similarly for $\hh_h^{\edge,\heartsuit}$, but we do not build
a $\HH_{\GN}(\ccurl,\Omega)$-conforming discrete field;
we call this process \revision{broken patchwise} equilibration.

\item The upper bound~\eqref{eq_upper_bound} does not come from the Prager--Synge
inequality~\eqref{eq_PS} and is typically larger than those obtained from~\eqref{eq_PS}
with an equilibrated flux $\hh_h^\star \in \NN_p(\TT_h) \cap \HH_{\GN}(\ccurl,\Omega)$,
because of the presence of the multiplicative factors $\sqrt{6} \Clift \Cconte$.
On the other hand, it is typically cheaper to compute the upper bound~\eqref{eq_upper_bound}
than those based on an equilibrated flux since
1)
the problems~\eqref{eq_definition_estimator} and~\eqref{eq_definition_estimator_sweep} involve
edge patches, whereas full equilibration would require solving problems also on vertex patches
which are larger than edge patches;
2)
the error estimators are computed in one stage only solving the
problems~\eqref{eq_definition_estimator_2} or~\eqref{eq_definition_estimator_sweep_2};
3)
\revision{
the broken patchwise equilibration procedure enables the construction of a $p$-robust error
estimator using polynomials of degree $p$, in contrast to the usual procedure requiring the
use of polynomials of degree $p+1$,
\cf~\cite{Brae_Pill_Sch_p_rob_09, Ern_Voh_p_rob_15, Ern_Voh_p_rob_3D_20};
the reason is that the usual procedure involves multiplication by the ``hat function''
$\psi_{\boldsymbol a}$ inside the estimators, which increases the polynomial degree by one,
whereas the current procedure only encapsulates an operation featuring $\ppsi_\edge$
into the multiplicative constant $\Cconte$, see~\eqref{eq_definition_Ccont}.}

\item The sequential sweep through the patch in~\eqref{eq_definition_estimator_sweep_2}
eliminates the patchwise problems~\eqref{eq_definition_estimator_2} and leads instead to
merely elementwise problems. These are much cheaper than~\eqref{eq_definition_estimator_2},
and, in particular, for $p=0$, \ie, for lowest-order N\'ed\'elec elements
in~\eqref{eq_maxwell_discrete} with one unknown per edge, \revision{they can be made} explicit.
\revision{Indeed, there is only one unknown  in~\eqref{eq_definition_estimator_sweep_2} for each
tetrahedron $K \in \TTe$ if $K$ is not the first or the last tetrahedron in the sweep. In the
last tetrahedron, there is no unknown left except if it contains a face that lies in $\GeD$,
in which case there is also only one unknown in~\eqref{eq_definition_estimator_sweep_2}. If the
first tetrahedron contains a face that lies in $\GeN$, there is again only one unknown
in~\eqref{eq_definition_estimator_sweep_2}. Finally, if the first tetrahedron does not contain
a face that lies in $\GeN$, it is possible, instead of $\FF = \emptyset$, to consider the
set $\FF$ formed by the face $F$ that either 1) lies in $\GeD$ (if any) or 2) is shared with
the last element and to employ for the boundary datum $\rr_{\FF}$
in~\eqref{eq_definition_estimator_sweep_2} the 1) value or 2) the mean value of the tangential
trace $\curl \ee_h$ on $F$. This again leads to only one unknown
in~\eqref{eq_definition_estimator_sweep_2}, with all the theoretical properties maintained.}
\end{itemize}

\section{Numerical experiments}
\label{sec_num}

In this section, we present some numerical experiments to illustrate the a posteriori
error estimates from Theorem~\ref{theorem_aposteriori} and its use within an adaptive
mesh refinement procedure. We consider a test case with a smooth solution and a test
case with a solution featuring an edge singularity.

\revision{Below, we rely on the indicator $\eta_\edge$ evaluated
using~\eqref{eq_definition_estimator}, \ie, involving the edge-patch
solves~\eqref{broken_equil}. Moreover, we let
\begin{equation} \label{eq_ests}
\left (\eta_{\rm ofree}\right )^2 \eq 6 \sum_{e \in \EE_h} \left (\Cconte \eta_\edge\right )^2, \qquad
\left (\eta_{{\rm cofree}}\right )^2 \eq \sum_{e \in \EE_h} \left (\eta_\edge\right )^2.
\end{equation}
Here, $\eta_{\rm ofree}$ corresponds to an ``oscillation-free'' error estimator, obtained
by discarding the oscillation terms $\osc_\edge$ in~\eqref{eq_upper_bound}, whereas
$\eta_{{\rm cofree}}$ corresponds to a ``constant-and-oscillation-free'' error estimator,
discarding in addition the multiplicative constants $\sqrt{6} \Clift$ and $\Cconte$}.

\subsection{Smooth solution in the unit cube}

We first consider an example in the unit cube $\Omega \eq (0,1)^3$
and Neumann boundary conditions, $\GN \eq \partial\Omega$ in~\eqref{eq_maxwell_strong}
and its weak form~\eqref{eq_maxwell_weak}. The analytical solution reads
\begin{equation*}
\ee(\xx) \eq \left (
\begin{array}{c}
 \sin(\pi \xx_1)\cos(\pi \xx_2)\cos(\pi \xx_3)
\\
-\cos(\pi \xx_1)\sin(\pi \xx_2)\cos(\pi \xx_3)
\\
0
\end{array}
\right ).
\end{equation*}
One checks that $\div \ee = 0$ and that
\begin{equation*}
\jj(\xx) \eq (\curl \curl \ee)(\xx) =
3\pi^2 \left (
\begin{array}{c}
 \sin(\pi \xx_1)\cos(\pi \xx_2)\cos(\pi \xx_3)
\\
-\cos(\pi \xx_1)\sin(\pi \xx_2)\cos(\pi \xx_3)
\\
0
\end{array}
\right ).
\end{equation*}
We notice that $\Clift = 1$ since $\Omega$ is convex.

We first propose an ``$h$-convergence'' test case in which, for a fixed polynomial degree $p$,
we study the behavior of the N\'ed\'elec approximation $\ee_h$
solving~\eqref{eq_maxwell_discrete} and of the error estimator of
Theorem~\ref{theorem_aposteriori}. We consider a sequence of meshes obtained by first
splitting the unit cube into an $N \times N \times N$ Cartesian grid and then splitting
each of the small cubes into six tetrahedra, with the resulting mesh size $h = \sqrt{3}/N$.
More precisely, each resulting edge patch is convex here, so that the constant $\CPe$
in~\eqref{eq_definition_Ccont} can be taken as $1/\pi$ for all internal patches, see the
discussion in Section~\ref{sec_Poinc}. Figure~\ref{figure_unit_cube_hconv} presents the results.
The top-left panel shows that the expected convergence rates of $\ee_h$ are obtained for
$p=0,\dots,3$. The top-right panel presents the local efficiency of the error estimator
based on the \revision{indicator} $\eta_\edge$ evaluated using~\eqref{eq_definition_estimator}.
We see that it is very good, the ratio of the patch indicator to the patch error being at
most $2$ for $p=0$, and close to $1$ for higher-order
polynomials. This seems to indicate that the constant $\Cste$ in~\eqref{eq_stab} is rather small.
The bottom panels of Figure~\ref{figure_unit_cube_hconv} report on the global efficiency of the
error indicators \revision{$\eta_{\rm ofree}$ and $\eta_{{\rm cofree}}$ from~\eqref{eq_ests}.}
As shown in the bottom-right panel, the global
efficiency of $\eta_{{\rm cofree}}$ is independent of the mesh size.
The bottom-left panel shows a slight dependency of the global efficiency of
$\eta_{\rm ofree}$ on the mesh size, but this is only due to the fact that
Poincar\'e's constants differ for boundary and internal patches.
These two panels show that the efficiency actually slightly improves as the polynomial
degree is increased, highlighting the $p$-robustness of the proposed error estimator.
\revision{We also notice that the multiplicative factor $\sqrt{6}\Cconte$ can lead to
some error overestimation.}

\begin{figure}
\begin{minipage}{.45\linewidth}
\begin{tikzpicture}
\begin{axis}[
	xmode = log,
	ymode = log,
	x dir = reverse,
	width=\linewidth,
	xlabel=$h/\sqrt{3}$,
	ylabel=$\|\curl(\ee-\ee_h)\|_\Omega$,
	xtick={1,.5,.25,.125,.0625,.03125},
	xticklabels={1,1/2,1/4,1/8,1/16,1/32}
]

\plot[color=black,mark=*]    table[x expr=1/\thisrow{N},y=true_err]%
	{figures/numerics/unit_cube/data/hconv_P0.dat};
\plot[color=blue,mark=o]    table[x expr=1/\thisrow{N},y=true_err]%
	{figures/numerics/unit_cube/data/hconv_P1.dat};
\plot[color=red,mark=square]    table[x expr=1/\thisrow{N},y=true_err]%
	{figures/numerics/unit_cube/data/hconv_P2.dat};
\plot[color=green,mark=square*]    table[x expr=1/\thisrow{N},y=true_err]%
	{figures/numerics/unit_cube/data/hconv_P3.dat};

\SlopeTriangle{.7}{-.2}{.7}{-1}{$h$}{}

\SlopeTriangle{.7}{-.2}{.45}{-2}{$h^2$}{}

\SlopeTriangle{.6}{-.1}{.35}{-3}{$h^3$}{}

\SlopeTriangle{.5}{-.1}{.175}{-4}{$h^4$}{}

\end{axis}
\end{tikzpicture}
\end{minipage}
\quad\begin{minipage}{.45\linewidth}
\begin{tikzpicture}
\begin{axis}[
	xmode = log,
	x dir = reverse,
	width=\linewidth,
	xlabel=$h/\sqrt{3}$,
	ylabel=$\max_{e \in \EE_h} \revision{\eta_\edge}/\|\curl(\ee-\ee_h)\|_\ome$,
	xtick={1,.5,.25,.125,.0625,.03125},
	xticklabels={1,1/2,1/4,1/8,1/16,1/32}
]

\plot[color=black,mark=*]       table[x expr=1/\thisrow{N},y expr=1/\thisrow{min_ratio}]%
	{figures/numerics/unit_cube/data/hconv_P0.dat};
\plot[color=blue,mark=o]        table[x expr=1/\thisrow{N},y expr=1/\thisrow{min_ratio}]%
	{figures/numerics/unit_cube/data/hconv_P1.dat};
\plot[color=red,mark=square]    table[x expr=1/\thisrow{N},y expr=1/\thisrow{min_ratio}]%
	{figures/numerics/unit_cube/data/hconv_P2.dat};
\plot[color=green,mark=square*] table[x expr=1/\thisrow{N},y expr=1/\thisrow{min_ratio}]%
	{figures/numerics/unit_cube/data/hconv_P3.dat};
\end{axis}
\end{tikzpicture}
\end{minipage}

\begin{minipage}{.45\linewidth}
\begin{tikzpicture}
\begin{axis}[
	xmode = log,
	x dir = reverse,
	width=\linewidth,
	xlabel=$h/\sqrt{3}$,
	ylabel=$\revision{\eta_{\rm ofree}}/\|\curl(\ee-\ee_h)\|_\Omega$,
	xtick={1,.5,.25,.125,.0625,.03125},
	xticklabels={1,1/2,1/4,1/8,1/16,1/32}
]

\plot[color=black,mark=*]       table[x expr=1/\thisrow{N},y expr=2.449489*\thisrow{err_cont}/\thisrow{true_err}]%
	{figures/numerics/unit_cube/data/hconv_P0.dat};
\plot[color=blue,mark=o]        table[x expr=1/\thisrow{N},y expr=2.449489*\thisrow{err_cont}/\thisrow{true_err}]%
	{figures/numerics/unit_cube/data/hconv_P1.dat};
\plot[color=red,mark=square]    table[x expr=1/\thisrow{N},y expr=2.449489*\thisrow{err_cont}/\thisrow{true_err}]%
	{figures/numerics/unit_cube/data/hconv_P2.dat};
\plot[color=green,mark=square*] table[x expr=1/\thisrow{N},y expr=2.449489*\thisrow{err_cont}/\thisrow{true_err}]%
	{figures/numerics/unit_cube/data/hconv_P3.dat};
\end{axis}
\end{tikzpicture}
\end{minipage}
\quad\begin{minipage}{.45\linewidth}
\begin{tikzpicture}
\begin{axis}[
	xmode = log,
	x dir = reverse,
	width=\linewidth,
	xlabel=$h/\sqrt{3}$,
	ylabel=$\revision{\eta_{\rm cofree}}/\|\curl(\ee-\ee_h)\|_\Omega$,
	xtick={1,.5,.25,.125,.0625,.03125},
	xticklabels={1,1/2,1/4,1/8,1/16,1/32}
]

\plot[color=black,mark=*]       table[x expr=1/\thisrow{N},y expr=\thisrow{err_sum}/\thisrow{true_err}]%
	{figures/numerics/unit_cube/data/hconv_P0.dat};
\plot[color=blue,mark=o]        table[x expr=1/\thisrow{N},y expr=\thisrow{err_sum}/\thisrow{true_err}]%
	{figures/numerics/unit_cube/data/hconv_P1.dat};
\plot[color=red,mark=square]    table[x expr=1/\thisrow{N},y expr=\thisrow{err_sum}/\thisrow{true_err}]%
	{figures/numerics/unit_cube/data/hconv_P2.dat};
\plot[color=green,mark=square*] table[x expr=1/\thisrow{N},y expr=\thisrow{err_sum}/\thisrow{true_err}]%
	{figures/numerics/unit_cube/data/hconv_P3.dat};
\end{axis}
\end{tikzpicture}
\end{minipage}

\begin{tikzpicture}
\draw[color=black] (0,0) -- (1,0);
\draw[fill=black] (.5,0) circle (.1);
\draw (1,0) node[anchor=west] {$p=0$};

\draw[color=blue] (0+3,0) -- (1+3,0);
\draw[color=blue] (.5+3,0) circle (.1);
\draw (1+3,0) node[anchor=west] {$p=1$};

\draw[color=red] (0+6,0) -- (1+6,0);
\draw[color=red] (.5-.1+6,0-.1) rectangle (.5+.1+6,0+.1);
\draw (1+6,0) node[anchor=west] {$p=2$};

\draw[color=green] (0+9,0) -- (1+9,0);
\draw[color=green,fill=green] (.5-.1+9,0-.1) rectangle (.5+.1+9,0+.1);
\draw (1+9,0) node[anchor=west] {$p=3$};
\end{tikzpicture}

\caption{$h$-convergence for the unit cube experiment}
\label{figure_unit_cube_hconv}
\end{figure}
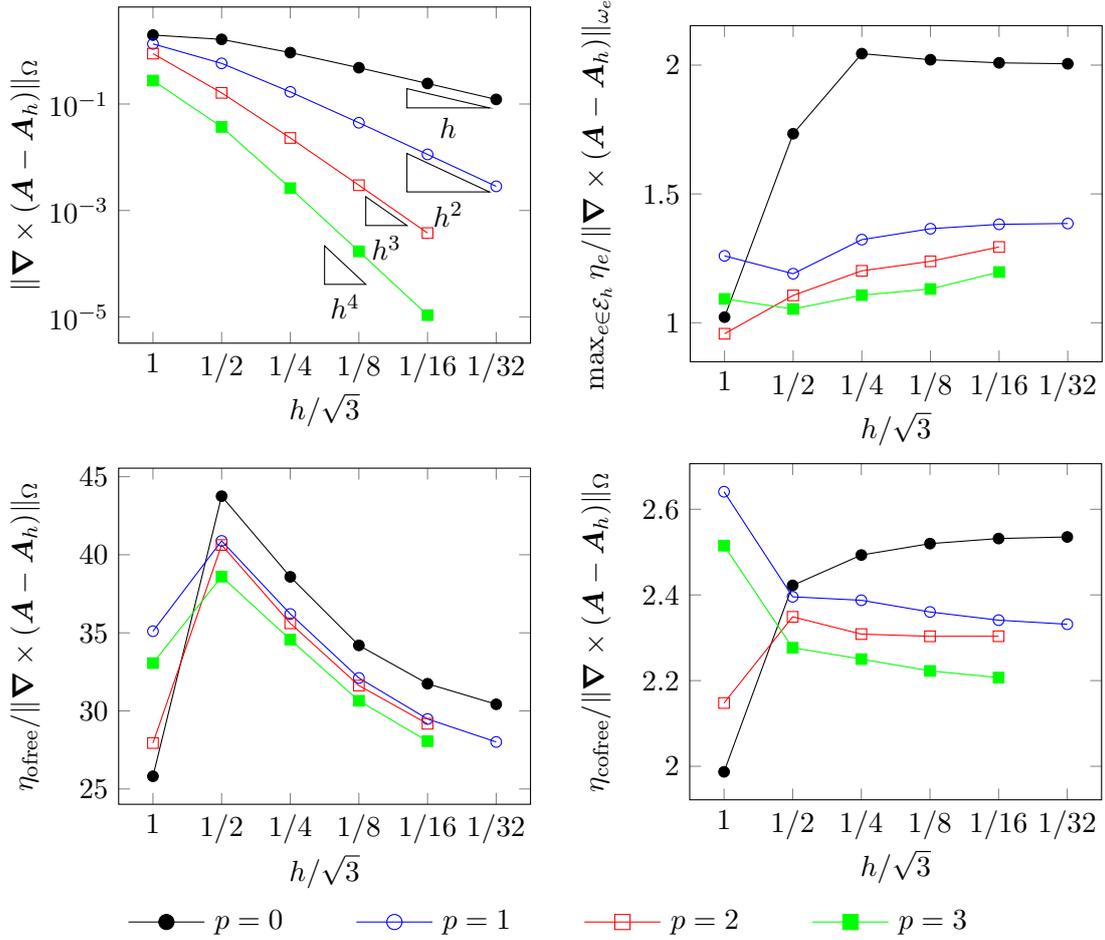

We then present a ``$p$-convergence'' test case where for a fixed mesh, we study
the behavior of the solution and of the error estimator when the polynomial degree
$p$ is increased. We provide this analysis for four different meshes. The three first meshes
are structured as previously described with $N=1,2$, and $4$, whereas the last
mesh is unstructured. The unstructured mesh has $358$ elements, $1774$ edges,
and $h = 0.37$. Figure~\ref{figure_unit_cube_pconv} presents the results.
The top-left panel shows an exponential convergence rate as $p$ is increased for all the
meshes, which is in agreement with the theory, since the solution is analytic.
The top-right panel shows that the local patch-by-patch
efficiency is very good, and seems to tend to $1$ as $p$ increases.
The bottom-right panel shows that the global efficiency of $\eta_{{\rm cofree}}$
also slightly improves as $p$ is increased, and it seems to be rather
independent of the mesh. The bottom-left panel shows that the global efficiency of
$\eta_{\rm ofree}$ is significantly worse on the unstructured mesh. This is because
in the absence of convex patches, we employ for $\CPe$ the estimate
from~\cite{Vees_Verf_Poin_stars_12} instead of the constant $1/\pi$. We believe that this
performance could be improved by providing sharper Poincar\'e constants.

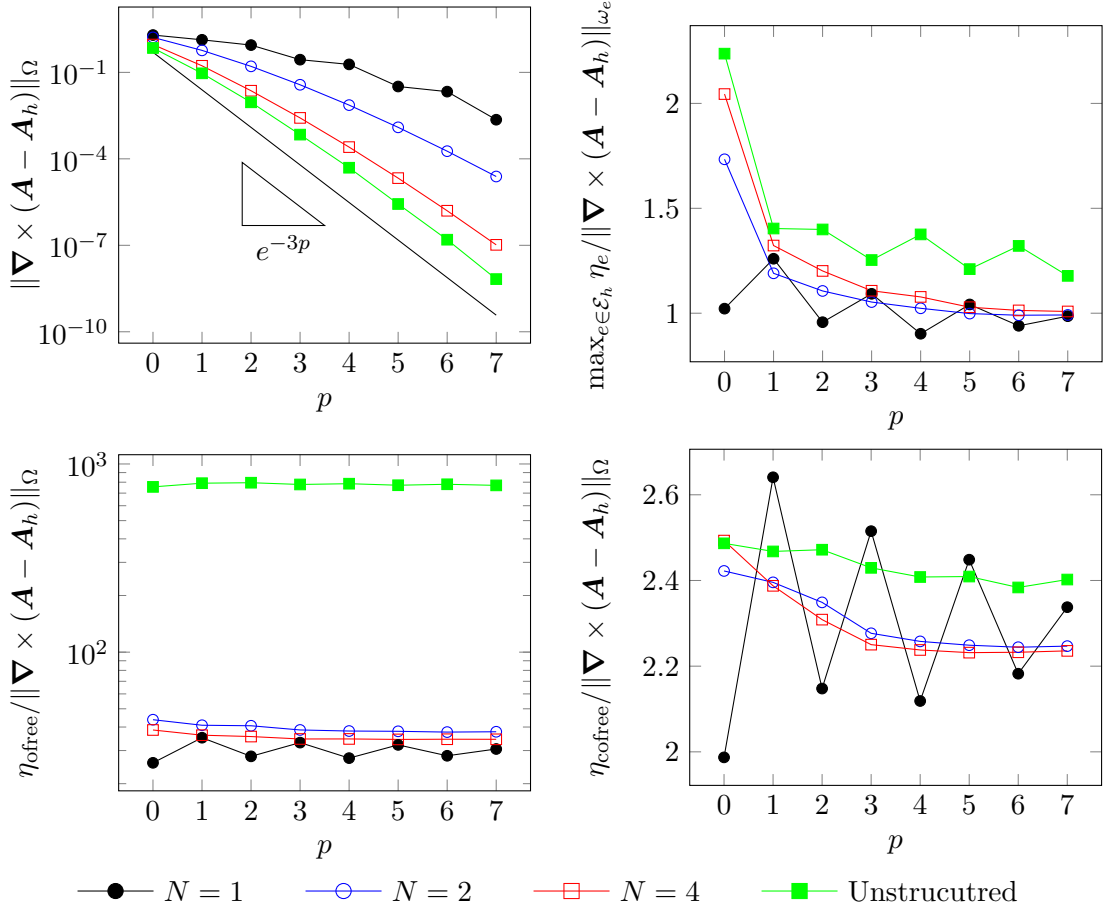
\begin{figure}
\begin{minipage}{.45\linewidth}
\begin{tikzpicture}
\begin{axis}[
	ymode = log,
	width=\linewidth,
	xlabel=$p$,
	ylabel=$\|\curl(\ee-\ee_h)\|_\Omega$,
	xtick={0,1,2,3,4,5,6,7}
]

\plot[color=black,mark=*]    table[x=deg,y=true_err]%
	{figures/numerics/unit_cube/data/pconv_N01.dat};
\plot[color=blue,mark=o]    table[x=deg,y=true_err]%
	{figures/numerics/unit_cube/data/pconv_N02.dat};
\plot[color=red,mark=square]    table[x=deg,y=true_err]%
	{figures/numerics/unit_cube/data/pconv_N04.dat};
\plot[color=green,mark=square*]    table[x=deg,y=true_err]%
	{figures/numerics/unit_cube/data/pconv_Nus.dat};

\addplot[domain=0:7]{10*exp(-3*(x+1))};

\SlopeTriangle{.3}{-.2}{.35}{-3}{$e^{-3p}$}{}

\end{axis}
\end{tikzpicture}
\end{minipage}
\quad\begin{minipage}{.45\linewidth}
\begin{tikzpicture}
\begin{axis}[
	width=\linewidth,
	xlabel=$p$,
	ylabel=$\max_{e \in \EE_h} \revision{\eta_\edge}/\|\curl(\ee-\ee_h)\|_\ome$,
	xtick={0,1,2,3,4,5,6,7}
]
\plot[color=black,mark=*]    table[x=deg,y expr=1/\thisrow{min_ratio}]%
	{figures/numerics/unit_cube/data/pconv_N01.dat};
\plot[color=blue,mark=o]    table[x=deg,y expr=1/\thisrow{min_ratio}]%
	{figures/numerics/unit_cube/data/pconv_N02.dat};
\plot[color=red,mark=square]    table[x=deg,y expr=1/\thisrow{min_ratio}]%
	{figures/numerics/unit_cube/data/pconv_N04.dat};
\plot[color=green,mark=square*]    table[x=deg,y expr=1/\thisrow{min_ratio}]%
	{figures/numerics/unit_cube/data/pconv_Nus.dat};

\end{axis}
\end{tikzpicture}
\end{minipage}

\begin{minipage}{.45\linewidth}
\begin{tikzpicture}
\begin{axis}[
	width=\linewidth,
	xlabel=$p$,
	ylabel=$\revision{\eta_{\rm ofree}}/\|\curl(\ee-\ee_h)\|_\Omega$,
	xtick={0,1,2,3,4,5,6,7},
	ymode=log
]
\plot[color=black,mark=*]    table[x=deg,y expr=2.449489*\thisrow{err_cont}/\thisrow{true_err}]%
	{figures/numerics/unit_cube/data/pconv_N01.dat};
\plot[color=blue,mark=o]    table[x=deg,y expr=2.449489*\thisrow{err_cont}/\thisrow{true_err}]%
	{figures/numerics/unit_cube/data/pconv_N02.dat};
\plot[color=red,mark=square]    table[x=deg,y expr=2.449489*\thisrow{err_cont}/\thisrow{true_err}]%
	{figures/numerics/unit_cube/data/pconv_N04.dat};
\plot[color=green,mark=square*]    table[x=deg,y expr=2.449489*\thisrow{err_cont}/\thisrow{true_err}]%
	{figures/numerics/unit_cube/data/pconv_Nus.dat};

\end{axis}
\end{tikzpicture}
\end{minipage}
\quad\begin{minipage}{.45\linewidth}
\begin{tikzpicture}
\begin{axis}[
	width=\linewidth,
	xlabel=$p$,
	ylabel=$\revision{\eta_{\rm cofree}}/\|\curl(\ee-\ee_h)\|_\Omega$,
	xtick={0,1,2,3,4,5,6,7}
]
\plot[color=black,mark=*]    table[x=deg,y expr=\thisrow{err_sum}/\thisrow{true_err}]%
	{figures/numerics/unit_cube/data/pconv_N01.dat};
\plot[color=blue,mark=o]    table[x=deg,y expr=\thisrow{err_sum}/\thisrow{true_err}]%
	{figures/numerics/unit_cube/data/pconv_N02.dat};
\plot[color=red,mark=square]    table[x=deg,y expr=\thisrow{err_sum}/\thisrow{true_err}]%
	{figures/numerics/unit_cube/data/pconv_N04.dat};
\plot[color=green,mark=square*]    table[x=deg,y expr=\thisrow{err_sum}/\thisrow{true_err}]%
	{figures/numerics/unit_cube/data/pconv_Nus.dat};

\end{axis}
\end{tikzpicture}
\end{minipage}

\begin{tikzpicture}
\draw[color=black] (0,0) -- (1,0);
\draw[fill=black] (.5,0) circle (.1);
\draw (1,0) node[anchor=west] {$N=1$};

\draw[color=blue] (0+3,0) -- (1+3,0);
\draw[color=blue] (.5+3,0) circle (.1);
\draw (1+3,0) node[anchor=west] {$N=2$};

\draw[color=red] (0+6,0) -- (1+6,0);
\draw[color=red] (.5-.1+6,0-.1) rectangle (.5+.1+6,0+.1);
\draw (1+6,0) node[anchor=west] {$N=4$};

\draw[color=green] (0+9,0) -- (1+9,0);
\draw[color=green,fill=green] (.5-.1+9,0-.1) rectangle (.5+.1+9,0+.1);
\draw (1+9,0) node[anchor=west] {Unstrucutred};
\end{tikzpicture}

\caption{$p$-convergence for the unit cube experiment}
\label{figure_unit_cube_pconv}
\end{figure}

\subsection{Singular solution in an L-shaped domain}

We now turn our attention to an L-shaped domain featuring a singular solution.
Specifically, $\Omega \eq L \times (0,1)$, where
\begin{equation*}
L \eq \left \{
\xx = (r\cos\theta,r\sin\theta) \ | \
|\xx_1|,|\xx_2| \leq 1 \quad 0 \leq \theta \leq 3\pi/2
\right \},
\end{equation*}
see Figure~\ref{figure_lshape_errors}, where $L$ is represented.
\revision{We consider the case $\GD \eq \partial \Omega$, and t}he solution reads
$\ee(\xx) \eq \big(0,0,\chi(r) r^\alpha \sin(\alpha \theta)\big)^{\mathrm{T}}$,
where $\alpha \eq 3/2$, $r^2 \eq |\xx_1|^2 + |\xx_2|^2$, $(\xx_1,\xx_2) = r(\cos\theta,\sin\theta)$,
and $\chi:(0,1) \to \mathbb R$ is a smooth cutoff function such that $\chi = 0$ in a neighborhood of
$1$. We emphasize that $\div \ee = 0$ and that, since
$\Delta \left (r^\alpha \sin(\alpha \theta)\right ) = 0$ near the origin, the right-hand side
$\jj$ associated with $\ee$ belongs to $\HH(\ddiv,\Omega)$.

We use an adaptive mesh-refinement strategy based on D\"orfler's
marking~\cite{Dorf_cvg_FE_96}. The initial mesh we employ for $p=0$ and $p=1$
consists of $294$ elements and $1418$ edges with $h=0.57$, whereas a mesh with $23$ elements,
$86$ edges, and $h=2$ is employed for $p=2$ and $3$. The meshing package {\tt MMG3D} is employed to generate
the sequence of adapted meshes~\cite{Dobr_mmg3d}.
Figure~\ref{figure_lshape_conv} shows the convergence
histories of the adaptive algorithm for different values of $p$. In the top-left panel,
we observe the optimal convergence rate (limited to $N_{\rm dofs}^{2/3}$ for isotropic
elements in the presence of an edge singularity).
We employ the indicator $\eta_{{\rm cofree}}$ defined in~\eqref{eq_ests}.
The top-right and bottom-left panels respectively present the
\revision{local} and \revision{global} efficiency
indices. In both cases, the efficiency is good considering that the mesh is fully
unstructured with localized features. We also emphasize that the efficiency does not
deteriorate when $p$ increases.

Finally, Figure~\ref{figure_lshape_errors} depicts the estimated and the actual errors at the
last iteration of the adaptive algorithm. The face on the top of the domain $\Omega$ is
represented, and the colors are associated with the edges of the mesh. The left panels correspond
to the values of the estimator $\eta_\edge$ of~\eqref{eq_definition_estimator}, whereas the value
of $\|\curl(\ee-\ee_h)\|_\ome$ is represented in the right panels. Overall, this figure shows
excellent agreement between the estimated and actual error distribution.

\begin{figure}
\begin{minipage}{.45\linewidth}
\begin{tikzpicture}
\begin{axis}[
	xmode = log,
	ymode = log,
	width=\linewidth,
	xlabel=$N_{\rm dofs}$,
	ylabel=$\|\curl(\ee-\ee_h)\|_\Omega$
]

\plot[color=black,mark=*] table[x expr=\thisrow{nr_dofs},y=true_err]%
	{figures/numerics/lshape/data/P0.dat};

\plot[color=blue,mark=o] table[x expr=\thisrow{nr_dofs},y=true_err]%
	{figures/numerics/lshape/data/P1.dat};

\plot[color=red,mark=square] table[x expr=\thisrow{nr_dofs},y=true_err]%
	{figures/numerics/lshape/data/P2.dat};

\plot[color=green,mark=square*] table[x expr=\thisrow{nr_dofs},y=true_err]%
	{figures/numerics/lshape/data/P3.dat};

\SlopeTriangle{.4}{-.1}{.2}{-.6666}{$N_{\rm dofs}^{2/3}$}{}

\SlopeTriangle{.7}{-.1}{.4}{-.3333}{$N_{\rm dofs}^{1/3}$}{}

\end{axis}
\end{tikzpicture}
\end{minipage}
\quad\begin{minipage}{.45\linewidth}
\begin{tikzpicture}
\begin{axis}[
	xmode=log,
	width=\linewidth,
	xlabel=$N_{\rm dofs}$,
	ylabel=$\max_{e \in \EE_h} \eta_e/\|\curl(\ee-\ee_h)\|_\ome$
]

\plot[color=black,mark=*] table[x expr=\thisrow{nr_dofs},y expr=1/\thisrow{min_ratio}]%
	{figures/numerics/lshape/data/P0.dat};

\plot[color=blue,mark=o] table[x expr=\thisrow{nr_dofs},y expr=1/\thisrow{min_ratio}]%
	{figures/numerics/lshape/data/P1.dat};

\plot[color=red,mark=square] table[x expr=\thisrow{nr_dofs},y expr=1/\thisrow{min_ratio}]%
	{figures/numerics/lshape/data/P2.dat};

\plot[color=green,mark=square*] table[x expr=\thisrow{nr_dofs},y expr=1/\thisrow{min_ratio}]%
	{figures/numerics/lshape/data/P3.dat};


\end{axis}
\end{tikzpicture}
\end{minipage}

\begin{minipage}{.45\linewidth}
\begin{tikzpicture}
\begin{axis}[
	xmode=log,
	width=\linewidth,
	xlabel=$N_{\rm dofs}$,
	ylabel=$\eta_{\rm cofree}/\|\curl(\ee-\ee_h)\|_\Omega$
]

\plot[color=black,mark=*] table[x expr=\thisrow{nr_dofs},y expr=\thisrow{est_err}/\thisrow{true_err}]%
	{figures/numerics/lshape/data/P0.dat};

\plot[color=blue,mark=o] table[x expr=\thisrow{nr_dofs},y expr=\thisrow{est_err}/\thisrow{true_err}]%
	{figures/numerics/lshape/data/P1.dat};

\plot[color=red,mark=square] table[x expr=\thisrow{nr_dofs},y expr=\thisrow{est_err}/\thisrow{true_err}]%
	{figures/numerics/lshape/data/P2.dat};

\plot[color=green,mark=square*] table[x expr=\thisrow{nr_dofs},y expr=\thisrow{est_err}/\thisrow{true_err}]%
	{figures/numerics/lshape/data/P3.dat};


\end{axis}
\end{tikzpicture}
\end{minipage}
\quad\begin{minipage}{.45\linewidth}
\hspace{1cm}
\begin{tikzpicture}
\draw[color=black] (0,0) -- (1,0);
\draw[fill=black] (.5,0) circle (.1);
\draw (1,0) node[anchor=west] {$p=0$};

\draw[color=blue] (0,0+1) -- (1,0+1);
\draw[blue] (.5,0+1) circle (.1);
\draw (1,0+1) node[anchor=west] {$p=1$};

\draw[color=red] (0,0+2) -- (1,0+2);
\draw[color=red] (.5-.1,0+2-.1) rectangle (.5+.1,0+2+.1);
\draw (1,0+2) node[anchor=west] {$p=2$};

\draw[color=green] (0,0+3) -- (1,0+3);
\draw[color=green,fill=green] (.5-.1,0+3-.1) rectangle (.5+.1,0+3+.1);
\draw (1,0+3) node[anchor=west] {$p=3$};
\end{tikzpicture}
\end{minipage}

\caption{Convergence histories for the L-shaped domain experiment}
\label{figure_lshape_conv}
\end{figure}
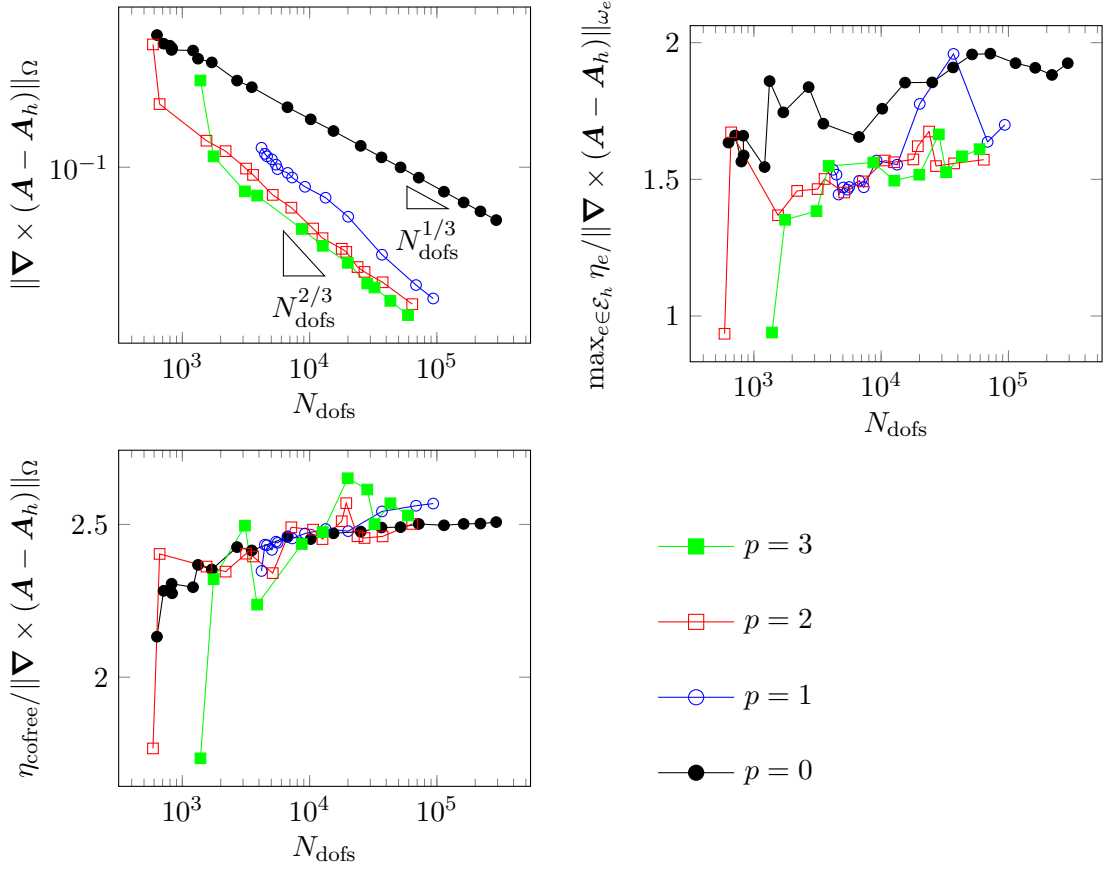

\begin{figure}
\includegraphics[width=\linewidth]{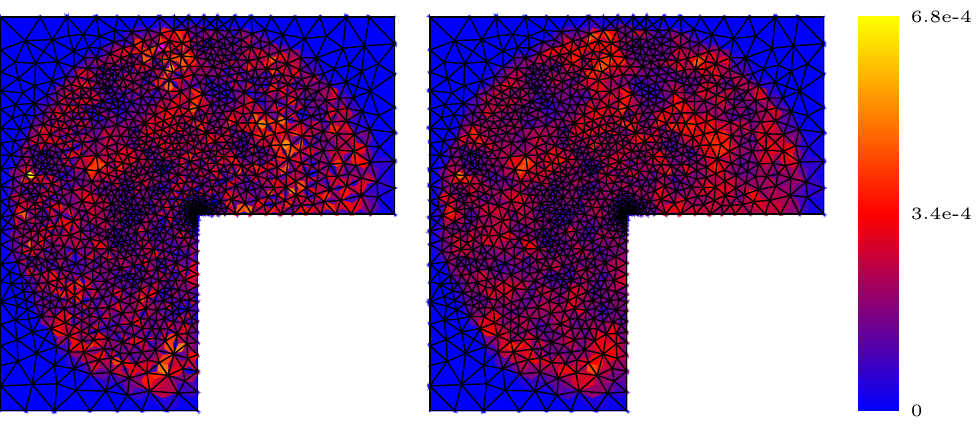}

$p=0$

\includegraphics[width=\linewidth]{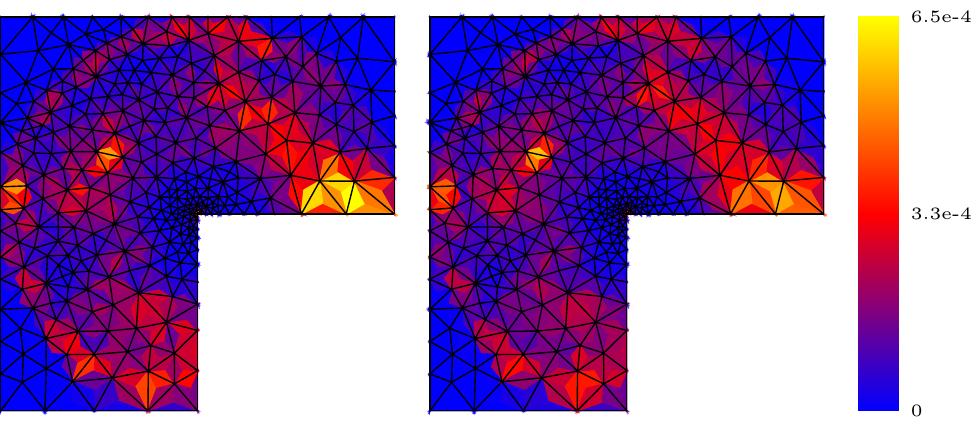}

$p=1$

\includegraphics[width=\linewidth]{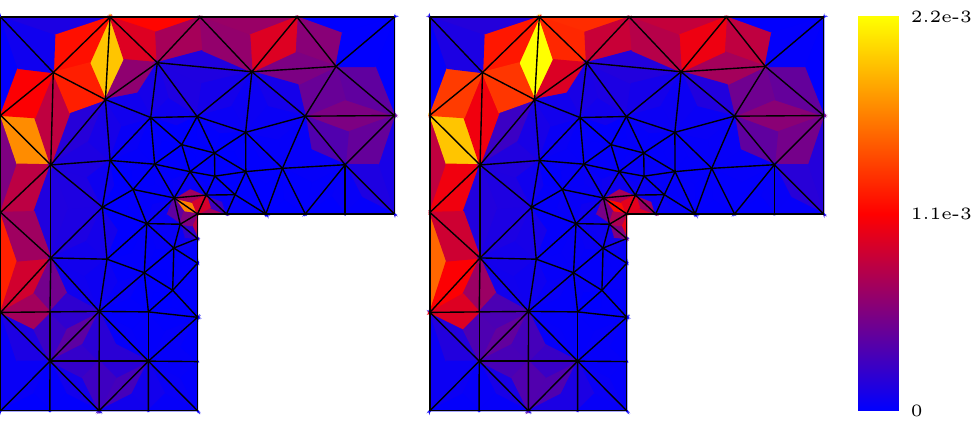}

$p=3$

\caption{Estimated error (left) and actual error (right) for the L-shaped domain experiment}
\label{figure_lshape_errors}
\end{figure}

\section{Proof of Theorem~\ref{theorem_aposteriori} ($p$-robust a posteriori error estimate)}
\label{sec_proof_a_post}

In this section, we prove Theorem~\ref{theorem_aposteriori}.

\subsection{Residuals}

Recall that $\ee_h \in \VVh \subset \HH_{\GD}(\ccurl,\Omega)$ solves
\eqref{eq_maxwell_discrete} and satisfies~\eqref{eq_maxwell_discrete_II}. In
view of the characterization of the weak solution~\eqref{eq_maxwell_weak_II},
we define the residual $\RR \in (\HH_{\GD}(\ccurl,\Omega))'$ by setting
\begin{equation*}
\langle \RR,\vv \rangle
\eq
(\jj,\vv)_{\Omega} - (\curl \ee_h,\curl \vv)_{\Omega}
=
(\curl (\ee - \ee_h),\curl \vv)_{\Omega}
\quad
\forall \vv \in \HH_{\GD}(\ccurl,\Omega).
\end{equation*}
Taking $\vv = \ee - \ee_h$ and using a duality characterization,
we have the error--residual link
\begin{equation}
\label{eq_res}
\|\curl(\ee - \ee_h)\|_{\Omega}
=
\langle \RR,\ee-\ee_h\rangle^{1/2}
=
\sup_{\substack{
\vv \in \HH_{\GD}(\ccurl,\Omega)\\
\|\curl \vv\|_{\Omega} = 1}
} \langle \RR,\vv \rangle.
\end{equation}
We will also employ local dual norms of the residual $\RR$.
Specifically, for each edge $\edge \in \EE_h$, we set
\begin{equation}
\label{eq_RR_ome}
\|\RR\|_{\star,\edge} \eq
\sup_{\substack{
\vv \in \HH_{\GeD}(\ccurl,\ome)\\
\|\curl \vv\|_{\ome} = 1}
}
\langle \RR,\vv \rangle.
\end{equation}
For each $\edge \in \EE_h$, we will also need an oscillation-free
residual $\RR_h^\edge \in \left (\HH_{\GeD}(\ccurl,\ome)\right )'$
defined using the projected right-hand side introduced in~\eqref{eq_definition_jj_h},
\begin{equation*}
\langle \RR_h^\edge,\vv \rangle
\eq
(\jj_h^\edge,\vv)_{\ome} - (\curl \ee_h,\curl \vv)_{\ome} \quad
\forall \vv \in \HH_{\GeD}(\ccurl,\ome).
\end{equation*}
We also employ the notation
\begin{equation*}
\|\RR^\edge_h\|_{\star,\edge} \eq
\sup_{\substack{
\vv \in \HH_{\GeD}(\ccurl,\ome)\\
\|\curl \vv\|_{\ome} = 1}}
\langle \RR_h^\edge,\vv \rangle
\end{equation*}
for the dual norm of $\RR_h^\edge$. Note that $\RR^\edge_h = \RR_{|\HH_{\GeD}(\ccurl,\ome)}$ whenever the source term $\jj$
is a piecewise $\RT_p(\TT_h)$ polynomial.

\subsection{Data oscillation}

Recalling the definition~\eqref{eq_definition_osc} of $\osc_\edge$, we
have the following comparison:

\begin{lemma}[Data oscillation]
The following holds true:
\begin{subequations}\begin{equation}
\label{eq_data_oscillations_lower_bound}
\|\RR_h^\edge\|_{\star,\edge} \leq \|\RR\|_{\star,\edge} + \osc_\edge
\end{equation}
and
\begin{equation}
\label{eq_data_oscillations_upper_bound}
\|\RR\|_{\star,\edge} \leq \|\RR_h^\edge\|_{\star,\edge} + \osc_\edge.
\end{equation}\end{subequations}
\end{lemma}

\begin{proof}
Let $\vv \in \HH_{\GeD}(\ccurl,\ome)$ with $\|\curl \vv\|_{\ome} = 1$ be fixed.
We have
\begin{equation*}
\langle \RR_h^\edge,\vv \rangle
=
\langle \RR,\vv \rangle - (\jj-\jj_h^\edge,\vv)_{\ome}.
\end{equation*}
We define $q$ as the unique element of $H^1_{\GeD}(\ome)$ such that
\begin{equation*}
(\grad q,\grad w) = (\vv,\grad w) \quad \forall w \in H^1_{\GeD}(\ome),
\end{equation*}
and set $\widetilde \vv \eq \vv - \grad q$.
Since $\div \jj = \div \jj_h^\edge = 0$ and $\jj-\jj_h^\edge \in \HH_{\GeN}(\ddiv,\ome)$,
we have $(\jj-\jj_h^\edge,\grad q)_{\ome} = 0$, and it follows that
\begin{equation*}
\langle \RR_h^\edge,\vv \rangle
=
\langle \RR,\vv \rangle - (\jj-\jj_h^\edge,\widetilde \vv)_{\ome}
\leq
\|\RR\|_{\star,\edge} + \|\jj - \jj_h^\edge\|_{\ome}\|\widetilde \vv\|_{\ome}.
\end{equation*}
Since $\tvv \in \HH_{\GeD}(\ccurl,\ome) \cap
\HH_{\GeN}(\ddiv,\ome)$
with $\div \tvv = 0$ in $\ome$, recalling~\eqref{eq_local_poincare_vectorial}, we have
\begin{equation*}
\|\widetilde \vv\|_{\ome}
\leq
\CPVe h_\ome \|\curl \widetilde \vv\|_{\ome}
=
\CPVe h_\ome \|\curl \vv\|_{\ome}
=
\CPVe h_\ome,
\end{equation*}
and we obtain~\eqref{eq_data_oscillations_lower_bound} by taking the supremum over all $\vv$.
The proof of~\eqref{eq_data_oscillations_upper_bound} follows exactly the same path.
\end{proof}

\subsection{Partition of unity and cut-off estimates}

We now analyze a partition of unity for vector-valued functions that we later employ to localize
the error onto edge patches. Recalling the notation $\ttau_{\edge}$ for the
unit tangent vector orienting $\edge$, we quote the following classical partition
of unity~\cite[Chapter~15]{Ern_Guermond_FEs_I_21}:

\begin{lemma}[Vectorial partition of unity] \label{lem_PU}
Let $\mathbb I$ be the identity matrix in ${\mathbb R}^{3\times 3}$.
The edge basis functions $\ppsi_\edge$ from~\eqref{eq_BF} satisfy
\[
\sum_{\edge \in \EE_h}\ttau_\edge \otimes \ppsi_\edge =
\sum_{\edge \in \EE_h}\ppsi_\edge \otimes \ttau_\edge = {\mathbb I},
\]
where $\otimes$ denotes the outer product, so that we have
\begin{equation}
\label{eq_partition_unity}
\ww = \sum_{\edge \in \EE_h} (\ww \cdot \ttau_\edge )|_\ome \ppsi_\edge
\qquad \forall \ww \in \LL^2(\Omega).
\end{equation}
\end{lemma}

\begin{lemma}[Cut-off stability] \label{lem_c_stab}
For every edge $\edge \in \EE_h$, recalling the space $H^1_\star(\ome)$ defined in~\eqref{eq_Hse} and the constant $\Cconte$ defined in~\eqref{eq_definition_Ccont}, we have
\begin{equation}
\label{eq_estimate_cont}
\|\curl(v \ppsi_\edge)\|_{\ome}
\leq
\Cconte \|\grad v\|_\ome \qquad \forall v \in H^1_\star(\ome).
\end{equation}
Moreover, the upper bound~\eqref{eq_bound_Ccont} on $\Cconte$ holds true.
\end{lemma}

\begin{proof}
Let an edge $\edge \in \EE_h$ and $v \in H^1_\star(\ome)$ be fixed.
Since $\curl (v \ppsi_\edge) = v \curl \ppsi_\edge + \grad v \times \ppsi_\edge$,
we have, using~\eqref{eq_local_poincare},
\begin{align*}
\|\curl (v \ppsi_\edge)\|_{\ome}
&\leq
\|\curl \ppsi_\edge\|_{\infty,\ome}\|v\|_\ome +
\|\ppsi_\edge\|_{\infty,\ome}\|\grad v\|_\ome
\\
&\leq
(\|\curl \ppsi_\edge\|_{\infty,\ome}\CPe h_\ome +
\|\ppsi_\edge\|_{\infty,\ome})\|\grad v\|_\ome
\\
&=
\Cconte \|\grad v\|_\ome.
\end{align*}
This proves~\eqref{eq_estimate_cont}. To prove~\eqref{eq_bound_Ccont},
we remark that in every tetrahedron $K \in \TTe$, we have
(see for instance~\cite[Section~5.5.1]{Monk_FEs_Maxwell_03}, \cite[Chapter~15]{Ern_Guermond_FEs_I_21})
\begin{equation*}
\ppsi_\edge|_K = |\edge| (\lambda_1 \grad \lambda_2 - \lambda_2 \grad \lambda_1),
\quad
(\curl \ppsi_\edge)|_K = 2 |\edge|\grad \lambda_1 \times \grad \lambda_2,
\end{equation*}
where $\lambda_1$ and $\lambda_2$ are the barycentric coordinates of $K$
associated with the two endpoints of $\edge$ such that $\ttau_\edge$
points from the first to the second vertex.
Since $\|\lambda_j\|_{\infty,K} = 1$ and $\|\grad \lambda_j\|_{\infty,K} \leq \rho_K^{-1}$,
we have
\begin{equation*}
\|\ppsi_\edge\|_{\infty,K} \leq \frac{2}{\rho_K}|\edge|, \quad
\|\curl \ppsi_\edge\|_{\infty,K} \leq \frac{2}{\rho_K^2}|\edge|
\end{equation*}
for every $K \in \TTe$.
Recalling the definition~\eqref{eq_patch_not} of $\rho_\edge$,
which implies that $\rho_\edge \leq \rho_K$, as well as the definition $\kappa_\edge$
in~\eqref{eq_regularities}, we conclude that
\begin{align*}
\Cconte = \|\ppsi_\edge\|_{\infty,\ome} + \CPe h_\ome \|\curl \ppsi_\edge\|_{\infty,\ome}
\leq \frac{2 |\edge|}{\rho_\edge}\left(1 + \CPe \frac{h_\ome}{\rho_\edge} \right) = \Cke.
\end{align*}
\end{proof}

\subsection{Upper bound using localized residual dual norms}

We now establish an upper bound on the error using the localized residual dual norms
$\|\RR_h^\edge\|_{\star,\edge}$, in the spirit of~\cite{Blech_Mal_Voh_loc_res_NL_20},
\cite[Chapter~34]{Ern_Guermond_FEs_II_21}, and the references therein.

\begin{proposition}[Upper bound by localized residual dual norms]
Let $\Cconte$ and $\Clift$ be defined in~\eqref{eq_definition_Ccont}
and~\eqref{eq_estimate_lift}, respectively. Then the following holds:
\begin{equation}
\label{eq_upper_bound_residual}
\|\curl(\ee - \ee_h)\|_{\Omega}
\leq
\sqrt{6} \Clift \left (
\sum_{\edge \in \EE_h} \Cconte^2 \left (\|\RR_h^\edge\|_{\star,\edge} + \osc_\edge\right )^2
\right )^{1/2}.
\end{equation}
\end{proposition}

\begin{proof}
We start with~\eqref{eq_res}. Let $\vv \in \HH_{\GD}(\ccurl,\Omega)$
with $\|\curl \vv\|_{\Omega} = 1$ be fixed. Following~\eqref{eq_estimate_lift},
we define $\ww \in \HH^1(\Omega) \cap \HH_{\GD}(\ccurl,\Omega)$
such that $\curl \ww = \curl \vv$ with
\begin{equation}
\label{tmp_estimate_lift}
\|\grad \ww\|_{\Omega} \leq \Clift \|\curl \vv\|_{\Omega}.
\end{equation}
As a consequence of~\eqref{eq_maxwell_weak_II} and~\eqref{eq_maxwell_discrete_II},
the residual $\RR$ is (in particular) orthogonal to
$\NN_0(\TT_h) \cap \HH_{\GD}(\ccurl,\Omega)$. Thus,
by employing the partition of unity~\eqref{eq_partition_unity}, we have
\begin{equation*}
\langle \RR,\vv \rangle
=
\langle \RR,\ww \rangle
=
\sum_{\edge \in \EE_h} \langle \RR, (\ww \cdot \ttau_\edge )|_\ome \ppsi_\edge \rangle
=
\sum_{\edge \in \EE_h}
\langle \RR, (\ww \cdot \ttau_\edge - \overline{w_\edge})|_\ome \ppsi_\edge \rangle,
\end{equation*}
where $\overline{w_\edge} \eq 0$ if $\edge \in \EED_h$ and
\begin{equation*}
\overline{w_\edge} \eq \frac{1}{|\ome|} \int_\ome \ww \cdot \ttau_\edge
\end{equation*}
otherwise. Since
$(\ww \cdot \ttau_\edge - \overline{w_\edge}) \ppsi_\edge \in \HH_{\GeD}(\ccurl,\ome)$
for all $\edge\in \EE_h$, we have from~\eqref{eq_RR_ome}
\begin{equation*}
\langle \RR,\vv \rangle
\leq
\sum_{\edge \in \EE_h} \|\RR\|_{\star,\edge} \|\curl \left (
\left (\ww \cdot \ttau_\edge - \overline{w_\edge}
\right ) \ppsi_\edge \right )\|_{\ome}.
\end{equation*}
We observe that $\ww \cdot \ttau_\edge - \overline{w_\edge} \in H^1_\star(\ome)$
for all $\edge \in \EE_h$ and that
\begin{equation*}
\|\grad(\ww \cdot \ttau_\edge - \overline{w_\edge})\|_{\ome}
=
\|\grad (\ww \cdot \ttau_\edge)\|_{\ome}
\leq
\|\grad\ww\|_{\ome}.
\end{equation*}
As a result, \eqref{eq_estimate_cont} shows that
\begin{equation*}
\langle \RR,\vv \rangle
\leq
\sum_{\edge \in \EE_h} \Cconte \|\RR\|_{\star,\edge}\|\grad \ww\|_{\ome}
\leq
\left (\sum_{\edge \in \EE_h} \Cconte^2 \|\RR\|_{\star,\edge}^2\right )^{1/2}
\left (\sum_{\edge \in \EE_h} \|\grad \ww\|_{\ome}^2\right )^{1/2}.
\end{equation*}
At this point, as each tetrahedron $K \in \TT_h$ has $6$ edges, we have
\begin{equation*}
\sum_{\edge \in \EE_h} \|\grad \ww\|_{\ome}^2
\revision{=}
6 \|\grad \ww\|_{\Omega}^2,
\end{equation*}
and using~\eqref{tmp_estimate_lift}, we infer that
\begin{equation*}
\langle \RR,\vv \rangle\leq
\sqrt{6} \Clift
\left (\sum_{\edge \in \EE_h} \Cconte^2\|\RR\|_{\star,\edge}^2\right )^{1/2}
\|\curl \vv\|_{\Omega}.
\end{equation*}
Then, we conclude with~\eqref{eq_data_oscillations_upper_bound}.
\end{proof}

\subsection{Lower bound using localized residual dual norms}

We now consider the derivation of local lower bounds on the error using
the residual dual norms. We first establish a result for the residual $\RR$.

\begin{lemma}[Local residual]
For every edge $\edge \in \EE_h$, the following holds:
\begin{equation}
\label{eq_minimization}
\|\RR\|_{\star,\edge} =
\min_{\substack{
\hh \in \HH_{\GeN}(\ccurl,\ome)\\
\curl \hh = \jj}}
\|\hh - \curl \ee_h\|_{\ome},
\end{equation}
as well as
\begin{equation}
\label{eq_lower_bound_residual}
\|\RR\|_{\star,\edge} \leq \|\curl(\ee - \ee_h)\|_{\ome}.
\end{equation}
\end{lemma}

\begin{proof}
Let us define $\hh^\star$ as the unique element of $\LL^2(\ome)$ such that
\begin{equation}
\label{tmp_definition_minimizer}
\left \{
\begin{array}{rcll}
\div \hh^\star &=& 0 & \text{ in } \ome,
\\
\curl \hh^\star &=& \jj & \text{ in } \ome,
\\
\hh^\star \cdot \nn_{\omega_\edge} &=& \curl \ee_h \cdot \nn_{\omega_\edge}  & \text{ on } \GeD,
\\
\hh^\star \times \nn_{\omega_\edge} &=& \boldsymbol 0 & \text{ on } \GeN.
\end{array}
\right .
\end{equation}
The existence and uniqueness of $\hh^\star$ follows from
\cite[Proposition 7.4]{Fer_Gil_Maxw_BC_97} after lifting by $\curl \ee_h$.

The second and fourth equations in~\eqref{tmp_definition_minimizer}
imply that $\hh^\star$ belongs to the minimization set of~\eqref{eq_minimization}.
If $\hh' \in \HH_{\GeN}(\ccurl,\ome)$ with $\curl \hh' = \jj$ is another element of the
minimization set, then $\hh^\star - \hh' = \grad q$ for some
$q \in H^1_{\GeN}(\ome)$, and we see that
\begin{align*}
\|\hh' - \curl \ee_h\|_\ome^2
&=
\|\hh^\star - \curl \ee_h - \grad q\|_\ome^2
\\
&=
\|\hh^\star - \curl \ee_h\|_\ome^2 - 2(\hh^\star - \curl \ee_h,\grad q)_\ome + \|\grad q\|_\ome^2
\\
&=
\|\hh^\star - \curl \ee_h\|_\ome^2 + \|\grad q\|_\ome^2
\\
&\geq
\|\hh^\star - \curl \ee_h\|_\ome^2,
\end{align*}
where we used that $\hh^\star$ is divergence-free,
$(\hh^\star - \curl \ee_h) \cdot \nn_{\omega_\edge} = \boldsymbol 0$
on $\GeD$, and $q=0$ on $\GeN$ to infer that $(\hh^\star - \curl \ee_h,\grad q)_\ome=0$.
Hence, $\hh^\star$ is a minimizer of~\eqref{eq_minimization}.

Let $\vv \in \HH_{\GeD}(\ccurl,\ome)$. Since $(\curl \hh^\star,\vv)_\ome
= (\hh^\star,\curl \vv)_\ome$, we have
\begin{align*}
\langle \RR,\vv \rangle
&=
(\jj,\vv)_\ome - (\curl \ee_h,\curl \vv)_\ome
\\
&=
(\curl \hh^\star,\vv)_\ome - (\curl \ee_h,\curl \vv)_\ome
\\
&=
(\hh^\star,\curl \vv)_\ome - (\curl \ee_h,\curl \vv)_\ome
\\
&=
(\pphi,\curl \vv)_\ome,
\end{align*}
where we have set $\pphi \eq \hh^\star - \curl \ee_h$. As above,
$\div \pphi = 0$ in $\ome$ and $\pphi \cdot \nn_{\omega_\edge} = 0$ on $\GeD$.
Therefore, Theorem 8.1 of~\cite{Fer_Gil_Maxw_BC_97}
shows that $\pphi = \curl \oome$ for some $\oome \in \HH_{\GeD}(\ccurl,\ome)$, and
\begin{equation*}
\langle \RR,\vv\rangle = (\curl \oome,\curl \vv)_\ome \quad \forall \vv \in \HH_{\GeD}(\ccurl,\ome).
\end{equation*}
At this point, it is clear that
\begin{equation*}
\|\RR\|_{\star,\edge} =
\sup_{\substack{
\vv \in \HH_{\GeD}(\ccurl,\ome)\\
\|\curl \vv\|_{\ome} = 1}}
(\curl \oome,\curl \vv)_\ome
=
\|\curl \oome\|_\ome
=
\|\hh^\star - \curl \ee_h\|_\ome.
\end{equation*}
Finally, we obtain~\eqref{eq_lower_bound_residual} by observing that
$\widetilde \hh \eq (\curl \ee )|_\ome$ is in the minimization set of~\eqref{eq_minimization}.
\end{proof}
We are now ready to state our results for the oscillation-free residuals $\RR_h^\edge$.
\begin{lemma}[Local oscillation-free residual] \label{lem_res}
For every edge $\edge \in \EE_h$, the following holds:
\begin{equation}
\label{eq_minimization_jh}
\|\RR_h^\edge\|_{\star,\edge} =
\min_{\substack{
\hh \in \HH_{\GeN}(\ccurl,\ome)\\
\curl \hh = \jj_h^\edge}}
\|\hh - \curl \ee_h\|_{\ome},
\end{equation}
as well as
\begin{equation}
\label{eq_lower_bound_residual_osc}
\|\RR_h^\edge\|_{\star,\edge} \leq \|\curl(\ee - \ee_h)\|_{\ome} + \osc_\edge.
\end{equation}
\end{lemma}

\begin{proof}
We establish~\eqref{eq_minimization_jh} by following the same path as for~\eqref{eq_minimization}.
On the other hand, we simply obtain~\eqref{eq_lower_bound_residual_osc} as a consequence
of~\eqref{eq_data_oscillations_lower_bound} and~\eqref{eq_lower_bound_residual}.
\end{proof}

\subsection{Proof of Theorem~\ref{theorem_aposteriori}}

We are now ready to give a proof of Theorem~\ref{theorem_aposteriori}.

On the one hand, the \revision{broken patchwise} equilibration estimator $\eta_\edge$ defined
in~\eqref{eq_definition_estimator} is evaluated from a field
$\hh_h^{\edge,\star} \in \NN_p(\TTe) \cap \HH_{\GeN}(\ccurl,\ome)$
such that $\curl \hh_h^{\edge,\star} = \jj_h^\edge$, and the sequential
sweep~\eqref{eq_definition_estimator_sweep} produces $\hh_h^{\edge,\heartsuit}$
also satisfying these two properties. Since the minimization set in~\eqref{eq_minimization_jh}
is larger, it is clear that
\begin{equation*}
\|\RR_h^\edge\|_{\star,\edge}
\leq
\eta_\edge
\end{equation*}
for both estimators $\eta_\edge$.
Then, \eqref{eq_upper_bound} immediately follows from~\eqref{eq_upper_bound_residual}.

On the other hand, Theorem~\ref{theorem_stability} with the choice $\ch_h \eq (\curl \ee_h)|_\ome$
and the polynomial degree $p$ together with~\eqref{eq_minimization_jh} of Lemma~\ref{lem_res}
implies that
\begin{equation*}
\eta_\edge \leq \Cste \|\RR_h^\edge\|_{\star,\edge}
\end{equation*}
for the estimator~\eqref{eq_definition_estimator}, whereas \revision{the same result for} $\hh_h^{\edge,\heartsuit}$
from~\eqref{eq_definition_estimator_sweep} \revision{follows from Theorem~\ref{thm_sweep} with again $\ch_h \eq (\curl \ee_h)|_\ome$.}
Therefore, \eqref{eq_lower_bound} is a direct consequence of~\eqref{eq_lower_bound_residual_osc}.

\section{Equivalent reformulation and proof of Theorem~\ref{theorem_stability}
($\HH(\ccurl)$ best-approximation in an edge patch)}
\label{sec_proof_stability}

In this section, we consider the minimization problem over an edge patch as posed
in the statement of Theorem~\ref{theorem_stability}\revision{, as well as its sweep
variant of Theorem~\ref{thm_sweep}}, which were central tools to establish
the efficiency of the \revision{broken patchwise equilibrated} error estimators in
Theorem~\ref{theorem_aposteriori}. These minimization problems are similar to the ones considered
in~\cite{Brae_Pill_Sch_p_rob_09, Ern_Voh_p_rob_15, Ern_Voh_p_rob_3D_20} in the framework of
$H^1$ and $\HH(\ddiv)$ spaces. We prove here Theorem~\ref{theorem_stability} via its equivalence
with a stable broken $\HH(\ccurl)$ polynomial extension on an edge patch, as formulated in
Proposition~\ref{prop_stability_patch} below. \revision{By virtue of Remark~\ref{rem_sweep_brok},
this also establishes the validity of Theorem~\ref{thm_sweep}}.

\subsection{Stability of discrete minimization in a tetrahedron}

\subsubsection{Preliminaries}
\label{sec:preliminaries_K}

We first recall some necessary notation from~\cite{Chaum_Ern_Voh_curl_elm_20}.
Consider an arbitrary mesh face $F\in \FF_h$ oriented
by the fixed unit vector normal $\nn_F$.
For all $\ww \in \LL^2(F)$, we define the tangential component of $\ww$ as
\begin{equation} \label{eq:def_pi_tau_F}
\ppi^\ttau_F (\ww) \eq \ww - (\ww \cdot \nn_F) \nn_F.
\end{equation}
Note that the orientation of $\nn_F$ is not important here.
Let $K\in\TT_h$ and let $\FF_K$ be the collection of the faces of $K$.
For all $\vv \in \HH^1(K)$ and all $F\in\FF_K$, the tangential trace of $\vv$ on $F$
is defined (with a slight abuse of notation) as $\ppi^\ttau_F(\vv) \eq \ppi^\ttau_F (\vv|_F)$.

Consider now a nonempty subset $\FF \subseteq \FF_K$. We denote $\Gamma_\FF \subset \partial K$ the
corresponding part of the boundary of $K$. Let $p\ge0$ be the polynomial degree
and recall that $\NN_p(K)$ is the N\'ed\'elec space on the tetrahedron $K$,
see~\eqref{eq_RT_N}. We define the piecewise polynomial space on $\Gamma_\FF$
\begin{equation}
\label{eq_tr_K}
\NN_p^\ttau(\Gamma_\FF) \eq
\left \{
\ww_\FF \in \LL^2(\Gamma_\FF) \; | \;
\exists \vv_p \in \NN_p(K);
\ww_F \eq (\ww_\FF)|_F = \ppi^\ttau_F (\vv_p) \quad \forall F \in \FF
\right \}.
\end{equation}
Note that $\ww_\FF \in \NN_p^\ttau(\Gamma_\FF)$ if and only if
$\ww_F\in \NN_p^\ttau(\Gamma_{\{F\}})$ for all $F\in\FF$ and whenever $|\FF|\ge2$,
for every pair $(F_-,F_+)$ of distinct faces in $\FF$, the following tangential trace
compatibility condition holds true along their common edge $\edge \eq F_+\cap F_-$:
\begin{equation}
\label{eq_edge_compatibility_condition}
(\ww_{F_+})|_\edge \cdot \ttau_\edge =
(\ww_{F_-})|_\edge \cdot \ttau_\edge.
\end{equation}
For all $\ww_\FF \in \NN_p^\ttau(\Gamma_\FF)$, we set
\begin{equation}
\label{eq_scurl_curl_el}
    \scurl_F (\ww_F) \eq (\curl \vv_p)|_F \cdot \nn_F \qquad \forall F \in \FF,
\end{equation}
which is well-defined independently of the choice of $\vv_p$. Note that the orientation of $\nn_F$
is relevant here.

The definition~\eqref{eq:def_pi_tau_F} of the tangential trace cannot be applied to fields
with the minimal regularity $\vv \in \HH(\ccurl,K)$. In what follows, we use the
following notion to prescribe the tangential trace of a field in $\HH(\ccurl,K)$.

\begin{definition}[Tangential trace by integration by parts in a single tetrahedron]
\label{definition_partial_trace}
Let $K$ be a tetrahedron and
$\FF \subseteq \FF_K$ a nonempty (sub)set of its faces.
Given $\rr_\FF \in \NN_p^\ttau(\Gamma_\FF)$ and $\vv \in \HH(\ccurl,K)$, we
employ the notation ``$\vv|^\ttau_\FF = \rr_\FF$'' to say that
\begin{equation*}
(\curl \vv,\pphi)_K - (\vv,\curl \pphi)_K = \sum_{F \in \FF} (\rr_F,\pphi \times \nn_K)_F
\quad
\forall \pphi \in \HH^1_{\ttau,\FF^{\mathrm{c}}}(K),
\end{equation*}
where
\begin{equation*}
\HH_{\ttau,\FF^{\mathrm{c}}}^1(K) \eq
\left \{
\pphi \in \HH^1(K) \; | \; \pphi|_F \times \nn_{\revision{K}} = \boldsymbol 0
\quad
\forall F \in \FF^{\mathrm{c}} \eq \FF_K \setminus \FF
\right \}.
\end{equation*}
Whenever $\vv \in \HH^1(K)$, $\vv|^\ttau_\FF = \rr_\FF$ if and only if
$\ppi^\ttau_F (\vv) = \rr_F$ for all $F \in \FF$.
\end{definition}

\subsubsection{Statement of the stability result in a tetrahedron}

Recall the Raviart--Thomas space $\RT_p(K)$ on the simplex $K$, see~\eqref{eq_RT_N}.
We are now ready to state a key technical tool
from~\cite[Theorem~2]{Chaum_Ern_Voh_curl_elm_20}, based
on~\cite[Theorem~7.2]{Demk_Gop_Sch_ext_II_09}
and~\cite[Proposition~4.2]{Cost_McInt_Bog_Poinc_10}.

\begin{proposition}[Stable $\HH(\ccurl)$ polynomial extension on a tetrahedron]
\label{prop_stability_tetrahedra}
Let $K$ be a tetrahedron and let
$\emptyset\subseteq \FF \subseteq \FF_K$ be a (sub)set of its faces.
Then, for every polynomial degree $p \geq 0$, for all $\rr_K \in \RT_p(K)$ such that
$\div \rr_K = 0$, and if $\FF \ne \emptyset$, for all $\rr_\FF \in \NN_p^\ttau(\Gamma_\FF)$
such that $\rr_K \cdot \nn_{F} = \scurl_F (\rr_F)$ for all $F \in \FF$, the following holds:
\begin{equation}
\label{eq_minimization_element_K}
\min_{\substack{
\vv_p \in \NN_p(K)
\\
\curl \vv_p = \rr_K
\\
\vv_{p}|^\ttau_\FF = \rr_\FF
}}
\|\vv_p\|_{K}
\le C_{\mathrm{st},K}
\min_{\substack{
\vv \in \HH(\ccurl,K)
\\
\curl \vv = \rr_K
\\
\vv|^\ttau_\FF = \rr_\FF
}}
\|\vv\|_{K},
\end{equation}
where the condition on the tangential trace in the minimizing sets is null if $\emptyset=\FF$.
Both minimizers in~\eqref{eq_minimization_element_K} are uniquely defined and the constant
$C_{\mathrm{st},K}$ only depends on the shape-regularity parameter $\kappa_K$ of $K$.
\end{proposition}

\subsection{Piola mappings}
\label{sec_Piola}

This short section reviews some useful properties of Piola mappings used below,
see~\cite[Chapter~9]{Ern_Guermond_FEs_I_21}. Consider two tetrahedra
$\Kin,\Kout \subset \mathbb R^3$ and an invertible affine mapping
$\TTT: \mathbb R^3 \to \mathbb R^3$ such that $\Kout = \TTT(\Kin)$.
Let $\JJJ_{\TTT}$ be the (constant) Jacobian matrix of $\TTT$.
Note that we do not require that $\det \JJJ_{\TTT}$ is positive. The affine mapping $\TTT$
can be identified by specifying the image of each vertex of $\Kin$. We consider the
covariant and contravariant Piola mappings
\begin{equation*}
\ppsi^{\mathrm{c}}_{\TTT}(\vv) = \left (\JJJ_{\TTT}\right )^{-T}
\left (\vv \circ \TTT^{-1} \right ),
\qquad
\ppsi^{\mathrm{d}}_{\TTT}(\vv) = \frac{1}{\det \left (\JJJ_{\TTT}\right )} \JJJ_{\TTT}
(\vv \circ \TTT^{-1})
\end{equation*}
for vector-valued fields
$\vv: \Kin \to \mathbb R^3$. It is well-known that $\ppsi^{\mathrm{c}}_{\TTT}$
maps $\HH(\ccurl,\Kin)$ onto $\HH(\ccurl,\Kout)$ and it maps
$\NN_p(\Kin)$ onto $\NN_p(\Kout)$ for any polynomial degree $p\ge0$.
Similarly, $\ppsi^{\mathrm{d}}_{\TTT}$
maps $\HH(\ddiv,\Kin)$ onto $\HH(\ddiv,\Kout)$ and it maps
$\RT_p(\Kin)$ onto $\RT_p(\Kout)$. Moreover,
the Piola mappings $\ppsi^{\mathrm{c}}_{\TTT}$
and $\ppsi^{\mathrm{d}}_{\TTT}$ commute with the curl operator in the
sense that
\begin{equation}
\label{eq_piola_commute}
\curl \left (\ppsi^{\mathrm{c}}_{\TTT}(\vv)\right )
=
\ppsi^{\mathrm{d}}_{\TTT}\left (\curl \vv\right )
\quad
\forall \vv \in \HH(\ccurl,\Kin).
\end{equation}
In addition, we have
\begin{equation}
\label{eq_piola_adjoint}
(\ppsi^{\mathrm{c}}_{\TTT}(\vv_{\rm in}),\vv_{\rm out})_{\Kout}
=
\sign (\det \JJJ_{\TTT}) (\vv_{\rm in},(\ppsi^{\mathrm{d}}_{\TTT})^{-1}(\vv_{\rm out}))_{\Kin},
\end{equation}
for all $\vv_{\rm in} \in \HH(\ccurl,\Kin)$ and $\vv_{\rm out} \in \HH(\ccurl,\Kout)$.
We also have $\|\ppsi_{\TTT}^{\mathrm{c}}(\vv)\|_{\Kout} \leq
\frac{h_{\Kin}}{\rho_{\Kout}} \|\vv\|_{\Kin}$ for all $\vv \in \LL^2(\Kin)$, so that
whenever $\Kin,\Kout$ belong to the same edge patch $\TTe$, we have
\begin{equation}
\label{eq_stab_piola_L2}
\|\ppsi_{\TTT}^{\mathrm{c}}(\vv)\|_{\Kout}
\leq
C \|\vv\|_{\Kin} \quad \forall \vv \in \LL^2(\Kin),
\end{equation}
for a constant $C$ only depending on the shape-regularity $\kappa_\edge$ of the patch $\TTe$
defined in~\eqref{eq_regularities}.

\subsection{Stability of discrete minimization in an edge patch}

\subsubsection{Preliminaries}

In this section, we consider an edge patch $\TTe$ associated with a mesh edge $\edge\in\EE_h$
consisting of tetrahedral elements $K$ sharing the edge $\edge$, cf. Figure~\ref{fig_patch}.
We denote by $n \eq |\TTe|$ the number of tetrahedra in the patch, by $\FFe$ the set of all faces of
the patch, by $\FFei \subset \FFe$ the set of ``internal'' faces, i.e., those being shared by two
different tetrahedra from the patch, and finally, by
$\FFee \eq \FFe \setminus \FFei$ the set of ``external'' faces.
The patch is either of ``interior'' type, corresponding to an edge in the interior of the
domain $\Omega$, in which case there is a full loop around $\edge$, see Figure~\ref{fig_patch},
left, or of ``boundary'' type, corresponding to an edge on the boundary of the domain $\Omega$,
in which case there is no full loop around $\edge$, see Figure~\ref{fig_patch}, right, and
Figure~\ref{figure_numbering_patch}.
We further distinguish three types of patches of boundary type depending on the status of
the two boundary faces sharing the associated boundary edge: the patch is of Dirichlet boundary
type if both faces lie in $\overline \GD$, of ``mixed boundary'' type if one face lies
in $\overline \GD$ an the other in $\overline \GN$, and of Neumann boundary type
if both faces lie in $\overline \GN$. Note that for an interior patch, $|\FFei| = n$,
whereas $|\FFei| = n-1$ for a boundary patch. The open domain associated with $\TTe$ is denoted
by $\ome$, and $\nn_{\ome}$ stands for the unit normal vector to $\partial \ome$ pointing
outward $\ome$.

\begin{figure}[htb]
\centerline{\includegraphics[height=0.35\textwidth]{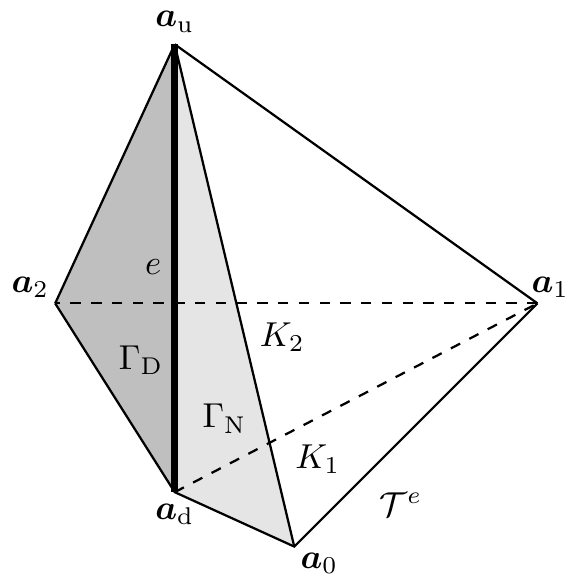}
\qquad \qquad
\includegraphics[height=0.35\textwidth]{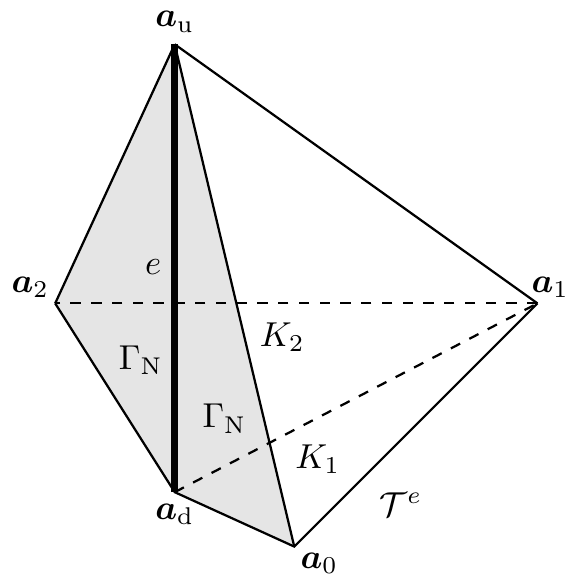}}
\caption{Mixed (left) and Neumann (right) boundary patch $\TTe$}
\label{figure_numbering_patch}
\end{figure}

We denote by $\adown$ and $\aup$ the two vertices of the edge $\edge$.
The remaining vertices are numbered consecutively in one sense of rotation around the edge $\edge$
(this sense is only specific for the ``mixed boundary'' patches) and denoted by
$\aaa_{0},\aaa_{1},\dots,\aaa_{n}$, with $\aaa_{0} = \aaa_{n}$ if the patch is interior.
Then $\TTe = \bigcup_{j\in\{1:n\}} K_j$ for $K_j \eq \conv(\aaa_{j-1},\aaa_{j},\adown,\aup)$;
we also denote $K_0 \eq K_n$ and $K_{n+1} \eq K_1$. For all $j\in\{0:n\}$, we define
$F_{j} \eq \conv(\aaa_{j},\adown,\aup)$, and for all $j\in\{1:n\}$, we let
$\Fdown_j \eq \conv(\aaa_{j-1},\aaa_{j},\adown)$ and $\Fup_j \eq \conv(\aaa_{j-1},\aaa_{j},\aup)$.
Then $\FF_{K_j} = \{F_{j-1},F_{j},\Fdown_j,\Fup_j\}$,
and $F_0 = F_n$ if the patch is interior. We observe that, respectively for interior
and boundary patches, $\FFei = \bigcup_{j\in\{0:n-1\}} \{F_{j}\}$,
$\FFee = \bigcup_{j\in\{1:n\}} \{\Fdown_j,\Fup_j\}$ and
$\FFei = \bigcup_{j\in\{1:n-1\}} \{F_{j}\}$,
$\FFee = \bigcup_{j\in\{1:n\}} \{\Fdown_j,\Fup_j\} \cup \{F_0,F_n\}$.
Finally, if $F_{j} \in \FFei$ is an internal face, we define its normal vector
by $\nn_{F_{j}} \eq \nn_{K_{j+1}} = -\nn_{K_j}$, whereas for any external face $F \in \FFee$, we
define its normal vector to coincide with the normal vector pointing outward the patch,
$\nn_F \eq \nn_{\ome}$.

We now extend the notions of Section~\ref{sec:preliminaries_K} to the edge patch $\TTe$.
Consider the following broken Sobolev spaces:
\begin{align*}
\HH(\ccurl,\TTe)
&\eq
\left \{
\vv \in \LL^2(\ome) \; | \; \vv|_K \in \HH(\ccurl,K)
\; \forall K \in \TTe
\right \},
\\
\HH^1(\TTe)
&\eq
\left \{
\vv \in \LL^2(\ome) \; | \; \vv|_K \in \HH^1(K)
\; \forall K \in \TTe
\right \},
\end{align*}
as well as the broken N\'ed\'elec space  $\NN_{p}(\TTe)$.
For all $\vv \in \HH^1(\TTe)$, we employ the notation $\jump{\vv}_F \in \LL^2(F)$
for the ``(strong) jump'' of $\vv$ across any face $F\in\FF^\edge$. Specifically, for an internal face
$F_{j} \in \FFei$, we set $\jump{\vv}_{F_{j}} \eq (\vv|_{K_{j+1}})|_{F_{j}} - (\vv|_{K_j})|_{F_{j}}$,
whereas for an external face $F \in \FFee$, we set $\jump{\vv}_F \eq \vv|_F$. Note in particular that
piecewise polynomial functions from $\NN_{p}(\TTe)$ belong to $\HH^1(\TTe)$, so that their strong
jumps are well-defined.

To define a notion of a ``weak tangential jump'' for functions of $\HH(\ccurl,\TTe)$,
for which a strong (pointwise) definition cannot apply, some preparation is necessary.
Let $\FF$ be a subset of the faces of an edge patch $\TTe$ containing the internal faces,
i.e. $\FFei \subseteq \FF \subseteq \FFe$, and denote by $\Gamma_\FF$ the corresponding
open set. The set $\FF$ \revision{represents the set of faces appearing in the minimization. It}
depends on the type of edge patch and is reported in Table~\ref{Tab:type_of_calF}. In extension
of~\eqref{eq_tr_K}, we define the piecewise polynomial space on $\Gamma_\FF$
\begin{equation}
\label{eq_tr_Te}
\NN_p^\ttau(\Gamma_\FF) \eq
\left \{
\ww_\FF \in \LL^2(\Gamma_\FF) \; | \;
\exists \vv_p \in \NN_{p}(\TTe);
\ww_F \eq (\ww_\FF)|_F = \ppi^\ttau_F (\jump{\vv_p}_F) \quad \forall F \in \FF
\right \}.
\end{equation}
In extension of~\eqref{eq_scurl_curl_el}, for all $\ww_\FF \in \NN_p^\ttau(\Gamma_\FF)$, we set
\begin{equation}
\label{eq_scurl_curl_patch}
    \scurl_F (\ww_F) \eq \jump{\curl \vv_p}_F \cdot \nn_F \qquad \forall F \in \FF.
\end{equation}
Then we can extend Definition~\ref{definition_partial_trace} to prescribe weak tangential jumps
of functions in $\HH(\ccurl,\TTe)$ as follows:

\begin{table}[tb]
\begin{center}\begin{tabular}{|l|l|l|}
\hline
patch type&\hfil $\FFei$&\hfil $\FF$\\
\hline
interior&$\{F_1,\ldots,F_n\}$&$\FFei = \{F_1,\ldots,F_n\}$\\
Dirichlet boundary&$\{F_1,\ldots,F_{n-1}\}$&$\FFei = \{F_1,\ldots,F_{n-1}\}$\\
mixed boundary&$\{F_1,\ldots,F_{n-1}\}$&$\{F_0\}\cup \FFei = \{F_0,F_1,\ldots,F_{n-1}\}$\\
Neumann boundary&$\{F_1,\ldots,F_{n-1}\}$&$\{F_0\}\cup\FFei\cup\{F_n\}=\{F_0,F_1,\ldots,F_{n-1},F_n\}$\\
\hline
\end{tabular}\end{center}
\caption{The set of internal faces $\FFei$ and the set $\FF$ used for the minimization problems on the edge patch for the four patch types.}
\label{Tab:type_of_calF}%
\end{table}

\begin{definition}[Tangential jumps by integration by parts in an edge patch]
\label{definition_jumps}
Given $\rr_{\FF} \in \NN_p^\ttau(\Gamma_\FF)$ and $\vv \in \HH(\ccurl,\TTe)$,
we employ the notation ``$\jump{\vv}^\ttau_\FF = \rr_\FF$'' to say that
\begin{equation}\label{eq_prescription_tang_jump_e}
\sum_{K \in \TTe}
\left \{
(\curl \vv,\pphi)_K - (\vv,\curl \pphi)_K
\right \}
=
\sum_{F \in \FF} (\rr_F,\pphi \times \nn_F)_F
\quad
\forall \pphi \in \HH^1_{\ttau,\FF^{\mathrm{c}}}(\TTe),
\end{equation}
where
\begin{equation} \label{eq_Htpatch}
\HH^1_{\ttau,\FF^{\mathrm{c}}}(\TTe)
\eq
\left \{
\pphi \in \HH^1(\TTe) \; | \;
\jump{\pphi}_F \times \nn_F = \boldsymbol 0 \quad
\forall F \in \FFei \cup (\FFee \setminus \FF)
\right \}.
\end{equation}
Whenever $\vv \in \HH^1(\TTe)$, $\jump{\vv}^\ttau_\FF = \rr_\FF$ if and only if
$\ppi^\ttau_F(\jump{\vv}_F) = \rr_F$ for all $F \in \FF$. Note that
$\pphi\times \nn_F$ in~\eqref{eq_prescription_tang_jump_e} is uniquely defined for
all $\pphi \in \HH^1_{\ttau,\FF^{\mathrm{c}}}(\TTe)$.
\end{definition}

\subsubsection{Statement of the stability result in an edge patch}

Henceforth, if $\rr_\TT \in \RT_p(\TTe)$ is an elementwise Raviart--Thomas function,
we will employ the notation $\rr_K \eq \rr_\TT|_K$ for all $K \in \TTe$. In addition,
if $\vv_{\TT} \in \NN_{p}(\TTe)$ is an elementwise N\'ed\'elec function, the notations
$\div \rr_\TT$ and $\curl \vv_{\TT}$ will be understood elementwise.

\begin{definition}[Compatible data]
\label{definition_compatible_data}
Let $\rr_\TT \in \RT_p(\TTe)$ and $\rr_\FF \in \NN_p^\ttau(\Gamma_\FF)$.
We say that the data $\rr_\TT$ and $\rr_\FF$ are compatible if
\bse\begin{align}
\div \rr_\TT & = 0, \label{eq_comp_a} \\
\jump{\rr_\TT}_{F} \cdot \nn_F & = \scurl_F (\rr_F) \quad \forall F \in \FF, \label{eq_comp_b}
\end{align}
and with the following additional condition whenever the patch is either of interior or Neumann
boundary type:
\begin{alignat}{2}
\sum_{j\in\{1:n\}} \rr_{F_j}|_\edge \cdot \ttau_\edge
&=
0
&\qquad&
\textrm{(interior type)},
\label{eq_comp_c} \\
\sum_{j\in\{0:n-1\}} \rr_{F_j}|_\edge \cdot \ttau_\edge
&=
\rr_{F_n}|_\edge \cdot \ttau_\edge
&\qquad&
\textrm{(Neumann boundary type)}.
\label{eq_comp_d}
\end{alignat}\ese
\end{definition}

\begin{definition}[Broken patch spaces]\label{def_spaces}
Let $\rr_\TT \in \RT_p(\TTe)$ and $\rr_\FF \in \NN_p^\ttau(\Gamma_\FF)$
be compatible data as per Definition~\ref{definition_compatible_data}. We define
\bse\begin{align}
\boldsymbol V(\TTe) & \eq
\left \{
\vv \in \HH(\ccurl,\TTe) \; \left |
\begin{array}{rl}
\curl \vv &= \rr_\TT
\\
\jump{\vv}^\ttau_\FF &= \rr_\FF
\end{array}
\right .
\right \}, \label{eq_VTe}\\
\boldsymbol V_p(\TTe) & \eq \boldsymbol V(\TTe) \cap \NN_p(\TTe). \label{eq_VqTe}
\end{align}\ese
\end{definition}
We will show in Lemma~\ref{lem_EU} below that the space $\boldsymbol V_p(\TTe)$
(and therefore also $\boldsymbol V(\TTe)$) is nonempty.
We are now ready to present our central result of independent interest.
To facilitate the reading, the proof is postponed to Section~\ref{sec:proof_stability_patch}.

\begin{proposition}[Stable broken $\HH(\ccurl)$ polynomial extension in an edge patch]
\label{prop_stability_patch}
Let an edge $\edge\in\EE_h$ and the associated edge patch $\TTe$ with subdomain $\ome$
be fixed. Let the set of faces $\FF$ be specified in Table~\ref{Tab:type_of_calF}.
Then, for every polynomial degree $p \geq 0$, all $\rr_\TT \in \RT_p(\TTe)$, and
all $\rr_\FF \in \NN_p^\ttau(\Gamma_\FF)$ compatible as per
Definition~\ref{definition_compatible_data}, the following holds:
\begin{equation} \label{eq_stab_shift}
\min_{\vv_p\in {\boldsymbol V}_p(\TTe)} \|\vv_p\|_{\ome}
= \min_{\substack{
\vv_p \in \NN_p(\TTe)
\\
\curl \vv_p = \rr_\TT
\\
\jump{\vv_p}^\ttau_\FF = \rr_\FF
}}
\|\vv_p\|_{\ome}
\le C_{\mathrm{st},\edge}
\min_{\substack{
\vv \in \HH(\ccurl,\TTe)
\\
\curl \vv = \rr_\TT
\\
\jump{\vv}^\ttau_\FF = \rr_\FF
}}
\|\vv\|_{\ome} =
C_{\mathrm{st},\edge} \min_{\vv\in {\boldsymbol V}(\TTe)} \|\vv\|_{\ome}.
\end{equation}
Here, all the minimizers are uniquely defined and the constant $C_{\mathrm{st},\edge}$
only depends on the shape-regularity parameter $\kappa_\edge$ of the patch $\TTe$
defined in~\eqref{eq_regularities}.
\end{proposition}

\begin{remark}[Converse inequality in Proposition~\ref{prop_stability_patch}]
\label{rem_conv}
Note that the converse to the inequality~\eqref{eq_stab_shift} holds trivially with constant one.
\end{remark}

\subsection{Equivalence of Theorem~\ref{theorem_stability} with Proposition~\ref{prop_stability_patch}}

$ $
\revision{We have the following important link, establishing Theorem~\ref{theorem_stability},
including the existence and uniqueness of the minimizers.

\begin{lemma}[Equivalence of Theorem~\ref{theorem_stability} with
Proposition~\ref{prop_stability_patch}]
Theorem~\ref{theorem_stability} holds if and only if Proposition~\ref{prop_stability_patch} holds.
{\color{black} More precisely, let $\hh_p^\star \in \NN_p(\TTe) \cap \HH_{\GeN}(\ccurl,\ome)$
and $\hh^\star \in \HH_{\GeN}(\ccurl,\ome)$ by any solutions to the minimization problems of
Theorem~\ref{theorem_stability} for the data
$\jj_h^{\edge} \in \RT_p(\TTe) \cap \HH_{\GeN}(\ddiv,\ome)$ with $\div \jj_h^{\edge} = 0$ and
$\ch_h \in \NN_p(\TTe)$. Let $\vv^\star_p \in {\boldsymbol V}_p(\TTe)$ and
$\vv^\star \in \boldsymbol V(\TTe)$ be any minimizers of
Proposition~\ref{prop_stability_patch} for the data
\begin{equation} \label{eq_data_eq}
\rr_\TT \eq \jj_h^{\edge}-\curl \ch_h, \qquad
\rr_F \eq - \ppi^\ttau_F \left (\jump{\ch_h}_F\right ) \quad \forall F \in \FF,
\end{equation}
where $\FF$ is specified in Table~\ref{Tab:type_of_calF}.} Then
{\color{black}
\begin{equation} \label{eq_eq}
\hh_p^\star - \ch_h =  \vv^\star_p, \quad
\hh^\star - \ch_h = \vv^\star.
\end{equation}
In the converse direction, for given data $\rr_\TT$ and $\rr_\FF$} in Proposition~\ref{prop_stability_patch},
{\color{black}
compatible as per Definition~\ref{definition_compatible_data}, taking any $\ch_h \in \NN_p(\TTe)$
such that $- \ppi^\ttau_F \left (\jump{\ch_h}_F\right ) = \rr_F$ for all $F \in \FF$ and
$\jj_h^{\edge} \eq \rr_\TT + \curl \ch_h$} gives minimizers of Theorem~\ref{theorem_stability} such that{\color{black} ~\eqref{eq_eq} holds true}.
\end{lemma}
}

\begin{proof} \revision{The proof follows} via a shift by the datum $\ch_h$.
In order to show~\eqref{eq_eq} in the forward direction \revision{(the converse
direction is actually easier)}, we merely need to show that
$\rr_\TT$ and $\rr_\FF$ prescribed by~\eqref{eq_data_eq} are compatible data
as per Definition~\ref{definition_compatible_data}.
Indeed, to start with, since $\jj_h^{\edge}, \curl \ch_h \in \RT_p(\TTe)$,
we have $\rr_\TT \in \RT_p(\TTe)$. In addition, since $\div \jj_h^{\edge} = 0$
from~\eqref{eq_definition_jj_h},
\begin{equation*}
\div \rr_\TT = \div \jj_h^{\edge} - \div (\curl \ch_h) = 0,
\end{equation*}
which is~\eqref{eq_comp_a}.
Then, for all $j\in\{1:n\}$ if the patch is of interior type and for all $j\in\{1:n-1\}$
if the patch is of boundary type, we have
\begin{equation*}
\rr_{F_{j}}
=
\ppi^\ttau_{F_j} (\ch_h|_{K_j})
-
\ppi^\ttau_{F_j} (\ch_h|_{K_{j+1}}),
\end{equation*}
and therefore, recalling the definition~\eqref{eq_scurl_curl_el} of the surface curl,
we infer that
\begin{align*}
\scurl_{F_{j}} (\rr_{F_{j}})
&=
\scurl_{F_{j}} (\ppi^\ttau_{F_j} (\ch_h|_{K_j}))
-
\scurl_{F_{j}} (\ppi^\ttau_{F_j} (\ch_h|_{K_{j+1}}))
\\
&=
(\curl \ch_h)|_{K_{j}}|_{F_j} \cdot \nn_{F_{j}}
-
(\curl \ch_h)|_{K_{j+1}}|_{F_j} \cdot \nn_{F_{j}}
\\
&=
- \jump{\curl \ch_h}_{F_{j}} \cdot \nn_{F_{j}}.
\end{align*}
On the other hand, since $\jj_h^{\edge} \in \HH(\ddiv,\ome)$,
we have $\jump{\jj_h^{\edge}}_{F_{j}} \cdot \nn_{F_{j}} = 0$,
and therefore
\begin{align*}
\jump{\rr_\TT}_{F_{j}} \cdot \nn_{F_{j}}
&=
- \jump{\curl \ch_h}_{F_{j}} \cdot \nn_{F_{j}}
=
\scurl_{F_{j}} (\rr_{F_{j}}).
\end{align*}
Since a similar reasoning applies on the face $F_0$
if the patch is of Neumann or mixed boundary type and on the face $F_n$
if the patch is of Neumann boundary type, \eqref{eq_comp_b} is established.

It remains to show that $\rr_\FF$ satisfies the edge compatibility
condition~\eqref{eq_comp_c} or~\eqref{eq_comp_d} if the patch is interior or
Neumann boundary type, respectively. Let us treat the first case
(the other case is treated similarly). Owing to the convention $K_{n+1}=K_1$,
we infer that
\begin{equation*}
\sum_{j\in\{1:n\}} \rr_{F_{j}}|_\edge \cdot \ttau_\edge
=
\sum_{j\in\{1:n\}} (\ch_h|_{K_j} - \ch_h|_{K_{j+1}})|_\edge \cdot \ttau_\edge
=
0,
\end{equation*}
which establishes~\eqref{eq_comp_c}. We have thus shown that $\rr_\TT$ and $\rr_\FF$
are compatible data as per Definition~\ref{definition_compatible_data}.
\end{proof}

\subsection{Proof of Proposition~\ref{prop_stability_patch}}
\label{sec:proof_stability_patch}

The proof of Proposition~\ref{prop_stability_patch} is performed in two steps.
First we prove that ${\boldsymbol V}_p(\TTe)$ is nonempty by
providing a generic elementwise construction of a field in ${\boldsymbol V}_p(\TTe)$;
this in particular implies the existence and uniqueness of all minimizers in~\eqref{eq_stab_shift}.
Then we prove the inequality~\eqref{eq_stab_shift} by using
one such field $\xxi_p^\star \in \boldsymbol V_p(\TTe)$.
Throughout this section, if $A,B \geq 0$ are two real numbers, we employ
the notation $A \lesssim B$ to say that there exists a constant $C$ that only depends on the
shape-regularity parameter $\kappa_\edge$ of the patch $\TTe$ defined in~\eqref{eq_regularities},
such that $A \leq CB$. We note that in particular we have $n = |\TTe| \lesssim 1$
owing to the shape-regularity of the mesh $\TT_h$.

\subsubsection{Generic elementwise construction of fields in ${\boldsymbol V}_p(\TTe)$}

The generic construction of fields in ${\boldsymbol V}_p(\TTe)$ is based
on a loop over the mesh elements composing the edge patch $\TTe$. This loop is
enumerated by means of an index $j\in\{1:n\}$.

\begin{definition}[Element spaces]\label{def_spaces_elm}
For each $j\in\{1:n\}$, let $\emptyset\subseteq \FF_j \subseteq \FF_{K_j}$ be a
(sub)set of the faces of $K_j$. Let $\rr_{K_j} \in \RT_p(K_j)$ with $\div \rr_{K_j} = 0$,
and if $\emptyset\ne \FF_j$, let $\trr_{\FF_j}^j \in \NN_p^\ttau(\Gamma_{\FF_j})$
in the sense of~\eqref{eq_tr_K} be given data. We define
\bse\begin{align}
\boldsymbol V(K_j) & \eq \left \{
\vv \in \HH(\ccurl,K_j) \; \left | \;
\begin{array}{l}
\curl \vv = \rr_{K_j}
\\
\vv|^\ttau_{\FF_j} = \trr^j_{\FF_j},
\end{array}
\right .
\right \}, \label{eq_VKj} \\
\boldsymbol V_p(K_j) & \eq \boldsymbol V(K_j) \cap \NN_p(K_j). \label{eq_VqKj}
\end{align}\ese
\end{definition}

In what follows, we are only concerned with the cases where $\FF_j$ is either empty
or composed of one or two faces of $K_j$. In this situation, the subspace ${\boldsymbol V}_p(K_j)$
is nonempty if and only if
\bse \label{eq:cond_comp}
\begin{alignat}{2}
&\scurl_{F} (\trr_{F}^{j}) = \rr_{K_j} \cdot \nn_F&\quad&\forall F\in\FF_j,\label{eq:cond_comp1} \\
&\trr_{F_+}^{j}|_{\edge} \cdot \ttau_\edge = \trr_{F_-}^{j}|_{\edge} \cdot \ttau_\edge&\quad&\text{if $\FF_j=\{F_+,F_-\}$ with $\edge=F_+\cap F_-$}, \label{eq:cond_comp2}
\end{alignat}
\ese
where $\nn_F$ is the unit normal orienting $F$ used in the definition of the surface curl
(see~\eqref{eq_scurl_curl_el}). The second condition~\eqref{eq:cond_comp2} is relevant only if
$|\FF_j|=2$.


\begin{lemma}[Generic elementwise construction]\label{lem_EU}
Let $\edge\in\EE_h$, let $\TTe$ be the edge patch associated with $\edge$, and let
the set of faces $\FF$ be specified in Table~\ref{Tab:type_of_calF}.
Let $\rr_\TT \in \RT_p(\TTe)$ and $\rr_\FF \in \NN_p^\ttau(\Gamma_\FF)$ be compatible data
as per Definition~\ref{definition_compatible_data}. Define
$\rr_{K_j}\eq \rr_{\TT}|_{K_j}$ for all $j\in\{1:n\}$. Then,
the following inductive procedure yields a sequence of nonempty spaces
$({\boldsymbol V}_p(K_j))_{j\in\{1:n\}}$ in the sense of Definition~\ref{def_spaces_elm},
as well as a sequence of fields $(\xxi_p^j)_{j\in\{1:n\}}$
such that $\xxi_p^j\in {\boldsymbol V}_p(K_j)$ for all $j\in\{1:n\}$. Moreover, the field
$\xxi_p$ prescribed by $\xxi_p|_{K_j} \eq \xxi_p^j$ for all $j\in\{1:n\}$ belongs to the space
$\boldsymbol V_p(\TTe)$ of Definition~\ref{def_spaces}:
\\
\bse
{\bf 1)} First element ($j=1$): Set $\FF_1 \eq \emptyset$ if the patch is of interior or
Dirichlet boundary type and set $\FF_1 \eq \{F_0\}$ if the patch is of Neumann or mixed
boundary type together with
\begin{equation}
\trr^1_{F_0} \eq  \rr_{F_0}.\label{eq_F_0}
\end{equation}
Define the space ${\boldsymbol V}_p(K_1)$ according to~\eqref{eq_VqKj} and pick any
$\xxi_p^1\in {\boldsymbol V}_p(K_1)$. \\
{\bf 2)} Middle elements ($j\in\{2:n-1\}$): Set $\FF_j\eq \{F_{j-1}\}$ together with
\begin{equation}
\trr^j_{F_{j-1}} \eq \ppi^\ttau_{F_{j-1}} (\xxi^{j-1}_p) + \rr_{F_{j-1}}, \label{eq_F_j}
\end{equation}
with $\xxi^{j-1}_p$ obtained in the previous step of the procedure.
Define the space ${\boldsymbol V}_p(K_j)$ according to~\eqref{eq_VqKj} and pick any
$\xxi_p^j\in {\boldsymbol V}_p(K_j)$. \\
{\bf 3)} Last element ($j=n$): Set $\FF_{n} \eq \{F_{n-1}\}$ if the patch is of Dirichlet or mixed boundary type and set $\FF_n \eq \{ F_{n-1}, F_n \}$ if the patch is of interior or Neumann boundary type and define $\trr^n_{\FF_n}$ as follows:
For the four cases of the patch,
\begin{equation}
\trr^n_{F_{n-1}} \eq \ppi^\ttau_{F_{n-1}} (\xxi^{n-1}_p) + \rr_{F_{n-1}}, \label{eq_F_nn}
\end{equation}
with $\xxi^{n-1}_p$ obtained in the previous step of the procedure, and in the two cases where $\FF_n$ also contains $F_n$:
\begin{alignat}{2}
\trr_{F_n}^n & \eq \ppi^\ttau_{F_n} (\xxi_p^1) - \rr_{F_n} &\qquad&\text{interior type}, \label{eq_F_int}\\
\trr_{F_n}^n & \eq \rr_{F_n} &\qquad&\text{Neumann boundary type}. \label{eq_F_n}
\end{alignat}
Define the space ${\boldsymbol V}_p(K_n)$ according to~\eqref{eq_VqKj} and pick any
$\xxi_p^n\in {\boldsymbol V}_p(K_n)$. \\
\ese
\end{lemma}

\begin{proof}
We first show that $\xxi^j_p$ is well-defined in ${\boldsymbol V}_p(K_j)$ for all
$j\in\{1:n\}$. We do so by verifying that $\trr^{j}_{\FF_{j}} \in \NN_p^\ttau(\Gamma_{\FF_{j}})$
(recall~\eqref{eq_tr_K}) and that the conditions~\eqref{eq:cond_comp} hold true for all
$j\in\{1:n\}$. Then, we show that $\xxi_p \in \boldsymbol V_p(\TTe)$.

{\bf (1)} First element ($j=1$). If the patch is of interior or Dirichlet boundary type,
there is nothing to verify since $\FF_1$ is empty. If the patch is of Neumann or mixed
boundary type, $\FF_1 = \{F_0\}$ and we need to verify that
$\trr^1_{F_0} \in \NN_p^\ttau(\Gamma_{\{F_0\}})$ and that
$\scurl_{F_0} (\trr^{1}_{F_0})=\rr_{K_1} \cdot \nn_{F_0}$, see~\eqref{eq:cond_comp1}.
Since $\trr^1_{\FF_1}=\rr_{F_0}\in \NN_p^\ttau(\Gamma_{\{F_0\}})$ by assumption, the first
requirement is met. The second one follows from
$\rr_{K_1} \cdot \nn_{F_0} = \jump{\rr_\TT}_{F_0} \cdot \nn_{F_0} =
\scurl_{F_0}(\rr_{F_0}) = \scurl_{F_0} (\trr^{1}_{F_0})$ owing to~\eqref{eq_comp_b}.

{\bf (2)} Middle elements ($j\in\{2:n-1\}$). Since $\FF_j = \{F_{j-1}\}$, we need to show
that $\trr^j_{F_{j-1}} \in \NN_p^\ttau(\Gamma_{\{F_{j-1}\}})$ and that
$\scurl_{F_{j-1}} (\trr^{j}_{F_{j-1}})=\rr_{K_j} \cdot \nn_{F_{j-1}}$.
The first requirement follows from the definition~\eqref{eq_F_j} of $\trr^j_{F_{j-1}}$.
To verify the second requirement, we recall the
definition~\eqref{eq_scurl_curl_el} of the surface curl
and use the curl constraint from~\eqref{eq_VKj} to infer that
\begin{align*}
\scurl_{F_{j-1}} (\trr_{F_{j-1}}^{j})
&=
\scurl_{F_{j-1}} (\ppi^\ttau_{F_{j-1}} (\xxi_p^{j-1})) + \scurl_{F_{j-1}} (\rr_{F_{j-1}})
\\
&=
\left (\curl \xxi_p^{j-1}\right )|_{F_{j-1}} \cdot \nn_{F_{j-1}} + \scurl_{F_{j-1}} (\rr_{F_{j-1}})
\\
&=
\rr_{K_{j-1}} \cdot \nn_{F_{j-1}} + \scurl_{F_{j-1}} (\rr_{F_{j-1}}).
\end{align*}
By virtue of assumption~\eqref{eq_comp_b}, it follows that
\begin{align*}
\rr_{K_{j}} \cdot \nn_{F_{j-1}} - \scurl_{F_{j-1}} (\trr_{F_{j-1}}^{j})
&=
\rr_{K_{j}} \cdot \nn_{F_{j-1}} - \rr_{K_{j-1}} \cdot \nn_{F_{j-1}} - \scurl_{F_{j-1}} (\rr_{F_{j-1}})
\\
&=
\jump{\rr_\TT}_{F_{j-1}} \cdot \nn_{F_{j-1}} - \scurl_{F_{j-1}} (\rr_{F_{j-1}}) = 0.
\end{align*}

{\bf (3)} Last element ($j=n$). We distinguish two cases.

{\bf (3a)} Patch of Dirichlet or mixed boundary type. In this case, $\FF_n = \{F_{n-1}\}$
and the reasoning is identical to the case of a middle element.

{\bf (3b)} Patch of interior or Neumann boundary type. In this case, $\FF_n = \{F_{n-1},F_n\}$.
First, the prescriptions~\eqref{eq_F_nn}--\eqref{eq_F_int}--\eqref{eq_F_n} imply
that $\trr^n_{\FF_n} \in \NN^\ttau_p(\Gamma_{\FF_n})$ in the sense of~\eqref{eq_tr_K}.
It remains to show~\eqref{eq:cond_comp1}, i.e.
\begin{equation} \label{eq:cond_comp_n1}
\scurl_{F_{n-1}}(\trr^n_{F_{n-1}})=\rr_{K_n}\cdot\nn_{F_{n-1}}, \qquad
\scurl_{F_{n}}(\trr^n_{F_{n}})=\rr_{K_n}\cdot\nn_{F_{n}},
\end{equation}
and, since $\FF_n$ is composed of two faces, we also need to show the edge compatibility
condition~\eqref{eq:cond_comp2}, i.e.
\begin{equation} \label{eq:cond_comp_n2}
\trr^n_{F_{n-1}}|_\edge \cdot \ttau_\edge = \trr^n_{F_{n}}|_\edge \cdot \ttau_\edge.
\end{equation}
The proof of the first identity in~\eqref{eq:cond_comp_n1} is as above, so we now detail
the proof of the second identity in~\eqref{eq:cond_comp_n1} and the proof
of~\eqref{eq:cond_comp_n2}.

{\bf (3b-I)} Let us consider the case of a patch of interior type. To prove the
second identity in~\eqref{eq:cond_comp_n1}, we use definition~\eqref{eq_scurl_curl_el}
of the surface curl together with the curl constraint in~\eqref{eq_VKj} and infer that
\begin{align*}
\scurl_{F_n} (\trr_{F_n}^n)
&=
\scurl_{F_n} (\ppi^\ttau_{F_n}(\xxi_p^1) - \rr_{F_n})
\\
&=
\curl \xxi_p^1 \cdot \nn_{F_n} - \scurl_{F_n} (\rr_{F_n})
\\
&=
\rr_{K_1} \cdot \nn_{F_n} - \scurl_{F_n} (\rr_{F_n}).
\end{align*}
This gives
\begin{align*}
\rr_{K_n} \cdot \nn_{F_n} - \scurl_F (\trr_{F_n}^n)
&=
(\rr_{K_n} - \rr_{K_1}) \cdot \nn_{F_n} + \scurl_F (\rr_{F_n})
\\
&=
-\jump{\rr}_{F_n} \cdot \nn_{F_n} + \scurl_F (\rr_{F_n})
=
0,
\end{align*}
where the last equality follows from~\eqref{eq_comp_b}.
This proves the expected identity on the curl.

Let us now prove~\eqref{eq:cond_comp_n2}.
For all $j\in\{1:n-1\}$, since $\xxi_p^j \in \NN_p(K_j)$, its tangential traces satisfy the
edge compatibility condition
\begin{equation} \label{eq:edge_compatibility_xxi}
\left . \left (\ppi^\ttau_{F_{j-1}} (\xxi_p^j)\right ) \right |_\edge \cdot \ttau_\edge
=
\left .\left (\ppi^\ttau_{F_{j}} (\xxi_p^j)\right ) \right |_\edge \cdot \ttau_\edge.
\end{equation}
Moreover, for all $j\in\{1:n-2\}$, we have $F_j\in\FF_{j+1}$, so that by~\eqref{eq_VKj}
and the definition~\eqref{eq_F_j} of $\trr^{j+1}_{F_{j}}$, we have
\begin{equation*}
\ppi^\ttau_{F_{j}} (\xxi_p^{j+1}) = \trr^{j+1}_{F_{j}}
=
\ppi^\ttau_{F_{j}} (\xxi_p^j) +  \rr_{F_{j}},
\end{equation*}
and, therefore, using~\eqref{eq:edge_compatibility_xxi} yields
\begin{equation*}
\left . \left (\ppi^\ttau_{F_{j-1}} (\xxi_p^j)\right ) \right |_\edge \cdot \ttau_\edge
=
\left . \left (\ppi^\ttau_{F_{j}} (\xxi_p^{j+1})\right ) \right |_\edge \cdot \ttau_\edge -
\rr_{F_{j}}|_\edge \cdot \ttau_\edge.
\end{equation*}
Summing this identity for all $j\in\{1:n-2\}$ leads to
\[
\left . \left (\ppi^\ttau_{F_{0}} (\xxi_p^1)\right ) \right |_\edge \cdot \ttau_\edge
=
\left . \left (\ppi^\ttau_{F_{n-2}} (\xxi_p^{n-1})\right ) \right |_\edge \cdot \ttau_\edge
- \sum_{j\in\{1:n-2\}} \rr_{F_{j}}|_\edge \cdot \ttau_\edge.
\]

In addition, using again the edge compatibility condition~\eqref{eq:edge_compatibility_xxi}
for $j=n-1$ and the definition~\eqref{eq_F_nn} of $\trr_{F_{n-1}}^n$ leads to
\[
\left . \left (\ppi^\ttau_{F_{n-2}} (\xxi_p^{n-1})\right ) \right |_\edge \cdot \ttau_\edge
= \trr_{F_{n-1}}^n|_\edge \cdot \ttau_\edge - \rr_{F_{n-1}}|_\edge \cdot \ttau_\edge.
\]
Summing the above two identities gives
\begin{equation}
\label{tmp_induction_edge2}
\left . \left (\ppi^\ttau_{F_{0}} (\xxi_p^1)\right ) \right |_\edge \cdot \ttau_\edge
=
\trr_{F_{n-1}}^n|_\edge \cdot \ttau_\edge -
\sum_{j\in\{1:n-1\}} \rr_{F_{j}}|_\edge \cdot \ttau_\edge.
\end{equation}
Since $F_0=F_n$ for a patch of interior type and $\trr_{F_n}^n = \ppi^\ttau_{F_n} (\xxi_p^1) - \rr_{F_n}$ owing to~\eqref{eq_F_int}, the identity~\eqref{tmp_induction_edge2} gives
\begin{align*}
\trr_{F_n}^n |_\edge\cdot \ttau_\edge
&=
\left . \left (\ppi^\ttau_{F_{0}} (\xxi_p^1)\right ) \right |_\edge \cdot \ttau_\edge
-
\left (
\rr_{F_n} |_\edge \cdot \ttau_\edge
\right )
\\
&= \trr_{F_{n-1}}^n|_\edge \cdot \ttau_\edge
-
\sum_{j\in\{1:n-1\}} \rr_{F_{j}}|_\edge \cdot \ttau_\edge
-
\left (
\rr_{F_n} |_\edge \cdot \ttau_\edge
\right )
\\
&=
\trr^n_{F_{n-1}}|_\edge \cdot \ttau_\edge
-
\sum_{j\in\{1:n\}} \rr_{F_{j}}|_\edge \cdot \ttau_\edge = \trr^n_{F_{n-1}}|_\edge \cdot \ttau_\edge,
\end{align*}
where we used the edge compatibility condition~\eqref{eq_comp_c} satisfied by $\rr_\FF$ in the last equality. This proves~\eqref{eq:cond_comp_n2} in the interior case.

{\bf (3b-N)} Let us finally consider a patch of Neumann boundary type. The
second identity in~\eqref{eq:cond_comp_n1} follows directly from~\eqref{eq_comp_b}
and~\eqref{eq_F_n}. Let us now prove~\eqref{eq:cond_comp_n2}.
The identity~\eqref{tmp_induction_edge2} still holds true. Using that
$(\ppi^\ttau_{F_{0}} (\xxi_p^1)) |_\edge \cdot \ttau_\edge = \rr_{F_0}|_\edge \cdot \ttau_\edge$,
this identity is rewritten as
\[
\trr_{F_{n-1}}^n|_\edge \cdot \ttau_\edge = \sum_{j\in\{0:n-1\}} \rr_{F_{j}}|_\edge \cdot \ttau_\edge = \rr_{F_n}|_\edge \cdot \ttau_\edge,
\]
where the last equality follows from the edge compatibility condition~\eqref{eq_comp_d} satisfied by
$\rr_\FF$. But since $\trr_{F_n}^n = \rr_{F_n}$ owing to~\eqref{eq_F_n}, this again
proves~\eqref{eq:cond_comp_n2}.

{\bf (4)} It remains to show that $\xxi_p \in \boldsymbol V_p(\TTe)$ as per
Definition~\ref{def_spaces}. By construction, we have
$\ppi^\ttau_{F_{j}}(\xxi_p|_{K_{j+1}}) - \ppi^\ttau_{F_{j}}(\xxi_p|_{K_j}) = \rr_{F_{j}}$
for all $j\in\{1:n-1\}$,
$\ppi^\ttau_{F_{n}}(\xxi_p|_{K_{0}}) - \ppi^\ttau_{F_{n}}(\xxi_p|_{K_n}) = \rr_{F_{n}}$
if the patch is of interior type, $\ppi^\ttau_{F_{0}}(\xxi_p|_{K_{1}}) = \rr_{F_{0}}$
if the patch is of Neumann or mixed boundary type, and
$\ppi^\ttau_{F_{n}}(\xxi_p|_{K_{n}}) = \rr_{F_{n}}$ if the patch is of Neumann type.
This proves that $\ppi^\ttau_F(\jump{\xxi_p}_F) = \rr_F$ for all $F \in \FF$, i.e.,
$\jump{\xxi_p}^\ttau_\FF = \rr_\FF$ in the sense of Definition~\ref{definition_jumps}.
\end{proof}

\subsubsection{The actual proof}

We are now ready to prove Proposition~\ref{prop_stability_patch}.

\begin{proof}[Proof of Proposition~\ref{prop_stability_patch}]
Owing to Lemma~\ref{lem_EU}, the fields
\begin{equation} \label{eq_min_K}
\xxi_p^{\star j} \eq \argmin{\vv_p \in \boldsymbol V_p(K_j)} \|\vv_p\|_{K_j},
\qquad j\in\{1:n\},
\end{equation}
are uniquely defined in ${\boldsymbol V}_p(K_j)$, and the field $\xxi_p^{\star}$ such that
$\xxi_p^{\star}|_{K_j} \eq \xxi_p^{\star j}$ for all $j\in\{1:n\}$ satisfies
$\xxi_p^\star \in \boldsymbol V_p(\TTe)$. Since the minimizing sets in~\eqref{eq_stab_shift}
are nonempty (they all contain $\xxi_p^{\star}$), both the discrete and the continuous minimizers
are uniquely defined owing to standard convexity arguments.  Let us set
\begin{equation*}
\vv^\star \eq \argmin{\vv \in \boldsymbol V(\TTe)} \|\vv\|_{\ome},
\qquad \vv^\star_j \eq \vv^\star|_{K_j}, \quad j\in\{1:n\}.
\end{equation*}
To prove Proposition~\ref{prop_stability_patch}, it is enough to show that
\begin{equation} \label{eq res}
\|\xxi_p^\star\|_{\ome} \lesssim \|\vv^\star\|_{\ome}.
\end{equation}
Owing to Proposition~\ref{prop_stability_tetrahedra} applied with $K\eq K_j$ and $\FF\eq\FF_j$
for all $j\in\{1:n\}$, we have
\begin{equation} \label{eq_ineq_K}
\|\xxi_p^\star\|_{K_j} \lesssim
\min_{\zzeta \in \boldsymbol V(K_j)} \|\zzeta\|_{K_j},
\end{equation}
where $\boldsymbol V(K_{j})$ is defined in~\eqref{eq_VKj}.
Therefore, recalling that $|\TTe|\lesssim 1$,
\eqref{eq res} will be proved if for all $j\in\{1:n\}$, we can construct a field
$\zzeta_j \in \boldsymbol V(K_j)$ such that
$\|\zzeta_j\|_{K_j} \revision{\lesssim} \|\vv^\star\|_{\ome}$.
To do so, we proceed once again by induction.

{\bf (1)} First element ($j=1$).
Since $\vv^\star_1 \in \boldsymbol V(K_1)$, the claim is established with
$\zzeta_1\eq \vv^\star_1$ which trivially satisfies
$\|\zzeta_1\|_{K_1} = \|\vv^\star\|_{K_1} \leq \|\vv^\star\|_{\ome}$.

{\bf (2)} Middle elements ($j\in\{2:n-1\}$).
We proceed by induction. Given $\zzeta_{j-1}\in {\boldsymbol V}(K_{j-1})$ such that
$\|\zzeta_{j-1}\|_{K_{j-1}} \lesssim \|\vv^\star\|_{\ome}$,
let us construct a suitable $\zzeta_{j}\in {\boldsymbol V}(K_{j})$
such that $\|\zzeta_j\|_{K_j} \lesssim \|\vv^\star\|_{\ome}$.
We consider the affine geometric mapping $\TTT_{j-1,j}: K_{j-1} \to K_{j}$
that leaves the three vertices $\adown$, $\aaa_{j-1}$, and $\aup$
(and consequently the face $F_{j-1}$) invariant, whereas $\TTT_{j-1,j}(\aaa_{j-2}) = \aaa_{j}$.
We denote by $\ppsi^{\mathrm{c}}_{j-1,j} \eq \ppsi^{\mathrm{c}}_{\TTT_{j-1,j}}$
the associated Piola mapping, see Section~\ref{sec_Piola}.
Let us define the function $\zzeta_{j} \in \HH(\ccurl,K_{j})$ by
\begin{equation} \label{eq_zeta_j}
\zzeta_{j} \eq \vv^\star_{j} - \epsilon_{j-1,j} \ppsi^{\mathrm{c}}_{j-1,j}(\xxi_p^{\star j-1} - \vv^\star_{j-1}),
\end{equation}
where $\epsilon_{j-1,j} \eq \sign \left (\det \JJJ_{\TTT_{j-1,j}}\right)$
\revision{(notice that here $\epsilon_{j-1,j}=-1$)}.
Using the triangle inequality, the $L^2$-stability of the Piola mapping
(see~\eqref{eq_stab_piola_L2}), inequality~\eqref{eq_ineq_K}, and the induction hypothesis,
we have
\begin{equation} \label{eq_bound} \begin{split}
\|\zzeta_{j}\|_{K_j}
&\leq
\|\vv^\star\|_{K_j} + \|\ppsi^{\mathrm{c}}_{j-1,j}(\xxi_p^{\star j-1} - \vv^\star_{j-1})\|_{K_j}
\\
&\lesssim
\|\vv^\star\|_{K_j} + \|\xxi_p^\star - \vv^\star\|_{K_{j-1}}
\\
&\leq
\|\vv^\star\|_{K_j} + \|\xxi_p^\star\|_{K_{j-1}} + \|\vv^\star\|_{K_{j-1}}\\
&\lesssim
\|\vv^\star\|_{K_j} + \|\zzeta_{j-1}\|_{K_{j-1}} + \|\vv^\star\|_{K_{j-1}}
\lesssim
\|\vv^\star\|_{\ome}.
\end{split}\end{equation}
Thus it remains to establish that $\zzeta_{j} \in \boldsymbol V(K_{j})$
in the sense of Definition~\ref{def_spaces_elm}, i.e., we need to show that
$\curl \zzeta_{j}=\rr_{K_{j}}$ and $\zzeta_{j}|^\ttau_{\FF_{j}} = \trr^{j}_{\FF_{j}}$.
Recalling the curl constraints in~\eqref{eq_VTe} and~\eqref{eq_VKj} which yield
$\curl \xxi_p^\star = \curl \vv^\star = \rr_{\TT}$ and using~\eqref{eq_piola_commute}, we have
\begin{equation}
\label{tmp_curl_vv_middle}
\begin{split}
\curl \zzeta_{j}
&=
\curl \vv^\star_{j} - \epsilon_{j-1,j}\curl \ppsi^{\mathrm{c}}_{j-1,j}(\xxi_p^{\star j-1} - \vv^\star_{j-1})
\\
&=
\rr_{K_{j}} - \epsilon_{j-1,j} \ppsi^{\mathrm{d}}_{j-1,j}\left (\curl (\xxi_p^{\star j-1} - \vv^\star_{j-1})\right )
=
\rr_{K_{j}},
\end{split} \end{equation}
which proves the expected condition on the curl of $\zzeta_j$.

It remains to verify the weak tangential trace condition
$\zzeta_{j}|^\ttau_{\FF_{j}} = \trr^{j}_{\FF_{j}}$
as per Definition~\ref{definition_partial_trace}.
To this purpose, let $\pphi \in \HH^1_{\ttau,\FF_{j}^{\mathrm{c}}}(K_{j})$
and define $\tpphi$ by
\begin{equation}\label{eq_thpi}
\tpphi|_{K_{j}} \eq \pphi \quad \tpphi|_{K_{j-1}} \eq (\ppsi_{j-1,j}^{\mathrm{c}})^{-1}(\pphi),
\quad
\tpphi|_{K_l} = \boldsymbol 0 \quad \forall l \in\{1:n\}\setminus\{ j-1,j\}.
\end{equation}
These definitions imply that
$\tpphi \in \HH(\ccurl,\ome) \cap \HH^1_{\ttau,\FF^{\mathrm{c}}}(\TTe)$
(recall~\eqref{eq_Htpatch}) with
\begin{equation*}
\left . \left (\tpphi|_{K_{j-1}} \right ) \right |_{F_{j-1}} \times \nn_{F_{j-1}}
=
\left . \left (\tpphi|_{K_{j}} \right ) \right |_{F_{j-1}} \times \nn_{F_{j-1}}
=
\pphi|_{F_{j-1}} \times \nn_{F_{j-1}},
\end{equation*}
as well as
\begin{equation*}
\tpphi|_F \times \nn_F = \boldsymbol 0 \quad \forall F \in \FFe \setminus \{F_{j-1}\}.
\end{equation*}
(Note that $\tpphi|_F\times\nn_F$ is uniquely defined by assumption.)
Recalling definition~\eqref{eq_zeta_j} of $\zzeta_j$ and that
$\curl \zzeta_j = \rr_{K_{j}} = \curl\vv^\star_{j}$, see~\eqref{tmp_curl_vv_middle}, we have
\begin{align*}
&
(\curl \zzeta_j,\pphi)_{K_{j}}
-
(\zzeta_j,\curl \pphi)_{K_{j}}
\\
&=
(\curl \vv^\star,\pphi)_{K_{j}}
-
(\vv^\star,\curl \pphi)_{K_{j}}
+
\epsilon_{j-1,j}
(\ppsi^{\mathrm{c}}_{j-1,j} (\xxi_p^{\star j-1} - \vv^\star_{j-1}),\curl \pphi)_{K_{j}}
\\
&=
(\curl \vv^\star,\tpphi)_{K_{j}}
-
(\vv^\star,\curl \tpphi)_{K_{j}}
+ (\xxi_p^\star - \vv^\star,\curl \tpphi )_{K_{j-1}},
\end{align*}
where we used the definition of $\tpphi$, properties~\eqref{eq_piola_adjoint},
\eqref{eq_piola_commute} of the Piola mapping, and the definition of $\epsilon_{j-1,j}$
to infer that
\begin{align*}
\epsilon_{j-1,j}
(\ppsi^{\mathrm{c}}_{j-1,j} (\xxi_p^{\star j-1} - \vv^\star_{j-1}),\curl \pphi)_{K_{j}}
&=
\epsilon_{j-1,j}^2
(\xxi_p^\star - \vv^\star,\curl \left ((\ppsi^{\mathrm{c}}_{j-1,j})^{-1}\pphi|_{K_j} \right ))_{K_{j-1}}.
\\
&=
(\xxi_p^\star - \vv^\star,\curl \tpphi )_{K_{j-1}}.
\end{align*}
Since $\curl \xxi_p^\star = \rr_{\TT} = \curl \vv^\star$ and
$\tpphi = \boldsymbol 0$ outside $K_{j-1} \cup K_{j}$, this gives
\begin{align}
\label{tmp_middle_trace0}
&
(\curl \zzeta_j,\pphi)_{K_{j}}
-
(\zzeta_j,\curl \pphi)_{K_{j}}
\\
\nonumber
&=
(\curl \vv^\star,\tpphi)_{K_{j}}
-
(\vv^\star,\curl \tpphi)_{K_{j}}
+
(\xxi_p^\star - \vv^\star,\curl \tpphi )_{K_{j-1}}
+
(\curl (\vv^\star-\xxi_p^\star),\tpphi)_{K_{j-1}}
\\
\nonumber
&= \sum_{K \in \TTe}
\left \{
(\curl \vv^\star,\tpphi)_K - (\vv^\star,\curl \tpphi)_K
\right \}
-
\left (
(\curl \xxi_p^\star,\tpphi)_{K_{j-1}} - (\xxi_p^\star,\curl \tpphi)_{K_{j-1}}
\right ).
\end{align}

Since $\vv^\star \in \boldsymbol V(\TTe)$, $\tpphi \in \HH^1_{\ttau,\FF^{\mathrm{c}}}(\TTe)$,
and $\jump{\vv^\star}^\ttau_\FF=\rr_\FF$, we have from Definitions~\ref{definition_jumps}
and~\ref{def_spaces}
\begin{align}
\label{tmp_middle_trace1}
\sum_{K \in \TTe}
\left \{
(\curl \vv^\star,\tpphi)_K - (\vv^\star,\curl \tpphi)_K
\right \}
&=
\sum_{F \in \FF} (\rr_F,\tpphi \times \nn_F)_{F}
\\
\nonumber
&=
(\rr_{F_{j-1}},\pphi \times\nn_{F_{j-1}})_{F_{j-1}},
\end{align}
where in the last equality, we employed the definition~\eqref{eq_thpi} of $\tpphi$.
On the other hand, since $\xxi^\star_p|_{K_{j-1}},\tpphi|_{K_{j-1}} \in \HH^1(K_{j-1})$,
we can employ the pointwise definition of the trace and infer that
\begin{align}
\label{tmp_middle_trace2}
(\curl \xxi_p^\star,\tpphi)_{K_{j-1}} - (\xxi_p^{\star},\curl \tpphi)_{K_{j-1}}
&=
(\ppi^\ttau_{F_{j-1}} (\xxi_p^{\star j-1}),\tpphi|_{K_{j-1}} \times \nn_{K_{j-1}})_{F_{j-1}}
\\
\nonumber
&=
-(\ppi^\ttau_{F_{j-1}} (\xxi_p^{\star j-1}),\pphi \times \nn_{F_{j-1}})_{F_{j-1}},
\end{align}
where we used that $\nn_{K_{j-1}} = -\nn_{F_{j-1}}$.
Then, plugging~\eqref{tmp_middle_trace1} and~\eqref{tmp_middle_trace2}
into~\eqref{tmp_middle_trace0} and employing~\eqref{eq_F_j} and $\nn_{K_{j}} = \nn_{F_{j-1}}$,
we obtain
\begin{align*}
(\curl \zzeta_j,\pphi)_{K_{j}}
-
(\zzeta_j,\curl \pphi)_{K_{j}}
&=
(\rr_{F_{j-1}},\pphi \times\nn_{F_{j-1}})_{F_{j-1}}
+(\ppi^\ttau_{F_{j-1}} (\xxi_p^{\star j-1}),\pphi \times \nn_{F_{j-1}})_{F_{j-1}}
\\
&=
(\rr_{F_{j-1}}+\ppi^\ttau_{F_{j-1}} (\xxi_p^{\star j-1}),\pphi \times \nn_{F_{j-1}})_{F_{j-1}}
\\
&=
(\trr^{j}_{F_{j-1}},\pphi \times \nn_{K_{j}})_{F_{j-1}}.
\end{align*}
Since $\FF_j\eq \{F_{j-1}\}$, this shows that $\zzeta_j$ satisfies the weak tangential
trace condition in $\boldsymbol V(K_{j})$ by virtue of Definition~\ref{definition_partial_trace}.

{\bf (3)} Last element ($j=n$). We need to distinguish the type of patch.

{\bf (3a)} Patch of Dirichlet or mixed boundary type. In this case, we can employ the same argument as
for the middle elements since $\FF_n=\{F_{n-1}\}$ is composed of only one face.

{\bf (3b)} Patch of interior type. Owing to the induction hypothesis,
we have $\|\zzeta_j\|_{K_j} \lesssim \|\vv^\star\|_{\ome}$ for all $j\in\{1:n-1\}$.
Let us first assume that there is an even number of tetrahedra in the patch $\TTe$, as in
Figure~\ref{fig_patch}, left. The case where this number is odd will be discussed below.
We build a geometric mapping $\TTT_{j,n}:K_j\to K_n$ for all $j\in\{1:n-1\}$ as follows:
$\TTT_{j,n}$ leaves the edge $\edge$ pointwise invariant, $\TTT_{j,n}(\aaa_{j-1})\eq \aaa_{n}$,
$\TTT_{j,n}(\aaa_j)\eq \aaa_{n-1}$ if $(n-j)$ is odd, and $\TTT_{j,n}(\aaa_{j})\eq \aaa_{n}$,
$\TTT_{j,n}(\aaa_{j-1})\eq \aaa_{n-1}$ if $(n-j)$ is even. Since $n$ is by assumption even,
one readily sees that $\TTT_{j,n}(F_j)=\TTT_{j+1,n}(F_j)$ with $F_j=K_j\cap K_{j+1}$ for all
$j\in\{1;n-2\}$.

We define $\zzeta_n \in \HH(\ccurl,K_n)$ by setting
\begin{equation}
\label{eq_v_def}
\zzeta_n \eq \vv^\star_n -
\sum_{j\in\{1:n-1\}} \epsilon_{j,n} \ppsi^{\mathrm{c}}_{j,n}(\xxi_p^{\star j}-\vv^\star_j),
\end{equation}
where $\epsilon_{j,n} \eq \sign(\det \JJJ_{\TTT_{j,n}})$ and $\ppsi^{\mathrm{c}}_{j,n}$ is the
Piola mapping associated with $\TTT_{j,n}$.
Reasoning as above in~\eqref{eq_bound} shows that
\begin{equation*}
\|\zzeta_n\|_{K_n} \lesssim \|\vv^\star\|_{\ome}.
\end{equation*}

It now remains to establish that $\zzeta_{n} \in \boldsymbol V(K_{n})$
as per Definition~\ref{def_spaces_elm},
i.e. $\curl \zzeta_{n}=\rr_{K_{n}}$ and $\zzeta_{n}|^\ttau_{\FF_{n}} = \trr^{n}_{\FF_{n}}$
with $\FF_n\eq\{F_{n-1},F_n\}$. Since $\curl \xxi_p^\star = \rr_{\TT} = \curl \vv^\star$,
using~\eqref{eq_piola_commute} leads to $\curl \zzeta_n = \curl \vv^\star_n = \rr_{K_n}$
as above in~\eqref{tmp_curl_vv_middle}, which proves the expected condition on the curl of
$\zzeta_n$. It remains to verify the weak tangential trace condition as per
Definition~\ref{definition_partial_trace}. To this purpose, let
$\pphi \in \HH_{\ttau,\FF_n^{\mathrm{c}}}^1(K_n)$ and define $\tpphi$ by
\begin{equation}
\label{eq_ttet}
\tpphi|_{K_n} \eq \pphi, \qquad
\tpphi|_{K_j} \eq \left (\ppsi_{j,n}^{\mathrm{c}}\right )^{-1}(\pphi)
\quad \forall j\in\{1:n-1\}.
\end{equation}
As $\pphi \in \HH_{\ttau,\FF_n^{\mathrm{c}}}^1(K_n)$, its trace is defined in
a strong sense, and the preservation of tangential traces by Piola mappings shows
that $\tpphi \in \HH^1_{\ttau,\FF^{\mathrm{c}}}(\TTe)$ in the sense of~\eqref{eq_Htpatch}.
Then, using $\curl \zzeta_n = \curl \vv^\star_n$ and~\eqref{eq_v_def}, we have
\begin{align*}
&
(\curl \zzeta_n,\pphi)_{K_n}
-
(\zzeta_n,\curl \pphi)_{K_n}
\\
&=
(\curl \vv^\star,\tpphi)_{K_n}
-
(\vv^\star,\curl \tpphi)_{K_n}
+
\sum_{j\in\{1:n-1\}} \epsilon_{j,n}
(\ppsi^{\mathrm{c}}_{j,n}(\xxi^{\star j}_p - \vv^\star_{j}),\curl \pphi)_{K_n},
\end{align*}
where we used the definition of $\tpphi$ for the first two terms on the right-hand side.
Moreover, using~\eqref{eq_piola_adjoint} and~\eqref{eq_piola_commute} for all $j\in\{1:n-1\}$,
we have
\begin{align*}
\epsilon_{j,n}
(\ppsi^{\mathrm{c}}_{j,n}(\xxi^{\star j}_p - \vv^\star_{j}),\curl \pphi)_{K_n}
&=
\epsilon_{j,n}^2
(\xxi^\star_p - \vv^\star,\curl ((\ppsi^{\mathrm{c}}_{j,n})^{-1}(\pphi|_{K_n})))_{K_j}
\\
&=
(\xxi^\star_p - \vv^\star,\curl \tpphi)_{K_j}
\\
&=
(\xxi^\star_p - \vv^\star,\curl \tpphi)_{K_j} - (\curl (\xxi^\star_p - \vv^\star),\tpphi)_{K_j},
\end{align*}
since $\curl \xxi^\star_p = \rr_\TT = \curl \vv^\star$.
It follows that
\begin{align*}
& (\curl \zzeta_n,\pphi)_{K_n}
-
(\zzeta_n,\curl \pphi)_{K_n}\\
&=
\sum_{j\in\{1:n\}}
\left \{
(\curl \vv^\star,\tpphi)_{K_j}
-
(\vv^\star,\curl \tpphi)_{K_j}
\right \}
-
\sum_{j\in\{1:n-1\}}
\left \{
(\curl \xxi_p^\star,\tpphi)_{K_j}
-
(\xxi_p^\star,\curl \tpphi)_{K_j}
\right \}
\\
&=
(\curl \xxi_p^\star,\tpphi)_{K_n}
-
(\xxi_p^\star,\curl \tpphi)_{K_n}
+
\sum_{j\in\{1:n\}}
\left \{
(\curl \vv^\star,\tpphi)_{K_j}
-
(\vv^\star,\curl \tpphi)_{K_j}
\right \}
\\
& \qquad -
\sum_{j\in\{1:n\}}
\left \{
(\curl \xxi_p^\star,\tpphi)_{K_j}
-
(\xxi_p^\star,\curl \tpphi)_{K_j}
\right \}
\\
&=
(\curl \xxi_p^\star,\tpphi)_{K_n}
-
(\xxi_p^\star,\curl \tpphi)_{K_n}
=
(\curl \xxi_p^\star,\pphi)_{K_n}
-
(\xxi_p^\star,\curl \pphi)_{K_n}\\
&=  \sum_{F\in\FF_n} (\trr^{n}_{F},\pphi \times \nn_{K_n})_{F},
\end{align*}
where we employed the fact that, since both $\xxi_p^\star,\vv^\star \in \VV(\TTe)$,
Definition~\ref{definition_jumps} gives
\begin{align*}
\sum_{j\in\{1:n\}}
\left \{
(\curl \vv^\star,\tpphi)_{K_j}
-
(\vv^\star,\curl \tpphi)_{K_j}
\right \}
&=
\sum_{F \in \FF} (\rr_F,\tpphi \times \nn_F)_{F}
\\
&=
\sum_{j\in\{1:n\}}
\left \{
(\curl \xxi_p^\star,\tpphi)_{K_j}
-
(\xxi_p^\star,\curl \tpphi)_{K_j}
\right \}.
\end{align*}
Thus $\zzeta_n|^\ttau_{\FF_n} = \trr^n_{\FF_n}$ in the sense of
Definition~\ref{definition_partial_trace}. This establishes the weak tangential
trace condition on $\zzeta_n$ when $n$ is even.

If $n$ is odd, one can proceed as
in~\cite[Section~6.3]{Ern_Voh_p_rob_3D_20}. For the purpose of the proof only,
one tetrahedron different from $K_n$ is subdivided into two subtetrahedra as
in~\cite[Lemma~B.2]{Ern_Voh_p_rob_3D_20}. Then, the above construction of $\zzeta_n$
can be applied on the newly created patch which has an even number of elements, and one
verifies as above that $\zzeta_n\in {\boldsymbol V}(K_n)$.

{\bf (3c)} Patch of Neumann boundary type. In this case, a similar argument as for
a patch of interior type applies, and we omit the proof for the sake of brevity.

\end{proof}

\begin{remark}[Quasi-optimality of $\xxi_p^\star$] \label{rem_sweep_brok}
Let $\xxi_p^\star\in{\boldsymbol V}_p(\TTe)$ be defined in the above proof
(see in particular~\eqref{eq_min_K}).
Since $\|\vv^\star\|_{\ome} \le \min_{\vv_p\in {\boldsymbol V}_p(\TTe)}\|\vv_p\|_{\ome}$,
inequality~\eqref{eq res} implies that
$\|\xxi_p^\star\|_{\ome} \lesssim
\min_{\vv_p\in {\boldsymbol V}_p(\TTe)}\|\vv_p\|_{\ome}$
(note that the converse inequality is trivial with constant one).
This elementwise minimizer is the one used in \revision{Theorem~\ref{thm_sweep} and in}
the simplified a posteriori error estimator~\eqref{eq_definition_estimator_sweep_2}.
\end{remark}

\bibliographystyle{amsplain}
\bibliography{biblio}

\appendix

\section{Poincar\'e-like inequality using the curl of divergence-free fields}
\label{appendix_weber}

\begin{theorem}[Constant in the Poincar\'e-like inequality~\eqref{eq_local_poincare_vectorial}]
For every edge $\edge \in \EE_h$, the constant
\begin{equation}
\CPVe \eq \frac{1}{h_\ome} \sup_{\substack{
\vv \in \HH_{\GeD}(\ccurl,\ome) \cap \HH_{\GeN}(\ddiv,\ome)
\\
\div\vv=0\\
\|\curl \vv\|_{\ome} = 1
}}
\|\vv\|_{\ome}
\end{equation}
only depends on the shape-regularity parameter $\kappa_\edge$ of the edge patch $\TT^\edge$.
\end{theorem}

\begin{proof}
We proceed in two steps.

(1) Let us first establish a result regarding the transformation of this type
of constant by a bilipschitz mapping.
Consider a Lipschitz and simply connected domain $U$ with its boundary
$\partial U$ partitioned into two disjoint relatively open subdomains
$\Gamma$ and $\Gamma_{\mathrm{c}}$. Let $\TTT: U \to \tU$ be a bilipschitz mapping
with Jacobian matrix $\JJJ$, and let $\tG \eq \TTT(\Gamma)$ and
$\tG_{\mathrm{c}} \eq \TTT(\Gamma_{\mathrm{c}})$. Let us set
\begin{equation*}
C_{\rm PFW}(U,\Gamma)
\eq
\sup_{\substack{
\uu \in \HH_{\Gamma}(\ccurl,U) \cap \HH_{\Gamma_{\mathrm{c}}}(\ddiv,U)
\\
\div \uu = 0\\
\|\curl \uu\|_U = 1
}}
\|\uu\|_U,
\qquad
C_{\rm PFW}(\tU,\tG)
\eq
\sup_{\substack{
\tuu \in \HH_{\tG}(\ccurl,\tU) \cap \HH_{\tG_{\mathrm{c}}}(\ddiv,\tU)
\\
\div \tuu = 0\\ \|\curl \tuu\|_{\tU} = 1
}}
\|\tuu\|_{\tU}.
\end{equation*}
Remark that both constants are well-defined real numbers owing to
\cite[Proposition 7.4]{Fer_Gil_Maxw_BC_97}.
Then, we have
\begin{equation} \label{eq:transfo_C_PFW}
C_{\rm PFW}(U,\Gamma) \leq
\|\phi\|_{L^\infty(U)}^2
C_{\rm PFW}(\tU,\tG),
\end{equation}
with $\phi(\xx)\eq |\det \JJJ(\xx)|^{-\frac12} \|\JJJ(\xx)\|$ for all $\xx\in U$.
To show~\eqref{eq:transfo_C_PFW}, let
$\uu \in \HH_{\Gamma}(\ccurl,U) \cap \HH_{\Gamma_{\mathrm{c}}}(\ddiv,U)$ be such that $\div\uu=0$.
Let us set $\tuu \eq (\ppsi^{\mathrm{c}}_U)^{-1}(\uu)$ where
$\ppsi^{\mathrm{c}}_U:\HH_{\tG}(\ccurl,\tU)\to \HH_{\Gamma}(\ccurl,U)$ is the
covariant Piola mapping. Since $\tuu$ is not necessarily divergence-free and does not have necessarily a zero normal trace on $\tG_{\mathrm{c}}$,
we define (up to a constant) the function $\tq \in H^1_{\tG}(\tU)$ such that
\begin{equation*}
(\grad \tq,\grad {\widetilde w})_{\tU} = (\tuu,\grad {\widetilde w})_{\tU}
\qquad
\forall {\widetilde w} \in H^1_{\tG}(\tU).
\end{equation*}
Then, the field
$\tvv \eq \tuu-\grad \tq$ is in $\HH_{\tG}(\ccurl,\tU) \cap \HH_{\tG_{\mathrm{c}}}(\ddiv,\tU)$
and is divergence-free. Therefore, we have
\begin{equation*}
\|\tvv\|_{\tU}
\leq
C_{\rm PFW}(\tU,\tG) \|\curl \tvv\|_{\tU}
=
C_{\rm PFW}(\tU,\tG) \|\curl \tuu\|_{\tU}.
\end{equation*}
Let us set
\begin{equation*}
\vv \eq \ppsi^{\mathrm{c}}_U(\tvv) = \uu - \ppsi^{\mathrm{c}}_U(\grad \tq)
= \uu - \grad q,
\end{equation*}
with $q\eq\psi^{\mathrm{g}}_U(\tq)\eq \tq\circ \TTT$. Since
$\uu \in \HH_{\Gamma_{\mathrm{c}}}(\ddiv,U)$ with $\div\uu=0$ and $q \in H^1_{\Gamma}(U)$,
there holds $(\uu,\grad q)_{U} = 0$, which implies that
$\|\uu\|_U \leq \|\uu-\grad q\|_U=\|\vv\|_U$. Moreover, proceeding as in the proof of
\cite[Lemma 11.7]{Ern_Guermond_FEs_I_21} shows that
\begin{equation*}
\|\vv\|_U
\leq \|\phi\|_{L^\infty(U)} \|\tvv\|_{\tU}.
\end{equation*}
Combining the above bounds shows that
\begin{equation*}
\|\uu\|_U \le \|\phi\|_{L^\infty(U)} C_{\rm PFW}(\tU,\tG) \|\curl \tuu\|_{\tU}.
\end{equation*}
Finally, we have $\curl \tuu=(\ppsi^{\mathrm{d}}_U)^{-1}(\curl \uu)$ where $\ppsi^{\mathrm{d}}_U$
is the contravariant Piola mapping, and proceeding as in the proof of
\cite[Lemma 11.7]{Ern_Guermond_FEs_I_21} shows that
\begin{equation*}
\|\curl\tuu\|_{\tU}
\leq \|\phi\|_{L^\infty(U)} \|\curl\uu\|_{U}.
\end{equation*}
Altogether, this yields
\begin{equation*}
\|\uu\|_U \le \|\phi\|_{L^\infty(U)}^2 C_{\rm PFW}(\tU,\tG) \|\curl\uu\|_{U},
\end{equation*}
and \eqref{eq:transfo_C_PFW} follows from the definition of $C_{\rm PFW}(U,\Gamma)$.

(2) The maximum
value of the shape-regularity parameter $\kappa_\edge$ for all $\edge\in\EE_h$
implicitly constrains
the minimum angle possible between two faces of each tetrahedron in the edge patch
$\TT^\edge$. Therefore, there exists an integer $n(\kappa_\edge)$ such that
$|\TT^\edge| \leq n(\kappa_\edge)$. Moreover, there is a finite possibility for
choosing the Dirichlet faces composing $\GeD$.
As a result, there exists a finite set of pairs $\{(\widehat \TT,\widehat \Gamma)\}$
(where $\widehat\TT$ is a reference edge patch and $\widehat\Gamma$ is a (possibly
empty) collection of its boundary faces)
such that, for every $\edge\in \EE_h$, there is a pair
$(\widehat \TT,\widehat \Gamma)$ and a
bilipschitz, piecewise affine mapping satisfying
$\TTT_\edge: \widehat \omega \to \ome$ and $\TTT_\edge(\widehat \Gamma) = \GeD$,
where $\widehat \omega$ is the simply connected domain associated with
$\widehat \TT$. Step (1) above implies that
\begin{equation*}
\CPVe \leq
\max_{\widehat \xx \in \widehat \omega} \left (
\frac{\|\JJJ_\edge(\widehat \xx)\|^2}{|\det \JJJ_\edge(\widehat \xx)|}
\right )
C_{\rm PFW}(\widehat \omega,\widehat \Gamma),
\end{equation*}
where $\JJJ_\edge$ is the Jacobian matrix of $\TTT_\edge$.
Standard properties of affine mappings show that
\begin{equation*}
\max_{\widehat \xx \in \widehat \omega} \left (
\frac{\|\JJJ_\edge(\widehat \xx)\|^2}{|\det \JJJ_\edge(\widehat \xx)|}
\right )
=
\max_{K \in \TT^\edge}
\frac{h_{\widehat K}^2}{\rho_K^{2}}\frac{|K|}{|\widehat K|},
\end{equation*}
where $\widehat K=\TTT_\edge^{-1}(K)$ for all $K\in\TT_\edge$.
Since $|K| \leq h_K^3$, we have
\begin{equation*}
\max_{\widehat \xx \in \widehat \omega} \left (
\frac{\|\JJJ(\widehat \xx)\|^2}{|\det \JJJ(\widehat \xx)|}
\right )
\leq
\left (
\max_{\widehat K \in \widehat \TT} \frac{h_{\widehat K}^2}{|\widehat K|}
\right )
\kappa_\edge^2 h_{\ome}.
\end{equation*}
This concludes the proof.
\end{proof}

\end{document}